\documentclass[11pt]{amsart}

\usepackage{amsmath,amsthm,amsfonts,amssymb,bm,mathrsfs,graphicx,cprotect}

\usepackage{hyperref}
\hypersetup{colorlinks=true,citecolor=blue,linkcolor=blue,urlcolor=blue,
pdfstartview=FitH, pdfauthor=Valentin Blomer and Gergely
Harcos and Peter Maga and Djordje Milicevic, pdftitle=The sup-norm problem for GL(2) over number fields}

\theoremstyle{plain}
\newtheorem{theorem}{Theorem}
\newtheorem{lemma}{Lemma}
\newtheorem{corollary}{Corollary}

\theoremstyle{remark}
\newtheorem{remark}{Remark}

\theoremstyle{definition}
\newtheorem*{acknowledgement}{Acknowledgement}
\newtheorem*{convention}{Convention}

\numberwithin{equation}{section}

\renewcommand{\geq}{\geqslant}
\renewcommand{\leq}{\leqslant}

\DeclareMathOperator{\PSL}{PSL}
\DeclareMathOperator{\GL}{GL}
\DeclareMathOperator{\SL}{SL}
\DeclareMathOperator{\SO}{SO}
\DeclareMathOperator{\oO}{O}

\DeclareMathOperator{\U}{U}
\DeclareMathOperator{\Z}{Z}
\DeclareMathOperator{\gP}{P}
\DeclareMathOperator{\M}{M}
\DeclareMathOperator{\sgn}{sgn}
\DeclareMathOperator{\vol}{vol}
\DeclareMathOperator{\re}{Re}
\DeclareMathOperator{\im}{Im}
\DeclareMathOperator{\sech}{sech}
\DeclareMathOperator{\height}{ht}
\DeclareMathOperator{\tr}{tr}
\DeclareMathOperator{\Gal}{Gal}

\newcommand{\eps}{\varepsilon}
\newcommand{\fin}{_{\text{fin}}}
\newcommand{\ov}[1]{\overline{#1}}
\newcommand{\abs}[1]{\left| #1 \right|}
\newcommand{\rest}{\!\!\restriction}

\newcommand{\RR}{\mathbb{R}}
\newcommand{\CC}{\mathbb{C}}
\newcommand{\QQ}{\mathbb{Q}}
\newcommand{\ZZ}{\mathbb{Z}}
\newcommand{\NN}{\mathbb{N}}
\newcommand{\HH}{\mathbb{H}}
\newcommand{\Aa}{\mathbb{A}}
\newcommand{\MM}{\mathbb{M}}
\newcommand{\FF}{\mathcal{F}}
\newcommand{\Hh}{\mathcal{H}}
\newcommand{\N}{\mathcal{N}}
\newcommand{\Hb}{\bm{\mathcal{H}}}
\newcommand{\ib}{\bm{i}}
\newcommand{\Hs}{\mathscr{H}}
\newcommand{\ma}{\mathfrak{a}}
\newcommand{\mb}{\mathfrak{b}}
\newcommand{\md}{\mathfrak{d}}
\newcommand{\mk}{\mathfrak{k}}
\newcommand{\ml}{\mathfrak{l}}
\newcommand{\mm}{\mathfrak{m}}
\newcommand{\mn}{\mathfrak{n}}
\newcommand{\mo}{\mathfrak{o}}
\newcommand{\mpr}{\mathfrak{p}}
\newcommand{\mq}{\mathfrak{q}}
\newcommand{\bj}{\mathbf{j}}
\newcommand{\bk}{\mathbf{k}}
\newcommand{\bl}{\mathbf{l}}
\newcommand{\bn}{\mathbf{n}}
\newcommand{\bp}{\mathbf{p}}
\newcommand{\bq}{\mathbf{q}}
\newcommand{\bz}{\mathbf{z}}

\newcommand{\SLZ}{\SL_2(\ZZ)}
\newcommand{\SLR}{\SL_2(\RR)}

\newcommand{\GLR}{\GL_2(\RR)}
\newcommand{\GLC}{\GL_2(\CC)}
\newcommand{\SLZH}{{\SLZ\backslash\Hh^2}}

\newcommand{\abcd}{\left(\begin{matrix}a&b\\c&d\end{matrix}\right)}
\newcommand{\sabcd}{\bigl(\begin{smallmatrix}a&b\\c&d\end{smallmatrix}\bigr)}
\newcommand{\yx}{\left(\begin{matrix}y&x\\&1\end{matrix}\right)}
\newcommand{\syx}{\bigl(\begin{smallmatrix}y&x\\&1\end{smallmatrix}\bigr)}
\newcommand{\yxt}{\left(\begin{matrix}\tilde y&\tilde x\\&1\end{matrix}\right)}
\newcommand{\syxt}{\bigl(\begin{smallmatrix}\tilde y&\tilde x\\&1\end{smallmatrix}\bigr)}
\newcommand{\ts}{\left(\begin{matrix}t&s\\&1\end{matrix}\right)}
\newcommand{\sts}{\bigl(\begin{smallmatrix}t&s\\&1\end{smallmatrix}\bigr)}
\newcommand{\tst}{\left(\begin{matrix}\tilde t&\tilde s\\&1\end{matrix}\right)}
\newcommand{\stst}{\bigl(\begin{smallmatrix}\tilde t&\tilde s\\&1\end{smallmatrix}\bigr)}

\begin{document}

\author{Valentin Blomer}
\author{Gergely Harcos}
\author{P\'eter Maga}
\author{Djordje Mili\'cevi\'c}

\address{Mathematisches Institut, Bunsenstr. 3-5, D-37073 G\"ottingen, Germany}\email{vblomer@math.uni-goettingen.de}
\address{Alfr\'ed R\'enyi Institute of Mathematics, Hungarian Academy of Sciences, POB 127, Budapest H-1364, Hungary}\email{gharcos@renyi.hu, magapeter@gmail.com}
\address{MTA R\'enyi Int\'ezet Lend\"ulet Automorphic Research Group}\email{gharcos@renyi.hu, magapeter@gmail.com}
\address{Central European University, Nador u. 9, Budapest H-1051, Hungary}\email{harcosg@ceu.edu}
\address{Bryn Mawr College, Department of Mathematics, 101 North Merion Avenue, Bryn Mawr, PA 19010, USA}\email{dmilicevic@brynmawr.edu}

\title{The sup-norm problem for $\GL(2)$ over number fields}

\thanks{First author supported in part by the Volkswagen Foundation and NSF grant DMS-1128155 while enjoying the hospitality of the Institute for Advanced Study. Second and third author supported by NKFIH (National Research, Development and Innovation Office) grants NK~104183, ERC\underline{\phantom{ }}HU\underline{\phantom{ }}15~118946, K~119528, and by the MTA R\'enyi Int\'ezet Lend\"ulet Automorphic Research Group. Second author also supported by NKFIH grant K~101855 and ERC grant AdG-321104. Third author also supported by the Premium Postdoctoral Fellowship of the Hungarian Academy of Sciences. Fourth author supported by NSF grant DMS-1503629 and ARC grant DP130100674, and he thanks the Max Planck Institute for Mathematics for their support and exceptional research infrastructure.}


\begin{abstract}
We solve the sup-norm problem for spherical Hecke--Maa{\ss} newforms of square-free level for the group $\GL(2)$ over a number field, with a power saving over the local geometric bound simultaneously in the eigenvalue and the level aspect. Our bounds feature a Weyl-type exponent in the level aspect, they reproduce or improve upon all known special cases, and over totally real fields they are as strong as the best known hybrid result over the rationals.
\end{abstract}


\setcounter{tocdepth}{2}

\maketitle

\section{Introduction}

\subsection{The sup-norm problem} The sup-norm problem has taken a prominent position in recent years at the interface of automorphic forms and analytic number theory. It is inspired by a classical question in analysis about comparing two norms on an infinite-dimensional Hilbert space: given an eigenfunction $\phi$ on a locally symmetric space $X$ with a large Laplace eigenvalue $\lambda$ and $\|\phi\|_2=1$, what can be said about its sup-norm $\|\phi\|_\infty$?

This question is closely connected to the multiplicity of eigenvalues \cite{Sa1}, and it is motivated by the correspondence principle of quantum mechanics, where the high energy limit $\lambda \rightarrow \infty$ provides a connection between classical and quantum mechanics. The sup-norm of an eigenform with large eigenvalue gives some information on the distribution of its mass on $X$, which sheds light on the question to what extent these eigenstates can localize (``scarring''). Despite a lot of work from different points of view, in the case of a classically chaotic Hamiltonian (for instance when $X$ is a compact hyperbolic manifold), the relation between the
classical mechanics and the quantum mechanics in the semi-classical limit is currently not well-understood. This goes by the name \emph{quantum chaos}. We refer the reader to the excellent surveys \cite{Sa3,Sa2} and the references therein for an introduction to this topic and further details.

Purely analytic techniques can be used to give a best-possible solution to the sup-norm problem on a general compact locally symmetric space $X$ of dimension $d$ and rank $r$: one has \cite{Sa1}
\begin{equation}\label{generic}
\|\phi\|_{\infty} \ll_X \lambda^{(d-r)/4},
\end{equation}
and this bound is sharp as it is attained, for instance, for the round sphere. (The symbol $\ll$ is introduced formally at the end of Section~\ref{section1}.) The bound is local in nature, in that its proof is insensitive to the global geometry of $X$, and in general it still allows for significant concentration of mass at individual points. In many cases, in particular for compact hyperbolic manifolds, a stronger bound is expected. The \emph{sup-norm problem} aims at decreasing the exponent in \eqref{generic} or in a refined version thereof.

The beauty of the sup-norm problem lies in particular in the fact that it is amenable to \emph{arithmetic} techniques when the manifold is equipped with additional arithmetic structure. Two classical examples in dimension $2$ are the round sphere $X = \mathcal{S}^2 = \SO_3(\RR)/\SO_2(\RR)$, realized as a quotient of the projective group of units in the Hamilton quaternions, and
the modular surface $X=\SLZH$, where $\Hh^2$ denotes the Poincar\'e upper half-plane of complex numbers on which $\SLR$ acts by hyperbolic isometries. In both cases, there is an arithmetically defined family of Hecke operators commuting with the Laplacian, so that it makes sense to consider joint eigenfunctions. A combination of analytic and arithmetic techniques led to a significant improvement of \eqref{generic} for joint Hecke--Laplace eigenfunctions on these two arithmetic surfaces~\cite{IS, VdK}: namely,
\[\|\phi\|_{\infty} \ll_{\eps} \lambda^{5/24 +\eps}\]
holds for every $\eps > 0$. For applications, for instance in connection with Faltings' delta function \cite{JK1, JK2, JK3}, it is also important to consider the dependence of the implied constant in \eqref{generic} with respect to $X$, in particular as $X$ varies through a sequence of covers. A typical situation is the case of the congruence covers
\begin{equation}\label{congruencecover}
\Gamma_0(N)\backslash\Hh^2 \rightarrow \SLZH,
\end{equation}
where $\Gamma_0(N)$ is the usual Hecke congruence subgroup consisting of $2\times 2$ integral unimodular matrices with lower left entry divisible by $N$.

Following the original breakthrough of Iwaniec and Sarnak~\cite{IS}, a lot of effort went into proving good upper bounds (and also lower bounds, but this is not the focus of the present paper) for joint eigenfunctions, in a great variety of situations and with various applications in mind: see for instance \cite{BHM, BHo, BM, BP, BMa, DS, HT1, HT2, HRR, Ki, Mar, Sa, St, Te, X}. The results fall roughly into two categories. On the one hand, one can try to establish bounds as strong as possible. Somewhat reminiscent of the \emph{subconvexity problem} in the theory of $L$-functions, this often leads eventually to a ``natural'' exponent that marks the limit of techniques of analytic number theory. For example, in the case of the congruence cover \eqref{congruencecover} for $N$ a square-free integer, Templier~\cite{Te} proved the important benchmark result
\begin{equation}\label{templier}
\|\phi\|_{\infty} \ll_\eps \lambda^{5/24 +\eps}N^{1/3+\eps},
\end{equation}
improving simultaneously\footnote{See also Remark~\ref{rem4} in Subsection~\ref{parabolicsection}.} in both aspects on the generic local bound $\lambda^{1/4}N^{1/2}$. On the other hand, one can also confine oneself to some small numerical improvement over the trivial bound, but use techniques that work on very general spaces $X$. Here the most general available result is due to Marshall~\cite{Mar} for semisimple split Lie groups over totally real fields and their totally imaginary quadratic extensions (CM-fields).

\subsection{General number fields} In this paper, we address both points of view, and for the first time we address the sup-norm problem for the group $\GL_2$ over a general number field $F$ of degree $n = r_1 + 2r_2$ over $\QQ$, with $r_1$ real embeddings and $r_2$ conjugate pairs of complex embeddings. From the perspective of automorphic forms, this is certainly the natural framework, and there is little reason to treat the ground field $\QQ$ separately. The underlying manifold is then a quotient of the product of $r_1$ copies of the upper half-plane $\Hh^2$ and $r_2$ copies of the upper half-space $\Hh^3$, so it has dimension $d = 2r_1 + 3r_2$ and rank $r=r_1+r_2$ (cf.\ \eqref{generic}). As is well-known, the passage from $\QQ$ to a general number field introduces two abelian groups, the finite class group and the (except for the imaginary quadratic case) infinite unit group. As has been observed in many contexts (e.g.\ in the context of cubic hypersurfaces \cite{BV} and the Ramanujan conjecture \cite{BB}), these groups cause considerable technical difficulties for arguments of analytic number theory; the general strategy is always to use an adelic treatment to deal with issues of the class group and to use carefully chosen units in order to work with algebraic integers whose size is comparable in all archimedean embeddings. Our paper provides a general adelic counting scheme for such situations, see Section~\ref{sec-counting}.

However, in our case the difficulties go much deeper than dealing with the class group and the unit group. As soon as $F$ has a complex place, the formalism of the amplified pre-trace formula leads to counting integral matrices $\gamma\in\M_2(\mo_F)$, which lie suitably close to a certain maximal compact subgroup of $\GL_2(F_\infty)$, and whose entries are described by conditions involving real and imaginary parts at each complex place separately. If $F$ is not a CM-field, there is no global complex conjugation (see e.g.\ \cite{Ok}), and hence the global counting techniques that work over number fields like $\QQ$ or $\QQ(i)$ break down in the general situation. In fact, the maximal compact subgroups of $\GL_2(F_\infty)$ cannot be defined over $F$ unless $F$ is a totally real field or a CM-field.

Another difficulty is signified by a fundamental difference between $\PSL_2(\RR)$ and $\PSL_2(\CC)$. On the one hand, every arithmetic Fuchsian subgroup of $\PSL_2(\RR)$ is commensurable with $\SO^{+}(L)$ for a suitable lattice $L$ in a quadratic space $V$ of signature $(2,1)$, upon identifying $\SO^+(V)$ with $\PSL_2(\RR)$. On the other hand, an arithmetic Kleinian subgroup of $\PSL_2(\CC)$ is commensurable with $\SO^{+}(M)$ for a suitable lattice $M$ in a quadratic space $W$ of signature $(3,1)$, upon identifying $\SO^+(W)$ with $\PSL_2(\CC)$, if and only if it contains a
non-elementary Fuchsian subgroup~\cite[Theorem~10.2.3]{MR}. These special Kleinian subgroups are already known to behave distinctively for the sup-norm problem~\cite{Mil}, and they can be described in terms of the invariant trace fields and quaternion algebras; in particular, their trace field is a quadratic extension of the maximal totally real subfield.

For a general number field $F$, these structural features make the sup-norm problem in many ways a very different problem. Therefore, we introduce a number of new devices into the argument to leverage the specific interplay between the maximal compact subgroups of $\GL_2(F_\infty)$ and the arithmetic of~$F$. In the hardest situation in our counting problem, $F$ is not totally real, and the field element $\xi:=\tr(\gamma)^2/\det(\gamma)$ is bounded in $F_\infty$ and very close to being totally real. In this case, we combine two observations that appear to be novel in this context. On the one hand, we exploit a certain \emph{rigidity of number fields} (see Section~\ref{rigidity}) to show that $\xi$ lies in a \emph{proper} subfield of $F$. However, the denominator of $\xi$ is arithmetically controlled by our specific \emph{amplifier} (see Section~\ref{amplifier}), so $\xi\in F$ must be an algebraic integer. This is already a very strong conclusion when coupled with the boundedness of $\xi$ in $F_{\infty}$; however, except for special number fields $F$, we do not know how to deal with the non-parabolic cases $\xi\neq 4$. On the other hand, by artificially extending the spectrum, we can improve the performance of the pre-trace formula on the geometric side so that $\gamma\in\M_2(\mo_F)$ is also localized modulo some auxiliary ideal $\mq$. Specifically, we can ensure that $\gamma$ is \emph{locally parabolic} modulo $\mq$. As a result, $\xi\in 4+\mq$, which forces $\xi=4$ when the norm of $\mq$ is large. In conclusion, in the hardest situation we can eliminate all but parabolic matrices, which are relatively simple to count. We refer the reader to Lemma~\ref{j=2} for a precise version of this argument, as well as to Lemma~\ref{improvement} for another application of the realness rigidity of number fields.

The precise setup of extending the spectrum and hence localizing $\gamma$ modulo $\mq$ is described in Sections~\ref{section1} and \ref{section2}, with a special view toward treating the units in $\mo_F$ efficiently. Indeed, there is a natural ambiguity of
$\det(\gamma)$ by units modulo squared units, while our congruence conditions force the units that appear here to be quadratic
residues modulo $\mq$; we can choose $\mq$ in such a way that these units are automatically squared units. Thus the success of our method rests on three pillars: passage to a suitably chosen congruence subgroup, a carefully designed amplifier equipped with arithmetic features as described in Subsection~\ref{napart}, and the rigidity results for number fields mentioned above. At the technical level, we rely heavily on Atkin--Lehner operators (see Section~\ref{atkinlehner}) and the geometry of numbers (see Section~\ref{section4}), which allow an efficient counting of the matrices $\gamma$ in Section~\ref{countingmatrices}.

In retrospect, the general idea of extending the spectrum to thin out the geometric side of the pre-trace formula is not unprecedented, the most spectacular example being Iwaniec's approach \cite{Iw1} to the Ramanujan conjecture for the metaplectic group (see also \cite{Sp} for another example). We believe that our variation of it, based on arithmetic properties of a certain congruence subgroup and the underlying number field, introduces a novel and flexible tool into the machinery of the sup-norm problem that may be useful in other situations.

As another useful feature, our argument also uses positivity more strongly than the previous treatments. Rather than carrying out an exact spectral average, we use positivity of our operators to establish a \emph{pre-trace inequality}. This streamlines the argument substantially, e.g.\ we do not even have to mention Eisenstein series and oldforms. A similar idea in the context of infinite volume subgroups was used by Gamburd~\cite{Ga}.

Finally, we mention that in Section~\ref{fourierbound} we develop a uniform Fourier bound for spherical Hecke--Maa{\ss} newforms for the
group $\GL_2$ over a number field, which might be of independent interest.

\subsection{Main results} Our main result is a solution of the sup-norm problem for $\GL_2$ over any number field simultaneously in the eigenvalue and the level aspect, provided the level is square-free. In certain cases, we recover a Weyl-type saving, the strongest bound one can expect with the current technology. To formulate our results, we introduce the tuple
\[\lambda := (\lambda_1, \ldots, \lambda_{r_1}, \lambda_{r_1 + 1}, \ldots, \lambda_{r_1 + r_2})\]
of Laplace eigenvalues at the $r_1$ real places and the $r_2$ complex places, and we write
\begin{equation}\label{lambda-norm}
|\lambda|_{\infty}:=|\lambda|_{\RR}\cdot|\lambda|_{\CC},\qquad
|\lambda|_{\RR}:=\prod_{j=1}^{r_1} \lambda_j, \qquad |\lambda|_{\CC}:=\prod_{j=r_1+1}^{r_1+r_2} \lambda_j^2.
\end{equation}
As usual, empty products are defined to be $1$.
We also denote by $\N\mn$ the norm of an integral ideal $\mn$ (see Section~\ref{section1} for further notation).

In classical language, we are looking at a cusp form $\phi$ on a congruence manifold $X$ (see Section~\ref{section1} for precise definitions). The connected components of $X$ correspond to the ideal classes of $F$: they are left quotients of $(\Hh^2)^{r_1}\times(\Hh^3)^{r_2}$ by $\Gamma_0(\mn)$ and related level $\mn$ subgroups (cf.\ \cite{Sh}). Assuming that $\|\phi\|_2=1$ holds with respect to the \emph{probability measure} coming from invariant measures on $\Hh^2$ and $\Hh^3$, the generic local bound reads
\[\|\phi\|_\infty\ll_{F,\eps}|\lambda|_{\infty}^{1/4+\eps} (\N\mn)^{1/2+\eps}.\]

\begin{theorem}\label{thm2} Let $\phi$ be an $L^2$-normalized Hecke--Maa{\ss} cuspidal newform on $\GL_2$ over $F$ of square-free level~$\mn$ and trivial central character. Suppose that $\phi$ is spherical at the archimedean places. Then for any $\eps>0$ we have
\[\|\phi\|_\infty\ll_{F,\eps}
|\lambda|_{\infty}^{5/24+\eps} (\N\mn)^{1/3+\eps} + |\lambda|_{\RR}^{1/8+\eps} |\lambda|_{\CC}^{1/4+\eps} (\N\mn)^{1/4+\eps}.\]
\end{theorem}

We emphasize that this result is new with any exponent less than $1/2$ over $\N\mn$, any exponent less than $1/4$ over $|\lambda|_{\RR}$, and for any number field $F$ other than $\QQ$ and $\QQ(i)$ (cf.\ \cite{Te,BHM}). In particular, for totally real fields this is the proper analogue of \eqref{templier}. In view of the above remarks on the difficulties with general number fields, it is remarkable that the methods in the level aspect -- which historically appeared to be the harder parameter -- are flexible enough to produce a Weyl-type exponent in a general setup.

For a general number field $F$, the strength of Theorem~\ref{thm2} in the eigenvalue aspect $|\lambda|_{\infty}$ depends on the relative sizes of $|\lambda|_{\RR}$ and $|\lambda|_{\CC}$. It is particularly strong for totally real fields; for other fields, it fails to solve the sup-norm problem when $|\lambda|_{\CC}$ gets large relative to $|\lambda|_{\RR}$ and $\N\mn$. The next theorem, in which $F_0$ denotes the maximal totally real subfield of $F$, fixes this issues by saving in all aspects for any number field other than a totally real field.

\begin{theorem}\label{thm3} Suppose that $[F:F_0]\geq 2$. Then, under the same assumptions as in Theorem~\ref{thm2}, we have
\[\|\phi\|_\infty\ll_{F,\eps}\bigl(|\lambda|^{1/2}_{\infty}\N\mn\bigr)^{\frac{1}{2}-\frac{1}{8[F:F_0]-4}+\eps}.\]
\end{theorem}

In the special case $[F:F_0]=2$ this bound reads
\[\|\phi\|_\infty\ll_{F,\eps}|\lambda|_{\infty}^{5/24+\eps}(\N\mn)^{5/12+\eps},\]
so Theorem~\ref{thm3} improves on \cite[Theorems~2--3]{BHM} even in the case $F = \QQ(i)$, and the proof differs substantially in several aspects.
Further, Theorem~\ref{thm3} with any exponent less than $1/4$ over $|\lambda|_{\infty}$ is new for any non-CM-field.
For a sequence of fields with $[F:F_0]\to\infty$, the exponents of $|\lambda|_{\infty}$ and $\N\mn$ degenerate to $1/4$ and $1/2$, respectively, but this defect only impacts the $|\lambda|_{\CC}$-aspect due to the uniform exponents in Theorem~\ref{thm2}. It would be desirable to treat all number fields on equal footing (as was accomplished in other contexts such as \cite{BB,BV}), but for that a new idea (or a completely new method) would be needed to handle more efficiently the difficulties described in the previous subsection.

Recently, Assing~\cite{As} extended Theorems~\ref{thm2} and \ref{thm3} to arbitrary level and central character by combining the ideas of the present paper with the methods of Saha~\cite{Sa}. In Assing's results, the dependence on the Laplace eigenvalues and the square-free part of the level is the same as ours, and this is coupled with a rather good dependence on the (remaining) square part of the level and the conductor of the central character.

\begin{convention}
In this paper, we regard the number field $F$ as being fixed, and we allow all implied constants to depend on it (unless we emphasize the opposite).
\end{convention}

\begin{acknowledgement} We thank the referee for a thorough reading and several constructive remarks.
\end{acknowledgement}

\section{Basic setup and notation}\label{section1}

Let $F$ be a number field of degree $n = r_1 + 2r_2$ over $\QQ$ with ring of integers $\mo$ and different ideal $\md$.
The completions $F_v$ at the various places $v$ are equipped with canonical norms (or modules) as in \cite{We}. In particular, at an archimedean place $v$ we have $|x|_v = |x|^{[F_v:\RR]}$, where $|\cdot|$ denotes the usual absolute value. We reserve the symbol $\mpr$ to the prime ideals of $\mo$, and we use it to label the non-archimedean places of $F$ in the usual way. For each prime $\mpr$, we fix a uniformizer $\varpi_\mpr\in\mo_\mpr$ of $\mpr\mo_\mpr$. As usual, we define the adele ring of $F$ as a restricted direct product
\[\Aa:=F_\infty\times\Aa\fin,\qquad
F_\infty:=\prod_{v\mid\infty}F_v,\qquad\Aa\fin:=\sideset{}{'}\prod_\mpr F_\mpr,\]
and we write accordingly
\[|x|_\Aa:=|x_\infty|_\infty\cdot|x\fin|\fin,\qquad
|x_\infty|_\infty:=\prod_{v\mid\infty}|x_v|_v,\qquad|x\fin|\fin:=\prod_\mpr|x_\mpr|_\mpr\]
for the module of an idele $x\in\Aa^\times$. We further decompose the archimedean module as
\begin{equation}\label{RnormCnorm}
|x_\infty|_\infty=|x_\infty|_{\RR}\cdot|x_\infty|_{\CC},\qquad
|x_\infty|_{\RR}:=\prod_{\text{$v$ real}}|x_v|,\qquad |x_\infty|_{\CC}:=\prod_{\text{$v$ complex}}|x_v|^2.
\end{equation}
This is consistent with \eqref{lambda-norm}. We introduce the following notation for the closure of $\mo$ in $\Aa\fin$:
\begin{equation}\label{hatmonew}\hat\mo:=\prod_\mpr\mo_\mpr.\end{equation}

We call a field element $x\in F^\times$ \emph{totally positive} if $x_v>0$ holds at every real place $v$. We denote the group of totally positive field elements by $F^\times_+$, and the group of totally positive units by $\mo^\times_+$. We choose a set of representatives $\theta_1,\dots,\theta_h\in\Aa\fin^\times$ for the ideal classes of $F$; without loss of generality, they lie in $\hat\mo$.

As mentioned in the introduction, we can fix a square-free ideal $\mq\subseteq\mo$ in such a way that the only elements of $\mo^\times$ that are quadratic residues modulo $\mq$ are the elements of $(\mo^\times)^2$. Indeed, if $u$ is a \emph{non-square} unit, then $F\bigl(\sqrt{u}\bigr)/F$ is one of the finitely many quadratic extensions corresponding to the square classes in $\mo^\times/(\mo^\times)^2$. Moreover, for any prime $\mpr$ that is inert in this extension, $u$ is a quadratic non-residue modulo $\mpr$. So if we choose an inert prime for each of the mentioned extensions, and $\mq$ is divisible by all these primes, then $\mq$ has the required property. We fix such an ideal $\mq$ once and for all, with the additional requirement that
\begin{equation}\label{sizeq}
\N\mq\geq 300^n.
\end{equation}
We can clearly think of $\mq$ as a function of $F$. (For concreteness, we could pick $\mq$ so that its norm is minimal, and with additional constraints we could even pin down $\mq$ uniquely).

We fix a square-free ideal $\mn\subseteq\mo$, and we consider the corresponding global Hecke congruence subgroup
\begin{equation}\label{Kdef}
K:=\prod_v K_v\qquad\quad\text{with}\qquad\quad
K_v:=\begin{cases}
\oO_2(\RR)&\text{for $v$ real},\\
\U_2(\CC)&\text{for $v$ complex},\\
\GL_2(\mo_\mpr)&\text{for $\mpr\nmid\mn$},\\
K_0(\mpr\mo_\mpr)&\text{for $\mpr\mid\mn$},
\end{cases}
\end{equation}
where
\[K_0(\mpr\mo_\mpr):=\left\{\abcd\in\GL_2(\mo_\mpr):\ c\in\mpr\mo_\mpr\right\}\]
is the subgroup of $\GL_2(\mo_\mpr)$ consisting of the matrices whose lower left is entry divisible by $\mpr\mo_\mpr$.
As explained in the introduction, we need to enlarge our spectrum a bit. With this in mind, we introduce
\begin{equation}\label{Kdef3}
K^\flat:=\prod_v K^\flat_v\qquad\quad\text{with}\qquad\quad
K^\flat_v:=\begin{cases}
K_v&\text{for $v\nmid\mq$},\\
K_1(\mpr\mo_\mpr)&\text{for $\mpr\mid \mq$},
\end{cases}
\end{equation}
where
\[K_1(\mpr\mo_\mpr):=\left\{\abcd\in\GL_2(\mo_\mpr):\ a-d\in\mpr\mo_\mpr,\ c\in\mpr\mo_\mpr\right\}\]
is the subgroup of $K_0(\mpr\mo_\mpr)$ consisting of the matrices whose diagonal entries are congruent to each other modulo $\mpr\mo_\mpr$.

We fix a Haar measure on $\GL_2(\Aa)$, and we use it to define the Hilbert space $L^2(\tilde X)$, where $\tilde X$ is the finite volume coset space\footnote{For any ring $R$, we denote by $\Z(R)$ the matrix group $\bigl\{\bigl(\begin{smallmatrix}a&0\\0&a\end{smallmatrix}\bigr): a\in R^\times\bigr\}$.}
\begin{equation}\label{Xtildedef}
\tilde X:=\GL_2(F)\backslash\GL_2(\Aa)\slash\Z(F_\infty).
\end{equation}
All the other $L^2$-spaces in this paper will be regarded or defined as Hilbert subspaces of $L^2(\tilde X)$.
We consider a spherical Hecke--Maa{\ss} newform $\phi$ on $\GL_2$ over $F$ of level $\mn$ and trivial central character. By definition,
$\phi:\GL_2(\Aa)\to\CC$ is a left $\GL_2(F)$-invariant and right $\Z(\Aa)K$-invariant function that generates an irreducible cuspidal representation $\pi$ of $\GL_2$ over $F$ of conductor $\mn$. It spans the one-dimensional newspace $\pi^K$, and it corresponds to a pure tensor $\otimes_v\phi_v$ of local newvectors $\phi_v\in\pi_v^{K_v}$ (cf.\ \cite[Cor~2]{Mi}, \cite[Th.~4]{Fl}, \cite[Th.~1]{Ca}). In particular, $\phi$ is a cuspidal eigenfunction of the Hecke algebra for $K$ (as defined in Section~\ref{section2}). Inspired by Venkatesh~\cite[Subsection~2.3]{Ve1}, we regard $\phi$ as a square-integrable function on the coset spaces
\begin{equation}\label{Xdef}
X:=\tilde X\slash K\qquad\text{and}\qquad X^{\flat}:=\tilde X\slash K^{\flat}.
\end{equation}
More precisely, we identify $L^2(X)$ and $L^2(X^{\flat})$ with the right $K$-invariant and right $K^{\flat}$-invariant subspaces of $L^2(\tilde X)$ defined above.

In adelic treatments, one usually divides by $\Z(\Aa)$ instead of $\Z(F_\infty)$, especially if the central character
is assumed to be trivial. Dividing by the smaller group $\Z(F_\infty)$ in \eqref{Xtildedef} makes the spaces in \eqref{Xdef} larger and separates
the infinite part and the finite part nicely. The cost to pay is that one has to deal with a bigger automorphic spectrum:
instead of the trivial central character, one needs to consider all ideal class characters as central characters\footnote{This
subtlety enters in \eqref{Hecke}, where we assume that $\gcd(\ml,\mm)$ is a product of principal prime ideals.}. Introducing the ideal $\mq$, i.e.\ switching from $X$ to $X^\flat$, allows one to work with Hecke operators of smaller support (see Section~\ref{section2}), which is immensely beneficial for our matrix counting scheme (see Section~\ref{countingmatrices}). However, this has a similar effect (already for $F=\QQ$) of enlarging the automorphic spectrum. Indeed, the resulting quotients $X$ and $X^\flat$ are orbifolds with finitely many connected components. The connected components of $X$ correspond to the ideal classes of $F$, while those of $X^\flat$ correspond to certain cosets of the ray class group modulo $\mq$. More concretely, reduction modulo $\mq$ embeds $U:=\mo^\times/(\mo^\times)^2$ into $V:=(\mo/\mq)^\times/(\mo/\mq)^{\times 2}$ by our choice of $\mq$, and each connected component of $X$ is covered by exactly $[V:U]$ connected components of $X^\flat$ under the natural covering map $X^\flat\to X$.

For a ramified place $v$ (i.e.\ for $v=\mpr$ dividing the level $\mn$), the matrix $A_v:=\bigl(\begin{smallmatrix}&1\\\varpi_\mpr&\end{smallmatrix}\bigr)\in\GL_2(F_v)$ normalizes $K_v$. The group $K_v^\ast$ generated by $K_v$ and $A_v$ contains the center $\Z(F_v)$, and $K_v^\ast/\Z(F_v)K_v$ has order $2$. By multiplicity one and the assumption that the central character of $\phi$ is trivial, we infer that the right action of $K_v^\ast$ on $\phi$ is given by a character $K_v^\ast\to\{\pm 1\}$. It follows that $\abs{\phi}$ is right invariant by the global Atkin--Lehner group $K^\ast:=\prod_v K_v^\ast$, where we put $K_v^\ast:=\Z(F_v)K_v$ for all unramified places $v$, including the archimedean ones. That is, for the purpose of studying the sup-norm $\|\phi\|_{\infty}$, we can regard $\abs{\phi}$ as a square-integrable function on the coset space
\begin{equation}\label{Xdef2}
X^\ast:=\GL_2(F)\backslash\GL_2(\Aa)\slash K^\ast.
\end{equation}
We emphasize again that we regard $L^2(X^\ast)$, $L^2(X)$, $L^2(X^\flat)$ as Hilbert subspaces of $L^2(\tilde X)$, and we assume that $\|\phi\|_2=1$ holds in all these spaces.

We allow all implied constants depend on the number field $F$, hence also on the auxiliary ideal $\mq$ chosen above for $F$. Accordingly, $A\ll B$ means that $|A|\leq C|B|$ holds for a constant $C=C(F)>0$ depending on $F$, while $A\ll_S B$ means the same for a constant $C=C(F,S)>0$ depending on $F$ and $S$. If $S$ is a list of quantities including $\eps$, then it is implicitly meant that the bound holds for any sufficiently small $\eps>0$. The relation $A\asymp B$ means that $A\ll B$ and $B\ll A$ hold simultaneously, while $A\asymp_S B$ means that $A\ll_S B$ and $B\ll_S A$ hold simultaneously. Finally, inspired by \cite{BHM,HT2}, we shall use the notation
\begin{equation}\label{notationformula}
A \preccurlyeq B \qquad \overset{\text{def}}\Longleftrightarrow \qquad A \ll_\eps B|\lambda|_\infty^\eps(\N\mn)^{\eps},
\end{equation}
which will be in force for the rest of the paper.

\section{Hecke algebras and the idea of amplification}\label{section2}

While our newform $\phi$ lives naturally on the space $X$, and in fact $|\phi|$ is well-defined even on the space $X^{\ast}$, it is convenient to view $\phi$ as a function on $X^{\flat}$ which is equipped with more suitable operators for the purpose of amplification.

The groups $\Z(F_\infty)K$ and $\Z(F_\infty)K^\flat$ contain the central subgroup (cf.\ \eqref{hatmonew}, \eqref{Kdef}, \eqref{Kdef3})
\begin{equation}\label{hatmo}
Z:=\Z(F_\infty\hat\mo)=\prod_{v\mid\infty}\Z(F_v)\prod_{\mpr}\Z(\mo_\mpr),
\end{equation}
hence we can identify $X$ with $\Gamma\backslash G\slash K$, and $X^\flat$ with $\Gamma\backslash G\slash K^\flat$, where
\begin{equation}\label{groups}
G:=\GL_2(\Aa)\slash Z,\qquad \Gamma:=\GL_2(F)Z\slash Z\cong\GL_2(F)\slash\Z(\mo).
\end{equation}
In particular, we can identify the functions on $X$ (resp.\ $X^\flat$) with those functions on $G$ that are left $\Gamma$-invariant and right $K$-invariant (resp.\ right $K^\flat$-invariant).
Accordingly, we have an inclusion of Hilbert spaces, each defined via the Haar measure that we fixed on $\GL_2(\Aa)$,
\begin{equation}\label{L2spaces}
L^2(X)\leq L^2(X^\flat)\leq L^2(\Gamma\backslash G)\leq L^2(\tilde X).
\end{equation}

We define the $\Z(F_\infty)$-invariant norm
\[\|g_\infty\|:=\prod_{v\mid\infty}\frac{|a_v|^2+|b_v|^2+|c_v|^2+|d_v|^2}{2|a_vd_v-b_vc_v|},\qquad g_\infty=\abcd\in\GL_2(F_\infty),\]
and we say that $f:G\to\CC$ is a \emph{rapidly decaying smooth} function if the following properties hold for $g=g_\infty g\fin$, where $g_\infty\in\GL_2(F_\infty)$ and $g\fin\in\GL_2(\Aa\fin)$:
\begin{itemize}
\medskip
\item $f(g)$ is compactly supported in $g\fin$, and it is locally constant in $g\fin$ for any fixed $g_\infty$;
\medskip\smallskip
\item $\|g_\infty\|^{N}f(g)$ is bounded for any $N>0$, and $f(g)$ is $C^\infty$ in $g_\infty$ for any fixed $g\fin$.
\medskip
\end{itemize}
We denote by $C(G)$ the convolution algebra of these rapidly decaying smooth functions on $G$. For $f\in C(G)$ and $\psi\in L^2(\Gamma\backslash G)$, we consider the function $R(f)\psi\in L^2(\Gamma\backslash G)$ given by
\begin{equation}\label{Rf1}
(R(f)\psi)(x):=\int_{G}f(y)\psi(xy)\,dy=\int_{G}f(x^{-1}y)\psi(y)\,dy=\int_{\Gamma\backslash G}\Bigl(\sum_{\gamma\in\Gamma}f(x^{-1}\gamma y)\Bigr)\psi(y)\,dy.
\end{equation}
That is, $R(f)$ is an integral operator on $\Gamma\backslash G$ with kernel
\begin{equation}\label{Rf2}
k_f(x,y):=\sum_{\gamma\in\Gamma}f(x^{-1}\gamma y).
\end{equation}
Then $R(f_1\ast f_2)=R(f_1)R(f_2)$ for $f_1,f_2\in C(G)$, and the adjoint of $R(f)$ equals $R(\check f)$ with
\[\check f(g):=\ov{f(g^{-1})},\qquad g\in G.\]

We shall define convenient subalgebras of $C(G)$ in terms of the restricted product decomposition
\begin{equation}\label{groups2}
G=\sideset{}{'}\prod_v G_v,\qquad
G_v:=\begin{cases}
\GL_2(F_v)/\Z(F_v)&\text{for $v\mid\infty$},\\
\GL_2(F_\mpr)/\Z(\mo_\mpr)&\text{for $v\nmid\infty$}.
\end{cases}\end{equation}
We choose a Haar measure on each of the groups $\GL_2(F_v)$ so that their product is the Haar measure we fixed on $\GL_2(\Aa)$ earlier, and the measure of $\GL_2(\mo_\mpr)$ within $\GL_2(F_\mpr)$ is $1$. We define $C(G_v)$ and its action $f_v\mapsto R(f_v)$ on $L^2(\Gamma\backslash G)$ similarly as for $C(G)$, but with integration over $G_v$ instead of $G$. The restricted tensor product of these algebras is the $\CC$-span of pure tensors $\otimes_v f_v$ such that $f_v\in C(G_v)$ for all places $v$ and $f_\mpr$ is the characteristic function of $\GL_2(\mo_\mpr)$ for all but finitely many primes $\mpr$. We regard this product as a subalgebra of $C(G)$ in the usual way, namely by identifying $\otimes_v f_v$ with the function $x\mapsto\prod_v f_v(x_v)$ so that also $R(\otimes_v f_v)=\prod_v R(f_v)$; the products are finite in the sense that the factors equal the identity for all but finitely many $v$'s.

We write $\hat\GL_2(\mo_\mpr)$ for $\M_2(\mo_\mpr)\cap\GL_2(F_\mpr)$, and we define $\hat K_0(\mpr\mo_\mpr)$ as the subsemigroup of $\hat\GL_2(\mo_\mpr)$ consisting of the matrices with lower left entry divisible by $\mpr$ and upper left entry coprime to $\mpr$.
In accordance with \eqref{Kdef} and \eqref{Kdef3}, we consider the following two open subsemigroups of $G$:
\begin{equation}\label{Mdef1}
M:=\prod_v M_v\qquad\quad\text{with}\qquad\quad
M_v:=\begin{cases}
G_v&\text{for $v\mid\infty$},\\
\hat\GL_2(\mo_\mpr)/\Z(\mo_\mpr)&\text{for $\mpr\nmid\mn$},\\
\hat K_0(\mpr\mo_\mpr)/\Z(\mo_\mpr)&\text{for $\mpr\mid\mn$};
\end{cases}
\end{equation}
\begin{equation}\label{Mdef3}
M^\flat:=\prod_v M^\flat_v\qquad\quad\text{with}\qquad\quad
M^\flat_v:=\begin{cases}
M_v&\text{for $v\nmid\mq$},\\
K_1(\mpr\mo_\mpr)/\Z(\mo_\mpr)&\text{for $\mpr\mid \mq$}.
\end{cases}
\end{equation}
Note that $M$ (resp.\ $M^\flat$) is left and right invariant by $K$ (resp.\ $K^\flat$). Finally, we define the \emph{Hecke algebra} for $K$, and the \emph{unramified Hecke algebra} at $\mq$ for $K^\flat$, as the restricted tensor products
\[\Hs:=\sideset{}{'}\bigotimes_v\Hs_v\qquad\quad\text{with}\qquad\quad\Hs_v:=C(K_v\backslash M_v\slash K_v),\]
\[\Hs^\flat:=\sideset{}{'}\bigotimes_v\Hs^\flat_v\qquad\quad\text{with}\qquad\quad\Hs^\flat_v:=C(K^\flat_v\backslash M^\flat_v\slash K^\flat_v).\]
These algebras have a unity element, unlike $C(G)$ or $C(G_v)$.

Note that $\Hs$ acts on $L^2(X)$, and $\Hs^\flat$ acts on $L^2(X^\flat)$, through $f\mapsto R(f)$. There is a $\CC$-algebra embedding $\iota:\Hs^\flat\hookrightarrow\Hs$ such that any $f\in\Hs^\flat$ acts on the subspace $L^2(X)$ of $L^2(X^\flat)$ exactly as $\iota(f)\in\Hs$ does. We define this embedding as $\iota:=\otimes_v\iota_v$, by choosing an appropriate $\CC$-algebra embedding $\iota_v:\Hs^\flat_v\hookrightarrow\Hs_v$ at each place $v$. For $v\nmid\mq$, the local factors $\Hs^\flat_v$ and $\Hs_v$ are equal, so we choose $\iota_v$ to be the identity map. For $v\mid\mq$, the local factor $\Hs^\flat_v$ is isomorphic to $\CC$, so there is a unique choice for $\iota_v$. From now on, we use the usual convention that the subscript ``$\infty$'' collects the local factors at $v\mid\infty$, while the subscript ``fin'' collects the local factors at $v\nmid\infty$. Then, in particular, we can talk about the $\CC$-algebra embedding $\iota\fin:\Hs^\flat\fin\hookrightarrow\Hs\fin$. Under this embedding, thinking of $\Hs\fin$ (resp.\ $\Hs^\flat\fin$) as a subalgebra of $C(K\fin\backslash G\fin\slash K\fin)$ (resp.\ $C(K^\flat\fin\backslash G\fin\slash K^\flat\fin)$), the constant function $\vol(K^\flat\fin)^{-1}$ on a double coset $K^\flat\fin gK^\flat\fin\subset M^\flat\fin$ becomes the constant function $\vol(K\fin)^{-1}$ on the double coset $K\fin gK\fin\subset M\fin$. In the next two paragraphs, we introduce the Hecke operators for $X$ in terms of $\Hs\fin$, and the (unramified at $\mq$) Hecke operators for $X^\flat$ in terms of $\Hs^\flat\fin$. Our presentation is based to some extent on \cite[Section~2]{Sh} and \cite[Ch.~3]{Sh2}.

For any nonzero ideal $\mm\subseteq\mo$, we consider the Hecke operator $T_\mm:=R(t_\mm)$ on $L^2(X)$, where $t_\mm\in\Hs\fin$ is given by
\begin{equation}\label{tmdef}
t_\mm(x):=\begin{cases}
(\N\mm)^{-1/2}\vol(K\fin)^{-1}&\text{for $x\in M\fin$ and $(\det x)\mo=\mm$},\\
0&\text{otherwise}.\end{cases}
\end{equation}
We note that here the determinant of $x\in\M_2(\hat\mo)/\Z(\hat\mo)$ is an element of $\hat\mo/(\hat\mo^\times)^2$ rather than an element of $\hat\mo$, but still it determines a unique ideal in $\mo$ that we denoted by $(\det x)\mo$. We also need the supplementary operator $S_\mm:=R(s_\mm)$ on $L^2(X)$, with $s_\mm\in\Hs\fin$ defined as follows.
First we assume that $\mm$ is coprime to the level $\mn$.
We take any finite idele $\mu\in\Aa\fin^\times$ representing $\mm$, i.e.\ $\mm=\mu\mo$, and then we put
\begin{equation}\label{smdef}
s_\mm(x):=\begin{cases}
\vol(K\fin)^{-1}&\text{for $x\in \bigl(\begin{smallmatrix}\mu&\\&\mu\end{smallmatrix}\bigr)K\fin/\Z(\hat\mo)$},\\
0&\text{otherwise}.\end{cases}
\end{equation}
The function $s_\mm\in\Hs\fin$ is independent of the representative $\mu$, because $K\fin$ contains $\Z(\hat\mo)$, and
$S_\mm$ agrees with the right action of $\bigl(\begin{smallmatrix}\mu&\\&\mu\end{smallmatrix}\bigr)$. In particular, $S_\mm$ commutes with the Hecke operators. Using that $L^2(X)$ consists of functions invariant under $\Z(FF_\infty\hat\mo)$, we see that $S_\mm$ only depends on the ideal class of $\mm$, and it is the identity map whenever $\mm$ is principal. If $\mm$ and $\mn$ are not coprime, then we define $s_\mm$ (hence also $S_\mm$) to be zero. The Hecke operators commute with each other as a consequence of the following multiplicativity relation, valid for all nonzero ideals $\ml,\mm\subseteq\mo$ (cf.\ \cite[(2.12)]{Sh}):
\begin{equation}\label{Heckeoriginal}
t_\ml\ast t_\mm=\sum_{\mk|\gcd(\ml,\mm)}s_\mk\ast t_{\ml\mm/\mk^2},\qquad\text{therefore,}\qquad T_\ml T_\mm=\sum_{\mk|\gcd(\ml,\mm)}S_\mk T_{\ml\mm/\mk^2}.
\end{equation}
In addition, if $\mm$ is coprime to the level $\mn$, then we have that $t_\mm=s_\mm\ast\check t_\mm$, whence $T_\mm$ is a normal operator, and it is even self-adjoint when $\mm$ is principal.

For any ideal $\mm\subseteq\mo$ coprime to $\mq$, we define the functions $t^\flat_\mm,s^\flat_\mm\in\Hs^\flat\fin$
in the same way as $t_\mm,s_\mm\in\Hs\fin$, but with $K\fin$ and $M\fin$ replaced by $K^\flat\fin$ and $M^\flat\fin$ (cf.\ \eqref{tmdef}--\eqref{smdef}). In particular,
\begin{equation}\label{tmdefflat}
t^\flat_\mm(x):=\begin{cases}
(\N\mm)^{-1/2}\vol(K^\flat\fin)^{-1}&\text{for $x\in M^\flat\fin$ and $(\det x)\mo=\mm$},\\
0&\text{otherwise}.\end{cases}
\end{equation}
The corresponding operators on $L^2(X^\flat)$ are $T^\flat_\mm:=R(t^\flat_\mm)$ and $S^\flat_\mm:=R(s^\flat_\mm)$.
Then in fact $\iota\fin(t^\flat_\mm)=t_\mm$ and $\iota\fin(s^\flat_\mm)=s_\mm$ under the $\CC$-algebra embedding $\iota\fin:\Hs^\flat\fin\hookrightarrow\Hs\fin$, hence \eqref{Heckeoriginal} implies the analogous relations
\[t^\flat_\ml\ast t^\flat_\mm=\sum_{\mk|\gcd(\ml,\mm)}s^\flat_\mk\ast t^\flat_{\ml\mm/\mk^2},\qquad\text{therefore,}\qquad T^\flat_\ml T^\flat_\mm=\sum_{\mk|\gcd(\ml,\mm)}S^\flat_\mk T^\flat_{\ml\mm/\mk^2}.\]
An important special case is when $\gcd(\ml,\mm)$ is a product of \emph{principal prime ideals} not dividing $\mn\mq$. In this case, the above relations simplify to
\begin{equation}\label{Hecke}
t^\flat_\ml\ast t^\flat_\mm=\sum_{\mk|\gcd(\ml,\mm)}t^\flat_{\ml\mm/\mk^2},\qquad\text{therefore,}\qquad T^\flat_\ml
T^\flat_\mm=\sum_{\mk|\gcd(\ml,\mm)}T^\flat_{\ml\mm/\mk^2}.
\end{equation}
In addition, if $\mm$ is coprime to $\mn\mq$, then we have that $t^\flat_\mm=s^\flat_\mm\ast\check t^\flat_\mm$, whence $T^\flat_\mm$ is a normal operator, and it is even self-adjoint when $\mm$ is principal.

Let $f\in\Hs$ be arbitrary. As $\phi\in L^2(X)$ is a newform of level $\mn$, we have $R(f)\phi=c(f)\phi$ for some $c(f)\in\CC$. The same conclusion also holds for $f\in\Hs^\flat$, because in this case $\iota(f)\in\Hs$, and $R(f)\phi=R(\iota(f))\phi$. Moreover, if $f=\otimes_v f_v$ is a pure tensor from $\Hs^\flat$, then $R(f_v)\phi=c(f_v)\phi$ for some $c(f_v)\in\CC$, and $c(f)=\prod_v c(f_v)$; there is a similar decomposition $c(f)=c(f_\infty)c(f\fin)$ for partial tensors $f=f_\infty\otimes f\fin\in\Hs^\flat_\infty\otimes\Hs^\flat\fin$.
In particular, $\phi$ is an eigenfunction of each Hecke operator $T_\mm$ with eigenvalue
\begin{equation}\label{lambdadef}
\lambda(\mm):=c(t_\mm),
\end{equation}
and for $\mm$ coprime $\mq$ it is also an eigenfunction of $T^\flat_\mm$ with the same eigenvalue.

Now we describe along these lines the idea of \emph{amplification}, a technique pioneered by Duke, Friedlander, Iwaniec, and Sarnak~\cite{DFI,FI,IS} to prove efficient bounds for automorphic $L$-functions on the critical line, and also for $|\phi(g)|$ at a given $g\in\GL_2(\Aa)$. Assume that $f\in\Hs^\flat$ is such that the operator $R(f)$ on $L^2(X^\flat)$ is positive. Then the eigenvalue $c(f)$ is nonnegative, and the orthogonal decomposition $L^2(X^\flat)=(\CC\phi)+(\CC\phi)^\perp$ is $R(f)$-invariant (because $R(f)$ is self-adjoint). Any $\psi\in L^2(X^\flat)$ decomposes uniquely as $\psi=\psi_1+\psi_2$, where $\psi_1\in\CC\phi$ and $\psi_2\in(\CC\phi)^\perp$, and therefore
\[\langle R(f)\psi,\psi\rangle=\langle R(f)\psi_1,\psi_1\rangle+\langle R(f)\psi_2,\psi_2\rangle\geq \langle R(f)\psi_1,\psi_1\rangle.\]
On the right hand side, we have explicitly $\psi_1=\langle\psi,\phi\rangle\phi$, hence flipping the two sides we obtain
\[c(f)|\langle\psi,\phi\rangle|^2\leq\langle R(f)\psi,\psi\rangle.\]
The inner products and also $R(f)\psi$ can be expressed as integrals over $X^\flat$ (cf.\ \eqref{L2spaces}--\eqref{Rf2}), yielding
\[c(f)\int_{X^{\flat}\times X^{\flat}} \phi(x)\ov{\phi(y)}\,\psi(y)\ov{\psi(x)}\,dx dy\leq \int_{X^{\flat}\times X^{\flat}} k_f(x,y)\,\psi(y)\ov{\psi(x)}\,dx dy.\]

We can use this inequality to estimate the value $|\phi(g)|$ at the given point $g\in\GL_2(\Aa)$ as follows.
Note that the integrals are over a rather concrete space: an orbifold with finitely many connected components. We take a basis of open neighborhoods $\{V\}\subset X^{\flat}$ of the point $\Gamma gZK^{\flat} \in X^{\flat}$ (the image of the coset $gZ\in G$), and we let $\psi=\psi_V\in L^2(X^\flat)$ run through the corresponding characteristic functions. Then we get by continuity, as $V$ approaches the point $\Gamma gZK^{\flat}\in X^{\flat}$,
\[c(f)\bigl(|\phi(g)|^2+o(1)\bigr)\vol(V\times V)\leq\bigl(k_f(g,g)+o(1)\bigr)\vol(V\times V).\]
We conclude
\begin{equation}\label{ampbound1}
c(f)|\phi(g)|^2\leq k_f(g,g)=\sum_{\gamma\in\Gamma}f(g^{-1}\gamma g).
\end{equation}
This is the pre-trace inequality mentioned in the introduction.
The idea of amplification is to find, in terms of $\phi$, a positive operator $R(f)$ as above such that $c(f)$ is relatively large, while the right hand side is relatively small. By dividing the last inequality by $c(f)$, we see that such an operator gives rise to an upper bound for $|\phi(g)|$. We note that the above argument goes back to Mercer~\cite{Me}; see especially the end of Section~6 in his paper, and see also \cite[Section~98]{RSz} for a modern account.

\section{Iwasawa decomposition modulo Atkin--Lehner operators}\label{atkinlehner}

In the next two sections, we establish a nice fundamental domain for the space $X^{\ast}$ (cf.\ \eqref{Xdef2}), which is the natural habitat of $|\phi|$. We start by developing a variant of the usual Iwasawa decomposition for $\GL_2(F_v)$. The results are probably known to experts.

First we recall the action of $\GLR$ on the hyperbolic plane $\Hh^2$, and the action of $\GLC$ on the hyperbolic 3-space $\Hh^3$.
We identify $\Hh^2$ with a half-plane in the set of complex numbers $\CC=\RR+\RR i$,
\begin{equation}\label{h2}\Hh^2:=\{x+yi:x\in\RR,\ y>0\}\subset\CC.\end{equation}
A matrix $\sabcd\in\GL_2^+(\RR)$ of positive determinant maps a point $P\in\Hh^2$ to
$(aP+b)(cP+d)^{-1}\in\Hh^2$, while $\bigl(\begin{smallmatrix}-1&\\&1\end{smallmatrix}\bigr)\in\GLR$ maps it to $-\ov{P}\in\Hh^2$. This determines a transitive action of $\GLR$ on $\Hh^2$, and by examining the stabilizer of the point $i\in\Hh^2$, we see that
\begin{equation}\label{eq1}\Hh^2\cong \GL_2(\RR)\slash \Z(\RR)\oO_2(\RR).\end{equation}
Similarly, we identify $\Hh^3$ with a half-space in the set of Hamilton quaternions $\HH=\RR+\RR i+\RR j+\RR k$,
\begin{equation}\label{h3}\Hh^3:=\{x+yj:x\in\CC,\ y>0\}\subset\HH.\end{equation}
A matrix $\sabcd\in\GL_2^+(\CC)$ of positive real determinant maps a point $P\in\Hh^3$ to $(aP+b)(cP+d)^{-1}\in\Hh^3$, while any central element $\bigl(\begin{smallmatrix}a&\\&a\end{smallmatrix}\bigr)\in\GLC$ fixes it. This determines a transitive action of $\GLC$ on $\Hh^3$, and by examining the stabilizer of the point $j\in\Hh^3$, we see that
\begin{equation}\label{eq2}\Hh^3\cong \GL_2(\CC)\slash \Z(\CC)\U_2(\CC).\end{equation}

The following two lemmas provide explicit local Iwasawa decompositions; in particular, Lemma~\ref{lemma1} (with $y>0$) explicates the isomorphisms \eqref{eq1} and \eqref{eq2}. Recall that $K_v^{\ast}=\Z(F_v)K_v$ for $v\mid\infty$.

\begin{lemma}\label{lemma1}
Let $v\mid\infty$ be an archimedean place. Any matrix $\sabcd\in\GL_2(F_v)$ can be decomposed as
\begin{equation}\label{eq3}\abcd=\yx k,\end{equation}
where\footnote{For any ring $R$, we denote by $\gP(R)$ the matrix group $\bigl\{\bigl(\begin{smallmatrix}a&b\\0&d\end{smallmatrix}\bigr): a,d\in R^\times, b\in R\bigr\}$.} $\syx\in\gP(F_v)$ and $k\in K_v^\ast$. Moreover, the absolute value of $y$ is uniquely determined by
\begin{equation}\label{eq4}|y|=\frac{|ad-bc|}{|c|^2+|d|^2}.\end{equation}
\end{lemma}

\begin{proof}
Existence with a unique $y>0$ is clear from our remarks above, especially from \eqref{eq1} and \eqref{eq2}. Equation \eqref{eq4} is well-known: we multiply both sides of \eqref{eq3} with its conjugate transpose. We have $k\in\bigl(\begin{smallmatrix}u&\\&u\end{smallmatrix}\bigr)K_v$ for some $u\in F_v^\times$, and then we get
\[\left(\begin{matrix}|a|^2+|b|^2&\ast\\\ast&|c|^2+|d|^2\end{matrix}\right)=
\left(\begin{matrix}|u|^2&\\&|u|^2\end{matrix}\right)\left(\begin{matrix}|y|^2+|x|^2&x\\\ov{x}&1\end{matrix}\right).\]
It follows that $|u|^2=|c|^2+|d|^2$, while taking the determinant of both sides reveals that
$|ad-bc|^2=|u|^4 |y|^2$, and the claim follows.
\end{proof}

\begin{lemma}\label{lemma2}
Let $v=\mpr$ be a non-archimedean place. Any matrix $\sabcd\in\GL_2(F_v)$ can be decomposed as
\begin{equation}\label{eq5}\abcd=\yx k,\end{equation}
where $\syx\in\gP(F_v)$ and $k\in K_v^\ast$. Moreover, the $\mpr$-adic absolute value of $y$ is uniquely determined by
\begin{equation}\label{eq6}|y|_v=\begin{cases}
|(ad-bc)/\gcd(c,d)^2|_v&\text{when $|c|_v<|d|_v$ or $\mpr\nmid\mn$},\\
|\varpi_\mpr(ad-bc)/\gcd(c,d)^2|_v&\text{when $|c|_v\geq|d|_v$ and $\mpr\mid\mn$}.
\end{cases}\end{equation}
Here, $\gcd(c,d)$ stands for any generator of the $\mo_\mpr$-ideal $c\mo_\mpr+d\mo_\mpr$. Similarly, the image of $k$ in the group
$K_v^\ast/\Z(F_v)K_v$ is uniquely determined, namely
\begin{equation}\label{eq6b}k\in\begin{cases}
\Z(F_v)K_v&\text{when $|c|_v<|d|_v$ or $\mpr\nmid\mn$},\\
\Z(F_v)K_vA_v&\text{when $|c|_v\geq|d|_v$ and $\mpr\mid\mn$}.
\end{cases}\end{equation}
\end{lemma}

\begin{proof}
First we show that a decomposition of the form \eqref{eq5} exists with a $y$-coordinate satisfying \eqref{eq6} and a $k$-component satisfying \eqref{eq6b}. We start with the decomposition, valid for $d\neq 0$,
\[\abcd=\left(\begin{matrix}(ad-bc)/d^2&b/d\\&1\end{matrix}\right)
\left(\begin{matrix}d&\\c&d\end{matrix}\right).\]
This is of the form \eqref{eq5} as long as $|c|_v\leq|d|_v$ and $\mpr\nmid\mn$, or $|c|_v<|d|_v$ and $\mpr\mid\mn$ (because $\mn$ is square-free). Moreover, the $y$-coordinate equals here $(ad-bc)/d^2$, and also $c\mo_\mpr+d\mo_\mpr=d\mo_\mpr$ by $|c|_v\leq|d|_v$, which verifies \eqref{eq6} for this particular case. Similarly, the $k$-component equals here $\bigl(\begin{smallmatrix}d&\\c&d\end{smallmatrix}\bigr)\in\Z(F_v)K_v$, so that \eqref{eq6b} holds as well. The two cases in which we have established \eqref{eq5} can be summarized as the case of $|c/w|_v\leq|d|_v$, where we put
\[w:=\begin{cases}
1&\text{for $\mpr\nmid\mn$},\\
\varpi_\mpr&\text{for $\mpr\mid\mn$}.
\end{cases}\]

Assume now that we are in the complementary case $|c/w|_v>|d|_v$ (including the case $d=0$), so that in particular $|c|_v\geq|d|_v$. We consider the decomposition
\[\abcd=\left(\begin{matrix}b&a/w\\d&c/w\end{matrix}\right)
\left(\begin{matrix}&1\\w&\end{matrix}\right).\]
By our initial case, the first factor on the right hand side has a suitable decomposition
\[\left(\begin{matrix}b&a/w\\d&c/w\end{matrix}\right)=\yx k\]
with $k\in\Z(F_v)K_v$, $x\in F_v$, and
\begin{equation}\label{eq7}
y:=\frac{(ad-bc)/w}{(c/w)^2}=\frac{w(ad-bc)}{c^2}.
\end{equation}
Therefore,
\[\abcd=\yx k\left(\begin{matrix}&1\\w&\end{matrix}\right),\]
and this is a suitable decomposition upon regarding the product of the last two factors as a single element $\tilde k\in K_v^\ast$.
Moreover, $c\mo_\mpr+d\mo_\mpr=c\mo_\mpr$ by $|c|_v\geq|d|_v$, hence \eqref{eq7} verifies \eqref{eq6} for this particular case.
Similarly, $k\in\Z(F_v)K_v$ shows that $\tilde k\in\Z(F_v)K_v$ or
$\tilde k\in\Z(F_v)K_vA_v$ depending on whether $\mpr\nmid\mn$ or $\mpr\mid\mn$, and so $\eqref{eq6b}$ holds as well.

Now we prove that the $\mpr$-adic absolute value $|y|_v$ and the image of $k$ in the group $K_v^\ast/\Z(F_v)K_v$ are constant along all decompositions \eqref{eq5} of a given matrix $\sabcd\in\GL_2(F_v)$. In other words,
\[\left(\begin{matrix}y_1&x_1\\&1\end{matrix}\right)k_1=\left(\begin{matrix}y_2&x_2\\&1\end{matrix}\right)k_2\]
for $y_1,y_2\in F_v^\times$, $x_1,x_2\in F_v$, $k_1,k_1\in K_v^\ast$ implies that
\[y_1/y_2\in\mo_\mpr^\times\qquad\text{and}\qquad k_2k_1^{-1}\in\Z(F_v)K_v.\]
Rearranging, we get that
\[\left(\begin{matrix}y_1/y_2&(x_1-x_2)/y_2\\&1\end{matrix}\right)=\left(\begin{matrix}y_2&x_2\\&1\end{matrix}\right)^{-1}\left(\begin{matrix}y_1&x_1\\&1\end{matrix}\right)=k_2k_1^{-1}.\]
In particular, both sides lie in $\gP(F_v)\cap K_v^\ast$ which, by inspection, equals $\Z(F_v)\gP(\mo_\mpr)$. It follows that
$y_1/y_2\in\mo_\mpr^\times$ and $k_2k_1^{-1}\in\Z(F_v)\gP(\mo_\mpr)\subset\Z(F_v)K_v$. The proof is complete.
\end{proof}

\section{Geometry of numbers and the fundamental domain}\label{section4}

We start this section with a simple but useful observation about balancing infinite ideles with units.

\begin{lemma}\label{balancing} Let $y,z\in F_\infty^\times$ be two infinite ideles such that $|y|_\infty=|z|_\infty$. Then for any positive integer $m$, there is an $m$-th powered unit $t\in(\mo^\times)^m$ such that
\[|ty_v|_v\asymp_m |z_v|_v,\qquad v\mid\infty.\]
\end{lemma}

\begin{proof} We fix $m$, and we look for $t\in(\mo^\times)^m$ in the form $t=u^m$ with $u\in\mo^\times$. The infinite idele $s:=z/y\in F_\infty^\times$ satisfies $|s|_\infty=1$, and the conclusion can be rewritten as
\[m\log(|u|_v)=\log(|s_v|_v)+O_m(1),\qquad v\mid\infty.\]
Let us introduce the notation
\begin{equation}\label{logarithmicvector}
l(x):=\bigl(\log(|x_v|_v)\bigr)_{v\mid\infty}\in\prod_{v\mid\infty}\RR,\qquad x\in F_\infty^\times.
\end{equation}
Then, $\{l(u):u\in\mo^\times\}$ is a lattice in the hyperplane
\begin{equation}\label{hyperplane}
W:=\Bigl\{w\in \prod_{v\mid\infty}\RR:\ \sum_{v\mid\infty} w_v=0\Bigr\}
\end{equation}
by Dirichlet's unit theorem (cf.~\cite[p.~93]{We}). As the vector $l(s)/m$ lies in $W$, there exists a lattice point $l(u)$ within $O(1)$ distance from it. Multiplying by $m$, we get the required conclusion in the stronger form
\[m\log(|u|_v)=\log(|s_v|_v)+O(m),\qquad v\mid\infty.\]
The proof is complete.
\end{proof}

\begin{remark}\label{fundamentaldomainremark}
Applying Lemma~\ref{balancing} with $m=1$ (or its proof if more geometric features are needed), we see that $F_\infty^\times\slash\mo^\times$ has a fundamental domain lying in $\{y\in F_\infty^\times:\text{$|y_v|\asymp |y_\infty|^{1/n}$ for all $v\mid\infty$}\}$. We recall here that $n=[F:\QQ]$, and $|y_v|_v=|y_v|$ when $v$ is a real place, but $|y_v|_v=|y_v|^2$ when $v$ is a complex place. Switching to $m=2$, we see the same for $F_\infty^\times\slash(\mo^\times)^2$, or (mutatis mutandis) for $F_\infty^\times\slash\mo^\times_+$ and $F^\times_+\slash\mo^\times_+$ (cf.\ Section~\ref{section1}).
\end{remark}

By Lemmata~\ref{lemma1} and \ref{lemma2}, any global matrix $g\in\GL_2(\Aa)$ can be decomposed (non-uniquely) as
\begin{equation}\label{eq8}g=\yx k,\end{equation}
where $\syx\in\gP(\Aa)$ and $k\in K^\ast$. Moreover, its height
\[\height(g):=|y|_{\Aa}=\prod_v|y_v|_v\]
and the image of $k$ in the group $K^\ast/\Z(\Aa)K\cong\prod_{\mpr\mid\mn}\{\pm 1\}$
are uniquely determined. By using the ideal class representatives $\theta_1,\dots,\theta_h\in\Aa\fin^\times$ introduced in Section~\ref{section1}, we can refine \eqref{eq8} and obtain a convenient fundamental domain for $X^\ast$ (cf.\ \eqref{Xdef2}).

\begin{lemma}\label{lemma3} Any $g\in\GL_2(\Aa)$ can be decomposed as
\begin{equation}\label{eq12}g=\ts\yx\left(\begin{matrix}\theta_i&\\&1\end{matrix}\right)k,\end{equation}
where $\sts\in\gP(F)$, $\syx\in\gP(F_\infty)$, and $k\in K^\ast$. Moreover, the possible modules $|y|_\infty$ that occur for a given $g\in\GL_2(\Aa)$ are essentially constant, namely $\height(g)\asymp|y|_\infty$.
\end{lemma}

\begin{proof} By \eqref{eq8}, we have a decomposition
\[g=\left(\begin{matrix}\tilde y&\tilde x\\&1\end{matrix}\right)\tilde k,\]
where $\bigl(\begin{smallmatrix}\tilde y&\tilde x\\&1\end{smallmatrix}\bigr)\in\gP(\Aa)$ and $\tilde k\in K^\ast$.
We can write $\tilde y=t\theta_iyy'$ with some $t\in F^\times$, $y\in F_\infty^\times$, $y'\in\hat\mo^\times$, and a unique index $i\in\{1,\dots,h\}$ depending on the fractional ideal $\tilde y\fin\mo$. Here $\hat\mo$ is as in \eqref{hatmonew}. In addition, as $F+F_\infty$ is dense in $\Aa$ (see~\cite[Cor.~2 to Th.~3 in Ch.~IV-2]{We}), we can write $\tilde x=s+t\theta_i(x+x')$ with some $s\in F$, $x\in F_\infty$, and $x'\in\hat\mo$. Therefore,
\[\left(\begin{matrix}\tilde y&\tilde x\\&1\end{matrix}\right)
=\left(\begin{matrix}t\theta_i&s\\&1\end{matrix}\right)\left(\begin{matrix}yy'&x+x'\\&1\end{matrix}\right)
=\ts\left(\begin{matrix}\theta_i&\\&1\end{matrix}\right)\yx\left(\begin{matrix}y'&x'\\&1\end{matrix}\right).\]
On the right hand side, the second and third factors commute, while the fourth factor (call it $k'$) lies in $K^\ast$, hence \eqref{eq12} follows
with $k:=k'\tilde k\in K^\ast$.

The second statement is straightforward by the fact that \eqref{eq12} is an instance of the Iwasawa decomposition \eqref{eq8}. Indeed,
\[\height(g)=|ty\theta_i|_\Aa=|y\theta_i|_\Aa=|y|_\infty|\theta_i|_\Aa\asymp|y|_\infty.\]
Here we used that $|t|_\Aa=1$ due to $t\in F^\times$ (see~\cite[Th.~5 in Ch.~IV-4]{We}).
\end{proof}

In contrast, the module $|y|_\infty$ can vary considerably if we allow any matrix $\gamma\in\GL_2(F)$ in place of $\sts\in\gP(F)$, and it will be useful for us to take $|y|_\infty$ as large as possible in this more general context.

\begin{lemma}\label{lemma4} Any $g\in\GL_2(\Aa)$ can be decomposed as
\begin{equation}\label{eq9}g=\gamma\yx\left(\begin{matrix}\theta_i&\\&1\end{matrix}\right)k,\end{equation}
where $\gamma\in\GL_2(F)$, $\syx\in\gP(F_\infty)$, and $k\in K^\ast$. Moreover, among the possible modules $|y|_\infty$ that occur for a given $g\in\GL_2(\Aa)$, there is a maximal one.
\end{lemma}

\begin{proof} By Lemma~\ref{lemma3}, or alternatively by Lemma~\ref{lemma1} and strong approximation for $\SL_2(\Aa)$, a decomposition of the form \eqref{eq9} certainly exists. Let us now fix any decomposition as in \eqref{eq9}, and consider the alternative decompositions
\[g=\tilde\gamma\yxt\left(\begin{matrix}\theta_j&\\&1\end{matrix}\right)\tilde k,\]
where $\tilde\gamma\in\GL_2(F)$, $\syxt\in\gP(F_\infty)$, and $\tilde k\in K^\ast$. It suffices to show that there are only finitely many possible values of $|\tilde y|_\infty$ that exceed $|y|_\infty$. So we assume that $|\tilde y|_\infty>|y|_\infty$, and rearrange the terms to get
\[\tilde\gamma^{-1}\gamma\yx\left(\begin{matrix}\theta_i&\\&1\end{matrix}\right)=
\yxt\left(\begin{matrix}\theta_j&\\&1\end{matrix}\right)(\tilde k k^{-1}).\]
Then, with the notation $\sabcd:=\tilde\gamma^{-1}\gamma\in\GL_2(F)$ and $k':=\tilde k k^{-1}\in K^\ast$, we have
\[\abcd\yx\left(\begin{matrix}\theta_i&\\&1\end{matrix}\right)=\yxt\left(\begin{matrix}\theta_j&\\&1\end{matrix}\right)k'.\]
Multiplying both sides by a suitable matrix in $\Z(F)$, we can further achieve that the greatest common divisor $c\mo+d\mo$ equals $\theta_m\mo$ for some $m$. Now we calculate the height of both sides, using Lemmata~\ref{lemma1}--\ref{lemma3}. On the right hand side, we get $\asymp|\tilde y|_\infty$ by Lemma~\ref{lemma3}. On the left hand side, the product of the local factors $|ad-bc|_v$ coming from \eqref{eq4} and \eqref{eq6} equals $1$ by $ad-bc\in F^\times$ (cf.~\cite[Th.~5 in Ch.~IV-4]{We}). The product of the other factors coming from \eqref{eq6} is $\asymp_\mn 1$ due to finitely many possibilities for the pair $(\theta_i,\theta_m)$ and the fact that $\syx$ has no finite components. So we can focus on the remaining factors coming from \eqref{eq4}, and we conclude that
\[\frac{|y|_\infty}{\prod_{v\mid\infty}(|c_vy_v|^2+|c_vx_v+d_v|^2)^{[F_v:\RR]}}\asymp_\mn|\tilde y|_\infty.\]
Along these lines, we also see that the fractional ideals $c\mo$ and $d\mo$ together with the denominator on the left hand side determine $|\tilde y|_\infty$ up to $\ll_\mn 1$ choices, so it suffices to show that these quantities only take on finitely many different values.
At any rate, the right hand side exceeds $|y|_\infty$ by assumption, hence we immediately get
\begin{equation}\label{eq10}\prod_{v\mid\infty}(|c_vy_v|^2+|c_vx_v+d_v|^2)^{[F_v:\RR]}\ll_\mn 1.\end{equation}
If $c=0$, then \eqref{eq10} yields $|d|_\infty\ll_\mn 1$, so in this case there are $\ll_\mn 1$ choices for the fractional ideal $d\mo\subseteq\theta_m\mo$ and its norm $|d|_\infty$, whose square is apparently the left hand side of \eqref{eq10}. If $c\neq 0$, then \eqref{eq10} yields $|cy|_\infty\ll_\mn 1$, so in this case there are $\ll_{\mn,y} 1$ choices for the fractional ideal $c\mo\subseteq\theta_m\mo$. Let us fix a nonzero choice for $c\mo$ (including an arbitrary choice of the generator $c$; however, none of the implied constants below will depend on this choice). Dividing \eqref{eq10} by the squared norm $|c|_\infty^2$ (which is $\gg 1$), we get
\begin{equation}\label{eq11}\prod_{v\mid\infty}(|y_v|^2+|x_v+d_v/c_v|^2)^{[F_v:\RR]}\ll_{\mn} 1.\end{equation}
The factors are $\gg_y 1$, hence $x_v+d_v/c_v\ll_{\mn,y}1$ for all $v\mid\infty$. Using also $d/c\in(\theta_m\mo)(c\mo)^{-1}$, we see that $d/c$ as an element of $\Aa$ lies in a fixed compact set, so there are finitely many choices for $d/c$ and consequently for $d$ as well. In the end, the left hand side of \eqref{eq11} takes on finitely many different values, and the same is true of the left hand side of \eqref{eq10}. The proof is complete.
\end{proof}

By Lemma~\ref{lemma4}, any double coset in $X^\ast$ can be represented by a matrix of the form
$\syx\bigl(\begin{smallmatrix}\theta_i&\\&1\end{smallmatrix}\bigr)$, where $\syx\in\gP(F_\infty)$ and $|y|_\infty$ is maximal.
We shall call such matrices \emph{special}.
We can further specify these representatives by observing that, for any unit $t\in\mo^\times$ and any field element $s\in\theta_i\mo$,
\[\ts\yx\left(\begin{matrix}\theta_i&\\&1\end{matrix}\right)=
\left(\begin{matrix}t_\infty y&t_\infty x+s_\infty\\&1\end{matrix}\right)\left(\begin{matrix}\theta_i&\\&1\end{matrix}\right)
\left(\begin{matrix}t\fin&\theta_i^{-1}s\fin\\&1\end{matrix}\right).\]
Indeed, the leftmost matrix lies in $\GL_2(F)$ and the rightmost matrix lies in $K^\ast$, hence we can replace $y$ by $t_\infty y$ and $x$ by $t_\infty x+s_\infty$ without changing $|y|_\infty$ and the double coset represented. In particular, we can restrict $y\in F_\infty^\times$ to a fixed fundamental domain for $F_\infty^\times\slash\mo^\times$ and $x\in F_\infty$ to a fixed fundamental domain for $F_\infty\slash\theta_i\mo$. We can further tweak $y\in F_\infty^\times$ by replacing each component $y_v$ by its absolute value $|y_v|$, thanks to the observation
\begin{equation}\label{rotate}
\left(\begin{matrix}y_v/|y_v|&\\&1\end{matrix}\right)\in K_v^\ast,\qquad v\mid\infty.
\end{equation}
Again, $|y|_\infty$ is invariant under such a replacement. In this way, using appropriate fundamental domains for $F_\infty^\times\slash\mo^\times$ and $F_\infty\slash\theta_i\mo$, we obtain a set of representatives $\FF(\mn)\subset\GL_2(\Aa)$ for $X^\ast$ consisting of special matrices $\syx\bigl(\begin{smallmatrix}\theta_i&\\&1\end{smallmatrix}\bigr)$ such that the components of $y\in F_\infty^\times$ and $x\in F_\infty$ satisfy (cf.\ Remark~\ref{fundamentaldomainremark}):
\begin{equation}\label{fund}0<y_v\asymp|y|_\infty^{1/n}\quad\text{and}\quad x_v\ll 1,\qquad v\mid\infty.\end{equation}
By possibly shrinking $\FF(\mn)$ further, we could get a true fundamental domain for $X^\ast$ (i.e.\ a nice subset of $\GL_2(\Aa)$ representing each double coset in $X^\ast$ exactly once), but the above construction is sufficient for our purposes.

Now we consider the $2n$-dimensional $\RR$-algebra
\begin{equation}\label{mm}
\MM:=\prod_{\text{$v$ real}}\CC\prod_{\text{$v$ complex}}\HH,
\end{equation}
which becomes a Euclidean space if we postulate that the standard basis (the union of the bases $\{1,i\}$ and $\{1,i,j,k\}$ of the various factors $\CC$ and $\HH$ embedded into $\MM$ as subspaces) is orthonormal. Within this space, we consider the generalized upper half-space (cf.\ \eqref{h2} and \eqref{h3})
\begin{equation}\label{hb}
\Hb:=\prod_{\text{$v$ real}}\Hh^2\prod_{\text{$v$ complex}}\Hh^3,
\end{equation}
and to each special matrix $\syx\bigl(\begin{smallmatrix}\theta_i&\\&1\end{smallmatrix}\bigr)\in\FF(\mn)$, we associate the point (cf.\ \eqref{fund})
\begin{equation}\label{associateP}
P=P(x,y):=\prod_{\text{$v$ real}}\{x_v+y_vi\}\times\prod_{\text{$v$ complex}}\{x_v+y_vj\}\in\Hb
\end{equation}
and the $2n$-dimensional $\ZZ$-lattice
\begin{equation}\label{LambdaP}
\Lambda(P):=\{cP+d:c,d \in \mo\}\subset\MM.
\end{equation}
Note that a given point $P\in\Hb$ corresponds to at most $h$ elements of $\FF(\mn)$ in the above sense.

The next lemma, a generalization of \cite[Lemma~2]{BHM}, shows in terms of these lattices that $\FF(\mn)$ has good diophantine properties.

\begin{lemma}\label{lemma5} Let $\syx\bigl(\begin{smallmatrix}\theta_i&\\&1\end{smallmatrix}\bigr)\in\FF(\mn)$ with corresponding point $P=P(x,y)\in\Hb$. Then the lattice $\Lambda(P)\subset\MM$ and its successive minima $m_1\leq m_2\leq\cdots\leq m_{2n}$ satisfy:
\begin{itemize}
\item[(a)] $m_1m_2\cdots m_{2n}\asymp |y|_\infty$;
\item[(b)] $|y|_\infty\gg m_1^{2n}\gg (\N\mn)^{-1}$;
\item[(c)] $m_1\asymp m_2\asymp\dots\asymp m_n$ and $m_{n+1}\asymp m_{n+2}\asymp\dots\asymp m_{2n}$;
\item[(d)] in any ball of radius $R$ the number of lattice points is $\ll 1+R^n(\N\mn)^{1/2}+R^{2n}|y|_\infty^{-1}$.
\end{itemize}
\end{lemma}

\begin{proof} (a) Let $\{u_1,\dots,u_n\}$ be a $\ZZ$-basis of $\mo$. Then $\{u_1,\dots,u_n,u_1P,\dots,u_nP\}$ is a $\ZZ$-basis of $\Lambda(P)$, and calculating its exterior product coupled with Minkowski's theorem~\cite[Th.~3 on p.~124]{GL} shows that
\[m_1m_2\cdots m_{2n}\asymp\vol(\MM\slash\Lambda(P))=|y|_\infty\vol(F_\infty\slash\mo)^2\asymp|y|_\infty.\]

(b) The bound $|y|_\infty\gg m_1^{2n}$ is clear from part (a), so it suffices to show $m_1^{2n}\gg (\N\mn)^{-1}$. Let $(c,d)\in\mo^2$ be any pair distinct from $(0,0)$. We choose any pair $(a,b)\in F^2$ such that $ad-bc\neq 0$, and we consider the translated matrix $\sabcd\syx\bigl(\begin{smallmatrix}\theta_i&\\&1\end{smallmatrix}\bigr)\in\GL_2(\Aa)$. According to Lemma~\ref{lemma3}, we have a decomposition
\[\abcd\yx\left(\begin{matrix}\theta_i&\\&1\end{matrix}\right)=
\ts\yxt\left(\begin{matrix}\theta_j&\\&1\end{matrix}\right)k,\]
where $\sts\in\gP(F)$, $\syxt\in\gP(F_\infty)$, and $k\in K^\ast$.
The key observation here is that $\syx\bigl(\begin{smallmatrix}\theta_i&\\&1\end{smallmatrix}\bigr)$ and $\bigl(\begin{smallmatrix}\tilde y&\tilde x\\&1\end{smallmatrix}\bigr)\bigl(\begin{smallmatrix}\theta_j&\\&1\end{smallmatrix}\bigr)$ represent the same double coset in
$X^\ast$, hence $|\tilde y|_\infty\leq|y|_\infty$ by the very assumption $\syx\bigl(\begin{smallmatrix}\theta_i&\\&1\end{smallmatrix}\bigr)\in\FF(\mn)$.
Now we proceed somewhat similarly as in the proof of Lemma~\ref{lemma4}. That is, we calculate the height of both sides, using Lemmata~\ref{lemma1}--\ref{lemma3}. On the right hand side, we get $\asymp|\tilde y|_\infty$ by Lemma~\ref{lemma3}. On the left hand side, the product of the local factors $|ad-bc|_v$ coming from \eqref{eq4} and \eqref{eq6} equals $1$ by $ad-bc\in F^\times$ (cf.~\cite[Th.~5 in Ch.~IV-4]{We}). The product of the other factors coming from \eqref{eq6} is $\gg(\N\mn)^{-1}$ due to the facts that $c,d\in\mo$ are integers and $\syx$ has no finite components. So, we can focus on the remaining factors coming from \eqref{eq4}, and we conclude that
\[\frac{|y|_\infty(\N\mn)^{-1}}{\prod_{v\mid\infty}(|c_vy_v|^2+|c_vx_v+d_v|^2)^{[F_v:\RR]}}\ll|\tilde y|_\infty\leq |y|_\infty.\]
Comparing the two sides,
\[\prod_{v\mid\infty}(|c_vy_v|^2+|c_vx_v+d_v|^2)^{[F_v:\RR]}\gg(\N\mn)^{-1}.\]
Regarding the left hand side as a product of $\sum_{v\mid\infty}[F_v:\RR]=n$ factors, the inequality between the arithmetic and geometric mean readily yields
\[\biggl(\sum_{v\mid\infty}(|c_vy_v|^2+|c_vx_v+d_v|^2)\biggr)^n\gg(\N\mn)^{-1}.\]
The sum on the left hand side equals the squared Euclidean norm $\|cP+d\|^2$, hence we conclude
\[\|cP+d\|^{2n}\gg(\N\mn)^{-1}.\]
This holds for all pairs $(c,d)\in\mo^2$ distinct from $(0,0)$, so $m_1^{2n}\gg (\N\mn)^{-1}$ as stated.

(c) It suffices to show that $m_n\ll m_1$ and $m_{2n}\ll m_{n+1}$, and for this we utilize the left $\mo$-invariance of $\Lambda(P)$.
To prove the first relation, we take a nonzero lattice point $Q\in\Lambda(P)$ of Euclidean norm at most $m_1$. Then $\mo Q$ is an $n$-dimensional sublattice of $\Lambda(P)$ with $\ZZ$-basis $\{u_1Q,\dots,u_nQ\}$, where $\{u_1,\dots,u_n\}$ is a $\ZZ$-basis of $\mo$ as before. Therefore,
\[m_n\leq\max\{\|u_1Q\|,\dots,\|u_nQ\|\}\ll\|Q\|\leq m_1.\]
To prove the second relation, we take $\ZZ$-independent lattice points $Q_1,\dots,Q_{n+1}\in\Lambda(P)$ of Euclidean norms at most $m_{n+1}$.
At least two of these, say $Q$ and $Q'$, must be $\mo$-independent, and then $\mo Q+\mo Q'$ is a $2n$-dimensional sublattice of $\Lambda(P)$ with $\ZZ$-basis $\{u_1Q,\dots,u_nQ,u_1Q',\dots,u_nQ'\}$. Therefore,
\[m_{2n}\leq\max\{\|u_1Q\|,\dots,\|u_nQ\|,\|u_1Q'\|,\dots,\|u_nQ'\|\}\ll\max\{\|Q\|,\|Q'\|\}\leq m_{n+1}.\]

(d) Let $B\subset\MM$ be any ball of radius $R>0$ (with respect to the Euclidean norm). Then, by a lattice
point counting argument using successive minima (see e.g.\ \cite[Lemma~1]{BHM}),
\[\#\bigl(\Lambda(P)\cap B\bigr)\ll 1 + \frac{R}{m_1} + \frac{R^2}{m_1m_2} + \dots + \frac{R^{2n}}{m_1 m_2\cdots m_{2n}}.\]
Denoting by $t_k$ the degree $k$ term, we can estimate by part (c)
\[t_k\ll\begin{cases}
t_n^{k/n}\leq\max(1,t_n)&\qquad\text{when $0\leq k\leq n$},\\
t_n\left(t_{2n}/t_n\right)^{(k-n)/n}\leq\max(t_n,t_{2n})&\qquad\text{when $n\leq k\leq 2n$},
\end{cases}\]
hence our bound simplifies to
\[\#\bigl(\Lambda(P)\cap B\bigr)\ll 1 + \frac{R^n}{m_1 m_2\cdots m_n} + \frac{R^{2n}}{m_1 m_2\cdots m_{2n}}.\]
The first denominator is $\gg(\N\mn)^{-1/2}$ by part (b), and the second denominator is $\asymp |y|_\infty$ by part (a), so the stated bound follows.
\end{proof}

\section{Counting the field elements}\label{sec-counting}

In this section, we collect some convenient counting bounds in the number field $F$. These bounds are rather standard but indispensable for our goals, and we present them in a consistent adelic language.

\begin{lemma}\label{lemma6} Let $t\in\Aa^\times$ be an arbitrary idele. Then
\begin{itemize}
\item[(a)] $\#\{x\in F^\times:\text{$|x|_v\leq |t_v|_v$ for all places $v$}\}\ll |t|_\Aa$;
\item[(b)] $\#\{x\in F^\times:\text{$|x|_v\leq |t_v|_v$ for all $v\mid\infty$ and $|x|_v=|t_v|_v$ for all $v\nmid\infty$}\}\ll_\eps |t|_\Aa^\eps$ for every $\eps>0$.
\end{itemize}
\end{lemma}

\begin{proof} (a) Let $S$ be the set in part (a). If $S$ is empty, we are done. Otherwise, we fix any element $\tilde x\in S$. By assumption, the idele $s:=t/\tilde x\in\Aa^\times$ satisfies $1\leq|s_v|_v$ for all places $v$. Moreover, any $x\in S$ is determined by the field element $u:=x/\tilde x\in F^\times$, which in turn satisfies $|u|_v\leq|s_v|_v$ for all places $v$. Now we consider a fundamental parallelotope $P$ for $F_\infty\slash\mo$. Then, $\bm{P}:=P\times\hat\mo$ (cf.\ \eqref{hatmonew})
is a fundamental domain for $\Aa/F$ (see~\cite[Prop.~6 in Ch.~V-4]{We}). The translates $u+\bm{P}$ (with $u$ as before) have pairwise disjoint interiors and, for a suitable idele $y\in\Aa^\times$ depending only on $P$, they all lie in the adelic box
\[\bm{B}:=\{z\in\Aa:\text{$|z_v|_v\leq |y_vs_v|_v$ for all places $v$}\}.\]
Here we used that $|s_v|_v\geq 1$ for all places $v$, and we remark that we can take $y_v=1$ at the non-archimedean places $v$.
An adelic volume calculation now gives
\[\#S\ll_P\#S\vol(\bm{P})\leq\vol(\bm{B})=|ys|_\Aa\ll_P|s|_\Aa=|t|_\Aa,\]
and the bound (a) follows.

(b) Let $S$ be the set in part (b). If $S$ is empty, we are done. Otherwise, we fix any element $\tilde x\in S$. By assumption, the idele $s:=t/\tilde x\in\Aa^\times$ satisfies $1\leq|s_v|_v$ for all $v\mid\infty$ and $1=|s_v|_v$ for all $v\nmid\infty$. Moreover, any $x\in S$ is determined by the field element $u:=x/\tilde x\in F^\times$, which in turn satisfies $|u|_v\leq|s_v|_v$ for all $v\mid\infty$ and $|u|_v=|s_v|_v=1$ for all $v\nmid\infty$.
The second relation implies that $u\in\mo^\times$ is a unit, while the first relation implies the lower bounds
\[|u|_v=\prod_{\substack{w\mid\infty\\w\neq v}}|u|_w^{-1}\geq \prod_{\substack{w\mid\infty\\w\neq v}}|s_w|_w^{-1}=\abs{s_v}_v\cdot|s|_\infty^{-1},\qquad v\mid\infty.\]
In short,
\[\log(|s_v|_v)-\log(|s|_\infty) \leq \log(|u|_v) \leq \log(|s_v|_v),\qquad v\mid\infty.\]
With the notations \eqref{logarithmicvector}--\eqref{hyperplane}, the vector $l(u)$ lies in a fixed ball of radius $\ll\log(|s|_\infty)$ intersected with a fixed lattice in $W$ by Dirichlet's unit theorem (cf.~\cite[p.~93]{We}), hence the usual volume argument yields
\[\#l(u)\ll 1+(\log|s|_\infty)^{\dim W}\ll_\eps|s|_\infty^\eps=|s|_\Aa^\eps=|t|_\Aa^\eps.\]
However, $u$ is determined by $l(u)$ up to a root of unity (cf.~\cite[Th.~8 in Ch.~IV-4]{We}), so the bound (b) follows.
\end{proof}

We shall use Lemma~\ref{lemma6} in the following classical formulation.

\begin{corollary}\label{corollary1} Let $y\in F_\infty^\times$ be an infinite idele, and let $\mm\subseteq F$ be a nonzero fractional ideal. Then
\begin{itemize}
\item[(a)] $\#\{x\in F^\times:\text{$|x|_v\leq |y_v|_v$ for all $v\mid\infty$ and $x\mo\subseteq\mm$}\}\ll\ \ |y|_\infty/\N\mm$;
\item[(b)] $\#\{x\in F^\times:\text{$|x|_v\leq |y_v|_v$ for all $v\mid\infty$ and $x\mo=\mm$}\}\ll_\eps (|y|_\infty/\N\mm)^\eps$ for every $\eps>0$.
\end{itemize}
\end{corollary}

\begin{proof} We take any finite idele $z\in\Aa\fin^\times$ such that $\mm=z\mo$, and we write $t:=yz\in\Aa^\times$. Clearly,
$|y_v|_v=|t_v|_v$ for all $v\mid\infty$, while
$x\mo\subseteq\mm$ (resp.\ $x\mo=\mm$) is equivalent to $|x|_v\leq |t_v|_v$ (resp.\ $|x|_v=|t_v|_v$) for all $v\nmid\infty$. Moreover,
\[|t|_\Aa=|y|_\infty\prod_{v\nmid\infty}|z_v|_v=|y|_\infty/\N(z\mo)=|y|_\infty/\N\mm.\]
Hence the bounds (a) and (b) are immediate from the corresponding parts of Lemma~\ref{lemma6}.
\end{proof}

\begin{remark} It is a valuable feature of Lemma~\ref{lemma6} that the bounds only depend on the module $|t|_\Aa$. This type of result proved to be useful in earlier investigations of the sup-norm problem; for instance, \cite[(4.1)--(4.2)]{BM} constitute a special case of Lemma~\ref{lemma6} and Corollary~\ref{corollary1}.
\end{remark}

\section{Realness rigidity of number fields}\label{rigidity}

This section is devoted to the proof of the following result which will be used in the proof of Theorem~\ref{thm3}, whose germ is the argument at the beginning of Subsection~11.2 in \cite{BHM}.

\begin{lemma}\label{aplusd}
Let $F_0$ be the maximal totally real subfield of $F$, and put $m:=[F:F_0]$. Let $\xi\in F$ be such that $F=F_0(\xi)$. Suppose that $\xi\mo=\mk/\ml$, where $\mk,\ml\subseteq\mo$ are integral ideals. Suppose we have the bounds
\begin{equation}\label{xiwbounds2}
|\xi_v|\leq A\quad\text{and}\quad|\im\xi_v|\leq A\sqrt{\delta_v},\qquad v\mid\infty,
\end{equation}
where $A\geq 1$ and $\delta_v>0$ are arbitrary. Then
\begin{equation}\label{conclusion}
1\leq\bigl((2A)^n(\N\ml)^2\bigr)^{2(m-1)}|\delta|_\CC.
\end{equation}
\end{lemma}

\begin{remark} The assumption $F=F_0(\xi)$ serves convenience. Without it, \eqref{conclusion} still holds with $m$ being the degree of $\xi$ over $F_0$, and with $|\delta|_\CC$ replaced by the relevant subproduct restricted to the places of $F$ that lie over the \emph{complex} places of $F_0(\xi)$. In this more general formulation, the result says that if $\xi\in F$ is ``close to being totally real'', then it is totally real. Thus the result expresses a certain ``rigidity of realness'' in number fields.
\end{remark}

\begin{remark}
If $F=\QQ(i)$, and $\xi=\kappa/\lambda$ with $\kappa,\lambda\in\ZZ[i]$ satisfying $|\kappa|\ll|\lambda|\asymp\Lambda$, then the above statement says that if $|\im\xi|\ll\Lambda^{-2}$ with a sufficiently small implied constant, then $\xi\in\QQ$. This is precisely the crucial input in \cite[Subsection~11.2]{BHM}, which in this special case is completely elementary.
\end{remark}

\begin{proof} We shall assume that $m\geq 2$, because \eqref{conclusion} is trivially true for $m=1$ (in this case the right hand side equals $1$).
Let $\tilde F$ be the Galois closure of $F$ with ring of integers $\tilde\mo$, and consider the Galois groups
\[G:=\Gal(\tilde F/\QQ),\qquad H:=\Gal(\tilde F/F_0),\qquad I:=\Gal(\tilde F/F).\]
Note that $I\leq H\leq G$, and
\begin{equation}\label{Galoisindex}
[G:I]=[F:\QQ]=n,\qquad [G:H]=[F_0:\QQ]=n/m,\qquad [H:I]=[F:F_0]=m.
\end{equation}
Indeed, the elements of $I\backslash G$ correspond to the $n$ embeddings $F\hookrightarrow\tilde F$, the elements of $H\backslash G$ correspond to the $n/m$ embeddings $F_0\hookrightarrow\tilde F$, and the elements of $I\backslash H$ correspond to the $m$ conjugates of $\xi$ over $F_0$ via the right action $\sigma\mapsto \xi^\sigma$ of $H$ on $\xi$. In particular,
\[p(X):=\prod_{\alpha\in I\backslash H}\bigl(X-\xi^\alpha\bigr)\]
is the minimal polynomial of $\xi$ over $F_0$ (a separable polynomial of degree $m$), and
\[\Delta:=\prod_{\substack{{\alpha,\beta\in I\backslash H}\\{\alpha\neq\beta}}}\bigl(\xi^\alpha-\xi^\beta\bigr)\]
is the discriminant of $p(X)$. Note that right multiplication by any $\sigma\in H$ permutes the elements of $I\backslash H$, hence $H$ fixes $\Delta$, and so $\Delta\in F_0^\times$. The key to deriving \eqref{conclusion} is to examine the absolute norm $\bigl|N_{F_0/\QQ}(\Delta)\bigr|\in\QQ^\times$, which is a \emph{positive rational} number.

Fixing an embedding $\tilde F\hookrightarrow\CC$, we get that (cf.\ \cite[Cor.~3 to Prop.~4 in Ch.~III-3]{We})
\[\bigl|N_{F_0/\QQ}(\Delta)\bigr|=\prod_{\gamma\in H\backslash G}\bigl|\Delta^\gamma\bigr|=
\prod_{\substack{{\alpha,\beta\in I\backslash H}\\{\alpha\neq\beta}}}
\prod_{\gamma\in H\backslash G}\bigl|\xi^{\alpha\gamma}-\xi^{\beta\gamma}\bigr|.\]
The products $\alpha\gamma$ and $\beta\gamma$ run through distinct representatives of $I\backslash G$ such that $H\alpha\gamma=H\beta\gamma$ (namely $H\alpha\gamma$ and $H\beta\gamma$ are equal to $H\gamma$), hence they correspond to distinct embeddings $\sigma,\tau:F\hookrightarrow\CC$
which agree on~$F_0$. We infer
\begin{equation}\label{norm1}
\bigl|N_{F_0/\QQ}(\Delta)\bigr|=
\prod_{\substack{{\sigma,\tau:F\hookrightarrow\CC}\\{\sigma\neq\tau}\\{\sigma\restriction_{F_0}=\tau\restriction_{F_0}}}}\bigl|\xi^\sigma-\xi^\tau\bigr|.
\end{equation}
On the right hand side, there are $(n/m)m(m-1)=n(m-1)$ factors by \eqref{Galoisindex}, all of which are at most $2A$ by the first part of \eqref{xiwbounds2}. By definition, each complex place $v$ of $F$ can be identified with an \emph{unordered} complex conjugate pair $\{\sigma,\tau\}$ of distinct embeddings $\sigma,\tau:F\hookrightarrow\CC$. For such a pair $\{\sigma,\tau\}$, we have $\{\xi^\sigma,\xi^\tau\}=\{\xi_v,\ov{\xi_v}\}$, while the condition $\sigma\rest_{F_0}=\tau\rest_{F_0}$ is automatic as $F_0$ is totally real (the restricted embedding $F_0\hookrightarrow\RR$ corresponds to the real place of $F_0$ below $v$). Hence each complex place $v$ of $F$ contributes \emph{two} factors $|\xi_v-\ov{\xi_v}|$ to the product in \eqref{norm1}, and these are at most $2A\sqrt{\delta_v}$ by the second part of \eqref{xiwbounds2}. Bounding the other factors $|\xi^\sigma-\xi^\tau\bigr|$ in \eqref{norm1} by $2A$ as remarked initially, we obtain in the end
\begin{equation}\label{norm2}
\bigl|N_{F_0/\QQ}(\Delta)\bigr|\leq (2A)^{n(m-1)}|\delta|_\CC^{1/2}.
\end{equation}

On the other hand, denoting $\tilde\mk:=\mk\tilde\mo$ and $\tilde\ml:=\ml\tilde\mo$, the nonzero fractional ideal in $\tilde F$ given by
\[\Delta\prod_{\substack{{\alpha,\beta\in I\backslash H}\\{\alpha\neq\beta}}}\bigl(\tilde\ml^\alpha\tilde\ml^\beta\bigr)=
\prod_{\substack{{\alpha,\beta\in I\backslash H}\\{\alpha\neq\beta}}}\bigl(\xi^\alpha-\xi^\beta\bigr)\bigl(\tilde\ml^\alpha\tilde\ml^\beta\bigr)
\subseteq
\prod_{\substack{{\alpha,\beta\in I\backslash H}\\{\alpha\neq\beta}}}\bigl(\tilde\mk^\alpha\tilde\ml^\beta+\tilde\ml^\alpha\tilde\mk^\beta\bigr)
\subseteq\tilde\mo\]
is integral, hence its norm is at least $1$. The norm of each factor $\tilde\ml^\alpha\tilde\ml^\beta\subseteq\tilde\mo$ on the left hand side equals
\[\N\bigl(\tilde\ml^\alpha\bigr)\N\bigl(\tilde\ml^\beta\bigr)=\bigl(\N\tilde\ml\bigr)^2=(\N\ml)^{2[\tilde F:F]},\]
hence we get
\[1\leq\bigl|N_{\tilde F/\QQ}(\Delta)\bigr|(\N\ml)^{2m(m-1)[\tilde F:F]}=\bigl|N_{\tilde F/\QQ}(\Delta)\bigr|(\N\ml)^{2(m-1)[\tilde F:F_0]}.\]
Raising the two sides to power $[\tilde F:F_0]^{-1}$, and combining with \eqref{norm2}, we conclude
\[1\leq\bigl|N_{F_0/\QQ}(\Delta)\bigr|(\N\ml)^{2(m-1)}\leq\bigl((2A)^n(\N\ml)^2\bigr)^{(m-1)}|\delta|_\CC^{1/2}.\]
Finally, by squaring the two sides, \eqref{conclusion} follows.
\end{proof}

\section{The Fourier bound}\label{fourierbound}

We denote by $t_v\in\RR\cup\RR i$ the \emph{spectral parameter} of $\phi$ at the archimedean place $v$, so that
the Laplace eigenvalue of $\phi$ at $v$ equals
\begin{equation}\label{tparameter}
\lambda_v=\begin{cases}
1/4+t_v^2&\text{when $v$ is real},\\
1+4t_v^2&\text{when $v$ is complex}.
\end{cases}\end{equation}
The archimedean Ramanujan conjecture states that $t_v\in\RR$, while it is known \cite[Th.~1]{BB} that
\[t_v\in\RR\cup[-7/64,7/64]i.\]
For convenience, we also introduce the tuple $T:=(T_v)_{v\mid\infty}$, where
\begin{equation}\label{Tparameter}
T_v:=\max(1/2,|t_v|)\asymp\sqrt{\lambda_v}.
\end{equation}
The aim of this section is to prove the following bound.

\begin{lemma}\label{lemma7} Let $\syx\in P(F_{\infty})$ and $i\in\{1,\dots,h\}$. Then for any $\eps>0$ we have
\begin{equation}\label{Fbound}
\phi\left(\yx\begin{pmatrix}\theta_i \\ & 1\end{pmatrix}\right)
\ll_{\eps}\left(|T|_{\infty}^{1/6}+|T/y|_{\infty}^{1/2}\right)^{1+\eps}(\N\mn)^{\eps}.
\end{equation}
\end{lemma}

\begin{proof}
Our starting point is the Fourier--Whittaker decomposition (cf.\ \cite[Subsection~2.3.2]{Ma})
\begin{equation}\label{fourier}
\phi\left(\yx\begin{pmatrix}\theta_i \\ & 1\end{pmatrix}\right)=
\rho(\mo)\sum_{0\neq n\in \theta_i^{-1}\md^{-1}} \frac{\lambda(n\theta_i \md)}{\sqrt{\N(n\theta_i \md)}} W(ny) \psi(nx).
\end{equation}
Here, $\md$ is the different ideal of $F$, $\rho(\mo)$ is a nonzero constant (``the first Fourier coefficient''), $\lambda(\mm)$ are the Hecke eigenvalues introduced in \eqref{lambdadef},
\[W(\xi):=\prod_{v\mid \infty} W_{v}(\xi_v),\qquad
W_{v}(\xi_v):=\begin{cases} \frac{|\xi_v|_v^{1/2} K_{it_v}(2\pi|\xi_v|)}{|\Gamma(1/2+it_v)\Gamma(1/2-it_v)|^{1/2}}, & \text{$v$ real}, \\ \frac{|\xi_v|_v^{1/2} K_{2it_v}(4\pi|\xi_v|)}{|\Gamma(1+2it_v)\Gamma(1-2it_v)|^{1/2}}, & \text{$v$ complex}, \end{cases}\]
is an $L^2$-normalized (weight zero) Whittaker function for an appropriate Haar measure on $F_\infty^\times$, and
\[\psi(\xi):=\prod_{\text{$v$ real}} e^{2\pi i \xi_v} \prod_{\text{$v$ complex}} e^{2\pi i(\xi_v+\overline{\xi_v})}\]
is the corresponding (appropriately normalized) additive character on $F_\infty$.

The local Whittaker function $W_v$ satisfies (even for $T_v=1/2$)
\begin{equation}\label{local_bounds}
W_{v}(\xi_v) \ll \begin{cases} \min(T_v^{1/6},T_v^{1/4}|2\pi|\xi_v|-T_v|^{-1/4}), & |\xi_v|\leq T_v, \\ e^{-\pi|\xi_v|}, & |\xi_v|> T_v, \end{cases}
\end{equation}
as follows from \cite[p.~679]{BHo} and \cite[Prop.~9]{HM}, upon noting that $|\im(t_v)|\leq 1/2$ since $\lambda_v\geq 0$ (cf.\ \eqref{tparameter}).
We fix an $\eps>0$ for the rest of this section. Then, by \cite[Prop.~2.2, (2.16), (3.5), Prop.~3.2]{Ma} and the fact that $\phi$ is a newform, we also have
\begin{equation}\label{negligible_rho}
|T|_\infty^{-\eps}(\N\mn)^{-\eps} \ll_{\eps} |\rho(\mo)|\ll_{\eps} |T|_\infty^{\eps}(\N\mn)^{\eps}.
\end{equation}

Multiplying $y\in F_\infty^\times$ and $x\in F_\infty$ by a given unit from $\mo^\times$ leaves the bound \eqref{Fbound} unchanged, hence we may assume that
\begin{equation}\label{yassumption}
|y_v|_v\asymp|T_v|_v^{\log|y|_\infty/\log|T|_\infty},\qquad v\mid\infty.
\end{equation}
To see this, we apply Lemma~\ref{balancing} with $m=1$ and observe that the product of the right hand side over all $v\mid\infty$ equals $|y|_\infty$. We may further assume that $y_v>0$ for all $v\mid\infty$, thanks to \eqref{rotate}.

We estimate the sum on the right hand side of \eqref{fourier} by the Cauchy--Schwarz inequality as
\begin{equation}\label{cs}
\leq\left(\sum_{0\neq n\in \theta_i^{-1}\md^{-1}} \frac{|\lambda(n\theta_i \md)|^2}{|n|_\infty \prod_{v\mid \infty}(1+|n_vy_v|_v)^{2\eps}}\right)^{1/2}
\left(\sum_{0\neq n\in \theta_i^{-1}\md^{-1}} |W(ny)|^2\prod_{v\mid \infty}(1+|n_vy_v|_v)^{2\eps} \right)^{1/2},
\end{equation}
and we claim that the first factor satisfies
\begin{equation}\label{cs_hecke}
\sum_{0\neq n\in \theta_i^{-1}\md^{-1}} \frac{|\lambda(n\theta_i \md)|^2}{|n|_\infty \prod_{v\mid \infty}(1+|n_vy_v|_v)^{2\eps}} \ll_{\eps} |T/y|_\infty^{\eps}(\N\mn)^{\eps}.
\end{equation}
To see this, we fix a nonzero principal fractional ideal $\mm\subseteq \theta_i^{-1}\md^{-1}$ along with a nonnegative integral vector
$\bl=(l_v)\in \NN^{\{v\mid \infty\}}$, and we estimate first the contribution of the $n$'s lying in
\[\{n\in\theta_i^{-1}\md^{-1}: \text{$2^{l_v}\leq 1+|n_vy_v|_v<2^{l_v+1}$ for all $v\mid \infty$ and $n\mo=\mm$}\}.\]
By part (b) of Corollary~\ref{corollary1}, this contribution is
\[\ll_{\eps} \left(\frac{|\lambda(\mm\theta_i\md)|^2}{\N\mm}\prod_{v\mid\infty}2^{-2\eps l_v}\right)\times \left((\N\mm)^{-\eps} |y|_{\infty}^{-\eps}\prod_{v\mid\infty}2^{\eps l_v}\right).\]
Summing over all pairs $(\mm,\bl)$, we infer that the left hand side of \eqref{cs_hecke} is indeed
\[\ll_{\eps} |y|_{\infty}^{-\eps}\sum_{0\neq \mm\subseteq \theta_i^{-1}\md^{-1}} \frac{|\lambda(\mm\theta_i\md)|^2}{(\N\mm)^{1+\eps}}\prod_{v\mid \infty} \left(\sum_{l_v=1}^{\infty} 2^{-\eps l_v}\right)
\ll_{\eps} |T/y|_\infty^{\eps}(\N\mn)^{\eps},\]
since the $\mm$-sum is $\ll_{\eps} |T|_\infty^{\eps}(\N\mn)^{\eps}$ by Iwaniec's trick \cite[pp.~72--73]{Iw2}.

To estimate the second factor in \eqref{cs}, we decompose $F_\infty^\times$ into generalized boxes
\[F_\infty^\times=\bigcup_\bk I(\bk),\qquad I(\bk):=\prod_{v\mid\infty} I_v(k_v),\]
where $\bk=(k_v)\in \ZZ^{\{v\mid\infty\}}$ runs through integral vectors, and the local components are defined as
\begin{align*}
I_v(k_v)&:=\left\{\xi_v\in F_v^\times: k_v\frac{T_v}{y_v} < |\xi_v| \leq (k_v+1)\frac{T_v}{y_v} \right\},&&k_v\geq 1;\\
I_v(k_v)&:=\left\{\xi_v\in F_v^\times: |\xi_v| \leq \frac{T_v}{y_v}\text{ and } -k_v \leq \left||\xi_v|-\frac{T_v}{2\pi y_v}\right| < -k_v+1 \right\},&&k_v\leq 0.
\end{align*}
It is easy to see that $I(\bk)=\emptyset$ unless $k_v\geq -\lfloor T_v/y_v\rfloor$ for each $v\mid\infty$.
Correspondingly, we have
\begin{equation}\label{cs_whittaker}
\sum_{0\neq n\in\theta_i^{-1}\md^{-1}} |W(ny)|^2\prod_{v\mid \infty}(1+|n_vy_v|_v)^{2\eps} =
\sum_{\substack{\bk=(k_v)\\k_v\geq-\lfloor T_v/y_v\rfloor}} \sum_{n\in\theta_i^{-1}\md^{-1}\cap I(\bk)} |W(ny)|^2\prod_{v\mid \infty}(1+|n_vy_v|_v)^{2\eps}.
\end{equation}
Then the inner sum (the contribution of a given $\bk$) is
\begin{displaymath}
\leq\# (\theta_i^{-1}\md^{-1}\cap I(\bk)) \times \sup_{n\in\theta_i^{-1}\md^{-1}\cap I(\bk)} |W(ny)|^2\prod_{v\mid \infty}(1+|n_vy_v|_v)^{2\eps}.
\end{displaymath}

We shall estimate both factors on the right hand side by a product of local factors over $v\mid\infty$. For estimating the lattice point count, we consider a fundamental parallelotope $P_i$ for the fractional ideal $\theta_i^{-1}\md^{-1}$. We can and we shall assume that $P_i$ contains the origin and has diameter $D_i\ll 1$. Then we observe that
\begin{displaymath}
\bigcup_{n\in\theta_i^{-1}\md^{-1}\cap I(\bk)}(n+P_i) \subseteq J(\bk),\qquad J(\bk):=\prod_{v\mid \infty} J_v(k_v),
\end{displaymath}
where the union on the left hand side is disjoint, and the local components are defined as
\begin{align*}
J_v(k_v)&:=\left\{\xi_v\in F_v: k_v\frac{T_v}{y_v}-D_i < |\xi_v| \leq (k_v+1)\frac{T_v}{y_v}+D_i \right\},&&k_v\geq 1;\\
J_v(k_v)&:=\left\{\xi_v\in F_v: -k_v-D_i \leq \left||\xi_v|-\frac{T_v}{2\pi y_v}\right| < -k_v+1+D_i \right\},&&k_v\leq 0.
\end{align*}
Using $\vol(P_i)\gg 1$ and $D_i\ll 1$, it is now clear that
\begin{displaymath}
\# (\theta_i^{-1}\md^{-1}\cap I(\bk)) \ll \vol(J(\bk))=\prod_{v\mid\infty}\vol(J_v(k_v))\ll\prod_{v\mid \infty} f_v(k_v),
\end{displaymath}
where
\begin{align*}
f_v(k_v)&:=1+T_v/y_v,&&\text{$v$ real and $k_v\geq 1$};\\
f_v(k_v)&:=1,&&\text{$v$ real and $k_v\leq 0$};\\
f_v(k_v)&:=k_v(1+T_v/y_v)^2,&&\text{$v$ complex and $k_v\geq 1$};\\
f_v(k_v)&:=1+T_v/y_v,&&\text{$v$ complex and $k_v\leq 0$}.
\end{align*}
By \eqref{local_bounds}, we also have
\begin{displaymath}
\sup_{n\in\theta_i^{-1}\md^{-1}\cap I(\bk)} |W(ny)|^2\prod_{v\mid \infty}(1+|n_vy_v|_v)^{2\eps} \ll \prod_{v\mid\infty} g_v(k_v),
\end{displaymath}
where
\begin{align*}
g_v(k_v)&:= |k_vT_v|_v^{2\eps}e^{-2\pi k_vT_v},&&k_v\geq 1;\\
g_v(k_v)&:= |T_v|_v^{2\eps} \min(T_v^{1/3},T_v^{1/2}|k_vy_v|^{-1/2}),&&k_v\leq 0.
\end{align*}

By distributivity, we infer that the right hand side of \eqref{cs_whittaker} is
\begin{equation}\label{cs_whittaker2}
\ll \sum_{\substack{\bk=(k_v)\\k_v\geq-\lfloor T_v/y_v\rfloor}}\left(\prod_{v\mid\infty}f_v(k_v)g_v(k_v)\right) = \prod_{v\mid\infty} \left(\sum_{k_v=-\lfloor T_v/y_v\rfloor}^{\infty} f_v(k_v)g_v(k_v)\right).
\end{equation}
The local factor at a real place $v$ is
\[\ll \sum_{k_v=1}^{\infty} (k_vT_v)^{2\eps} (1+T_v/y_v) e^{-2\pi k_vT_v} + \sum_{k_v=-\lfloor T_v/y_v\rfloor}^{0} T_v^{2\eps} \min(T_v^{1/3},T_v^{1/2}|k_vy_v|^{-1/2})
\ll_\eps T_v^{2\eps}(T_v^{1/3}+T_v/y_v),\]
as follows by estimating the minimum in the second sum by $T_v^{1/3}$ for $k_v=0$ and by $T_v^{1/2}|k_vy_v|^{-1/2}$ for $k_v<0$.
Similarly, the local factor at a complex place $v$ is
\begin{displaymath}
\begin{split}
& \ll \sum_{k_v=1}^{\infty} (k_vT_v)^{4\eps} k_v(1+T_v/y_v)^2 e^{-2\pi k_vT_v} + \sum_{k_v=-\lfloor T_v/y_v\rfloor}^{0} T_v^{4\eps} (1+T_v/y_v) \min(T_v^{1/3},T_v^{1/2}|k_vy_v|^{-1/2}) \\
& \ll_\eps T_v^{4\eps}(1+T_v/y_v)(T_v^{1/3}+T_v/y_v) \ll T_v^{4\eps}(T_v^{1/3}+T_v/y_v)^2\ll T_v^{4\eps}(T_v^{2/3}+T_v^2/y_v^2).
\end{split}
\end{displaymath}
All in all, the local factor at each archimedean place $v$ is
\begin{displaymath}
\ll_\eps |T_v|_v^{2\eps} (|T_v|_v^{1/3}+|T_v/y_v|_v).
\end{displaymath}
We conclude, by \eqref{cs_whittaker} and \eqref{cs_whittaker2},
\begin{equation}\label{cs_whittaker3}
\sum_{0\neq n\in\theta_i^{-1}\md^{-1}} |W(ny)|^2\prod_{v\mid \infty}(1+|n_vy_v|_v)^{2\eps}
\ll_\eps \prod_{v\mid \infty } |T_v|_v^{2\eps} (|T_v|_v^{1/3}+|T_v/y_v|_v) \ll |T|_\infty^{2\eps}(|T|_{\infty}^{1/3}+|T/y|_{\infty}),
\end{equation}
where we used the balancing assumption \eqref{yassumption} in the last step as follows:
\begin{align*}
|y|_\infty\leq|T|_\infty^{2/3}\qquad &\Longrightarrow\qquad\text{$|T_v|_v^{1/3}\ll |T_v/y_v|_v$ for all $v\mid\infty$},\\
|y|_\infty\geq|T|_\infty^{2/3}\qquad &\Longrightarrow\qquad\text{$|T_v|_v^{1/3}\gg |T_v/y_v|_v$ for all $v\mid\infty$}.
\end{align*}
Substituting \eqref{cs_whittaker3} into \eqref{cs}, and also using \eqref{cs_hecke} and \eqref{negligible_rho}, we obtain finally
\begin{align*}
\phi\left(\yx\begin{pmatrix}\theta_i \\ & 1\end{pmatrix}\right)
&\ll_{\eps} \left(|T|_{\infty}^{1/6}+|T/y|_{\infty}^{1/2}\right) |T|_{\infty}^{2\eps} \,\, |T/y|_\infty^{\eps/2} \, (\N\mn)^{2\eps}\\
&\ll_{\eps} \left(|T|_{\infty}^{1/6}+|T/y|_{\infty}^{1/2}\right)^{1+12\eps+\eps} (\N\mn)^{2\eps}.
\end{align*}
Replacing $\eps$ by $\eps/13$ completes the proof of \eqref{Fbound}.
\end{proof}

\begin{remark}\label{remark2} Lemma~\ref{lemma7} is close to optimal when $|y|_\infty$ is around $|T|_\infty$, but it does not capture the exponential decay of $\phi$ in the cusps. This result, however, suffices for our purposes, and in combination with Lemma~\ref{lemma4}, it yields a bound for $|\phi(g)|$ at any $g\in\GL_2(\Aa)$.
\end{remark}

\section{The amplifier}\label{amplifier}

In Section~\ref{section2}, we have seen that any function $f\in\Hs^\flat$ such that the operator $R(f)$ is positive
can be used to bound $|\phi(g)|$ at a given point $g\in\GL_2(\Aa)$ (cf.\ \eqref{ampbound1}), complementing the Fourier bound of the previous section (cf.\ Remark~\ref{remark2}). In this section, we construct this \emph{amplifier} as a pure tensor $f=f_\infty\otimes f\fin$, where
$f_\infty\in\Hs^\flat_\infty$ and $f\fin\in\Hs^\flat\fin$. Then, $R(f)$ is the product of the commuting operators $R(f_\infty)$ and $R(f\fin)$, hence it is positive as long as both $R(f_\infty)$ and $R(f\fin)$ are (cf.\ \cite[Section~104]{RSz}).

\subsection{The archimedean part of the amplifier} In Section~\ref{section4}, we introduced the generalized upper half-space $\Hb$ (cf.\ \eqref{hb}) as a subset of the Euclidean space $\MM$ (cf.\ \eqref{mm}), and we identify it with $G_\infty\slash K_\infty$ via \eqref{eq1}, \eqref{eq2}, \eqref{Kdef}, \eqref{groups2}. In particular, the left action of $G_\infty$ on $\Hb$ is given by generalized fractional linear transformations. For a point $P=(P_v)_{v\mid\infty}\in\Hb$ as in \eqref{associateP}, we write $\Im(P_v)$ for $y_v$, and we define
\begin{equation}\label{udef}
u_v(P_v,Q_v):=\frac{\|P_v-Q_v\|^2}{2\Im(P_v) \Im(Q_v)},\qquad P,Q \in \Hb,
\end{equation}
where $\|\cdot\|$ stands for the Euclidean norm (length) in the corresponding $\CC$ or $\HH$ component of $\MM$. Then $u=(u_v)_{v\mid\infty}$ is a point-pair invariant on $\Hb\times\Hb$, i.e.\ it is invariant under the diagonal left action of~$G_\infty$.

We define $f_\infty(g)$ in terms of the nonnegative vector $u(g\ib,\ib)$, where $\ib:=P(0,1)\in\Hb$ corresponds to the identity element of $G_\infty$ (cf.\ \eqref{associateP}). Specifically, in the next subsection, we shall choose\footnote{The notations $f_v$, $g_v$, $k_v$ are independent of those in the previous section. Hopefully, this does not cause confusion.} a smooth, compactly supported function $k_v:[0,\infty)\to\RR$ for each $v\mid\infty$, and we put for $g\in G_\infty$
\begin{equation}\label{finf}
f_\infty(g):=\prod_{v\mid\infty}f_v(g_v),\qquad f_v(g_v):=k_v\bigl(u_v(g_v\ib_v,\ib_v)\bigr).
\end{equation}
That is, $R(f_\infty)=R(\otimes_{v\mid\infty}f_v)=\prod_{v\mid\infty} R(f_v)$ is the product of the commuting operators $R(f_v)$, hence it is positive as long as the factors $R(f_v)$ are (cf.\ \cite[Section~104]{RSz}).

We determine $k_v(u)$ in terms of its Selberg/Harish-Chandra transform $h_v(t)$ on $\Hh^2$ or $\Hh^3$ (depending on the type of $v$), which is necessarily holomorphic, even, and rapidly decaying in every strip $|\im(t)|<A$, and it is real-valued on $\RR\cup\RR i$. We shall focus on the strip $|\im(t)|<1$, because for the positivity of $R(f_v)$ it suffices that $h_v(t)$ has positive real part there. Indeed, in this case $h_v(t)$ is the square of an even, holomorphic, rapidly decaying function $\tilde h_v(t)$ in the strip $|\im(t)|<1$, which is real-valued on $\RR\cup(-1,1)i$, therefore $R(f_v)$ is the square of the corresponding self-adjoint operator $R(\tilde f_v)$.

\subsection{Selberg/Harish-Chandra pairs} We fix once and for all a holomorphic, rapidly decaying function
\begin{equation}\label{jdef}
j:\{t\in\CC:|\im(t)|<1\}\to\{z\in\CC:\re(z)>0\}
\end{equation}
satisfying the symmetries
\begin{equation}\label{symmetry}
\ov{j(\ov t)}=j(t)=j(-t).
\end{equation}
We assume, further, that its \emph{inverse Fourier transform}
\[\hat j(x)=\frac{1}{2\pi}\int_{-\infty}^\infty j(t)e^{itx}dt,\qquad x\in\RR,\]
is smooth and supported in $[-1,1]$. Such a function $j(t)$ certainly exists, and we provide an example based on the original construction of Iwaniec--Sarnak~\cite{IS}. We take a smooth, even, not identically zero function $m:\RR\to\RR$ supported in $[-1/2,1/2]$ with \emph{Fourier transform}
\[\check m(t):=\int_{-\infty}^\infty m(x)e^{-itx}dx,\qquad t\in\CC,\]
and we define $j(t)$ as the convolution
\[j(t):=\int_{-\infty}^\infty \sech\left(\frac{\pi(t-s)}{2}\right)\check m(s)^2\,ds,\qquad |\im(t)|<1.\]
Then $\hat j(x)$ is the product of $2\sech(x)$ and the self-convolution $(m\ast m)(x)$, and the required properties of $j(t)$ and $\hat j(x)$ are straightforward to verify.

After this preparation, for each archimedean place $v\mid\infty$ we consider the even holomorphic function
\begin{equation}\label{hdef}
h_v(t):=j(t-T_v)+j(t+T_v),
\end{equation}
where $T_v$ is defined by \eqref{Tparameter}. Then, in accordance with \cite[(1.64)]{Iw} and Lemma~5.5 in \cite[Ch.~3]{EGM}, we obtain the inverse Selberg/Harish-Chandra transform $k_v(u)$ of $h_v(t)$ in three steps:
\begin{align}
g_v(x)&:=\hat h_v(x)=2\cos(T_v x)\hat j(x),& q_v(w)&:=\frac{1}{2}\,g_v\bigl(|2|_v\log\bigl(\sqrt{w+1}+\sqrt{w}\bigr)\bigr),&\label{gdef}\\
k_v(u)&:=\int_{u/2}^\infty\frac{-q_v'(w)\,dw}{\pi\sqrt{w-u/2}}\quad\text{for $v$ real,}&
k_v(u)&:=\frac{-q_v'(u/2)}{\pi}\quad\text{for $v$ complex.}&\label{kdef}
\end{align}
In checking these formulae, it is good to keep in mind the following. For $v$ real, $u$ in \cite[(1.64)]{Iw} is $u/2$ here. For $v$ complex, $h(1+t^2)$ and $g(x)$ in \cite[Ch.~3, (5.32)]{EGM} are $h_v(t/2)$ and $2g_v(2x)$ here, while $k(x)$ and $Q(\cosh(x))$ in \cite[Ch.~3, (5.35)]{EGM} are $k_v(x-1)$ and $4q_v(w)$ here, upon writing $\cosh(x)$ as $1+2w$, i.e. $x=2\log\bigl(\sqrt{w+1}+\sqrt{w}\bigr)$. The next lemma summarizes the properties that we need of these functions.

\begin{lemma}\label{pairlemma} The function $h_v(t)$ is holomorphic, even, and rapidly decaying in the strip $|\im(t)|<1$. It is positive on $\RR\cup(-1,1)i$, and it has positive real part in the strip $|\im(t)|<1$. Moreover, at the spectral parameter $t_v$ of $\phi$ (cf.\ \eqref{tparameter}) it satisfies the bound
\begin{equation}\label{boundh}
h_v(t_v)\gg 1.
\end{equation}
The inverse Selberg/Harish-Chandra transform $k_v(u)$ of $h_v(t)$ is smooth, real-valued, and supported in $[0,1]$. Moreover, it satisfies the bound
\begin{equation}\label{boundk}
k_v(u) \ll \min\left(|T_v|_v,|T_v|_v^{1/2}|u|_v^{-1/4}\right),\qquad u\geq 0.
\end{equation}
\end{lemma}

\begin{proof} It is clear by our construction that $h_v(t)$ is even, holomorphic, and rapidly decaying in the strip $|\im(t)|<1$. It is also clear by \eqref{jdef} and \eqref{hdef} that $h_v(t)$ has positive real part in the strip $|\im(t)|<1$. From \eqref{symmetry} it follows that $h_v(t)$ is real on $\RR\cup(-1,1)i$, hence in fact it is positive there. In particular, $h_v(t_v)$ is positive. Now we prove \eqref{boundh}. If $|t_v|\leq 1/2$, then $T_v=1/2$ by \eqref{Tparameter}, hence $h_v(t_v)$ lies in a fixed compact subset of $\CC$, and \eqref{boundh} follows. If $|t_v|>1/2$, then $t_v\in\RR$ and $T_v=|t_v|$ by \eqref{Tparameter}, hence $h_v(t_v)>j(0)$, and \eqref{boundh} follows again.

It is clear by our construction that $k_v(u)$ is smooth and real-valued. It is also clear by \eqref{gdef} that $g_v(x)$ is supported in $[-1,1]$, hence $k_v(u)$ vanishes when $u\geq 2\sinh^2(1/2)$. This shows that $k_v(u)$ is supported in $[0,1]$. Finally, we prove \eqref{boundk}. For a real place $v$, we follow closely the proof of \cite[Lemma~1.1]{IS}, so we shall be brief. By \eqref{gdef} we have
\[q_v(w)\ll 1\qquad\text{and}\qquad q_v'(w)\ll\min\bigl(T_v^2,T_vw^{-1/2}\bigr),\]
hence by \eqref{kdef} we have, for any $u>0$ and $\eta>0$,
\[k_v(u)=\int_0^\infty \frac{-q_v'(t+u/2)}{\pi\sqrt{t}}\,dt=\int_0^\eta+\int_\eta^\infty
\ll \eta^{1/2}\min\bigl(T_v^2,T_vu^{-1/2}\bigr)+\eta^{-1/2}.\]
Choosing $\eta:=\min\bigl(T_v^2,T_vu^{-1/2}\bigr)^{-1}$, the bound \eqref{boundk} follows. For a complex place $v$, the bound \eqref{boundk} is immediate from \eqref{gdef} and \eqref{kdef}, since in this case
\[q_v'(u/2)\ll\min\bigl(T_v^2,T_vu^{-1/2}\bigr)=\min\left(|T_v|_v,|T_v|_v^{1/2}|u|_v^{-1/4}\right).\]
The proof is complete.
\end{proof}

\subsection{The non-archimedean part of the amplifier}\label{napart} Let $C_0\geq 2$ be a sufficiently large constant depending only on the number field $F$, and let
\begin{equation}\label{sizeL}
C_0 \leq L\leq|T|_\infty(\N\mn)
\end{equation}
be a parameter to be optimized later\footnote{See also the first sentence in Section~\ref{endgame}.}. We consider all the \emph{totally split principal prime ideals} in $\mo$ coprime to $\mn\mq$ which are generated by an element from $F^\times_+\cap(1+\mq)$ and whose norm lies in $[L,2L]$, and for each of them we choose a totally positive generator
\begin{equation}\label{cong1}
l\in F^\times_+\cap(1+\mq).
\end{equation}
For each rational prime that occurs as a norm here, we keep exactly one of the $l$'s above it; furthermore, by restricting $l$ to an appropriate fundamental domain for $F^\times_+\slash\mo^\times_+$, we can and we shall assume that $l_v\asymp L^{1/n}$ holds for each archimedean place $v$ (cf.\ Remark~\ref{fundamentaldomainremark}). We denote the set of these prime elements $l\in\mo$ by $P(L)$. If the constant $C_0\geq 2$ in \eqref{sizeL} is sufficiently large, then $P(L)$ is non-empty and its cardinality satisfies
\begin{equation}\label{PL}
\#P(L) \asymp L/\log L
\end{equation}
by the extension of Dirichlet's theorem\footnote{For the purposes of this paper, the following weaker version of \eqref{cong1} would suffice:
for each real $v\mid\infty$ the sign of $l_v\in F_v^\times$ is constant, and for each $\mpr\mid\mq$ the square class of $l_\mpr\in\mo_\mpr^\times$ is constant. Hence, by the pigeonhole principle, we could do with Dirichlet's theorem for the class group only, and we could even avoid this variant by allowing a more general amplifier.} to narrow ray class groups \cite[\S 13 in Ch.~VII]{Ne}, since $\mq$ is a fixed ideal depending only on $F$.

Inspired by \cite[(4.11)]{Ve}, and using the notation of \eqref{tmdefflat} and \eqref{lambdadef}, we set
\begin{equation}\label{signcoeff}
x_m:=\sgn\bigl(\lambda(m\mo)\bigr)\qquad\text{for $m\in\mo$ coprime to $\mn\mq$},
\end{equation}
\begin{equation}\label{ffin}
f\fin:=\biggl(\sum_{l\in P(L)}x_l t^\flat_{l\mo}\biggr)\ast\biggl(\sum_{l\in P(L)}x_l t^\flat_{l\mo}\biggr)+
\biggl(\sum_{l\in P(L)}x_{l^2} t^\flat_{l^2\mo}\biggr)\ast\biggl(\sum_{l\in P(L)}x_{l^2} t^\flat_{l^2\mo}\biggr).
\end{equation}
Clearly, $R(f\fin)$ is positive, because it is a sum of squares of self-adjoint operators (cf.\ Section~\ref{section2}):
\begin{equation}\label{linearized}
R(f\fin)=\biggl(\sum_{l\in P(L)}x_l T^\flat_{l\mo}\biggr)^2+\biggl(\sum_{l\in P(L)}x_{l^2} T^\flat_{l^2\mo}\biggr)^2.
\end{equation}
In addition, by the multiplicativity relation \eqref{Hecke}, we can linearize the quadratic expression \eqref{ffin} as
\begin{equation}\label{ffin2}
f\fin=\sum_{0\neq m\in\mo} w_m t^\flat_{m\mo},\qquad
w_m:=\begin{cases}
\sum_{l\in P(L)}\left(x_l^2+x_{l^2}^2\right),&m=1;\\
x_{l_1}x_{l_2}+\delta_{l_1=l_2}x_{l_1^2}x_{l_2^2},&\text{$m=l_1l_2$ for some $l_1,l_2\in P(L)$};\\
x_{l_1^2}x_{l_2^2},&\text{$m=l_1^2l_2^2$ for some $l_1,l_2\in P(L)$};\\
0,&\text{otherwise}.
\end{cases}
\end{equation}
It follows from \eqref{signcoeff}, \eqref{linearized}, and the discussion of Section~\ref{section2}, that $R(f\fin)\phi=c(f\fin)\phi$ holds with
\begin{equation}\label{cfinbound}
c(f\fin)=\biggl(\sum_{l\in P(L)}|\lambda(l\mo)|\biggr)^2+\biggl(\sum_{l\in P(L)}|\lambda(l^2\mo)|\biggr)^2
\geq\frac{1}{2}\biggl(\sum_{l\in P(L)}|\lambda(l\mo)|+|\lambda(l^2\mo)|\biggr)^2>\frac{1}{8}\bigl(\#P(L)\bigr)^2.
\end{equation}
In the last step we used the bound $|\lambda(l\mo)|+|\lambda(l^2\mo)|>1/2$ which follows from \eqref{Hecke}.

\subsection{Reduction to a counting problem} Combining \eqref{tparameter}, \eqref{Tparameter}, \eqref{boundh}, \eqref{cfinbound}, we see that
\[c(f)=c(f\fin)c(f_\infty)=c(f\fin)\prod_{v\mid\infty}h_v(t_v)\gg c(f\fin)\gg \bigl(\#P(L)\bigr)^2,\]
hence \eqref{ampbound1}, \eqref{sizeL}, \eqref{PL} imply for the amplifier $f\in\Hs^\flat$ constructed above (cf.\ \eqref{notationformula}, \eqref{Tparameter}) that
\[L^2|\phi(g)|^2\preccurlyeq\sum_{\gamma\in\Gamma}f(g^{-1}\gamma g).\]
We recall that $\gamma$ runs through the elements of $\GL_2(F)$ modulo $\Z(\mo)$ (cf.\ \eqref{groups}). Now we fix a special matrix
\begin{equation}\label{gspecial}
g:=\yx\begin{pmatrix}\theta_i&\\&1\end{pmatrix}\in\FF(\mn)
\end{equation}
and the corresponding point $P=P(x,y)\in\Hb$ as in Lemma~\ref{lemma5}. Then, by \eqref{finf} and \eqref{ffin2}, the previous inequality becomes
\[L^2|\phi(g)|^2\preccurlyeq \sum_{0\neq m\in\mo} |w_m|\sum_{\gamma\in\Gamma}
t^\flat_{m\mo}\left(\begin{pmatrix}\theta_i^{-1}&\\&1\end{pmatrix}\gamma\fin\begin{pmatrix}\theta_i&\\&1\end{pmatrix}\right)
\left|k\bigl(u(\gamma P,P)\bigr)\right|,\]
where
\begin{equation}\label{kudef}
k\bigl(u(\gamma P,P)\bigr):=\prod_{v\mid\infty}k_v\bigl(u_v(\gamma_v P_v,P_v)\bigr).
\end{equation}

Assume that $\gamma\in\Gamma$ contributes to the inner sum for a given nonzero integer $m\in\mo$. By \eqref{ffin2} and \eqref{cong1}, $m$ is totally positive and congruent to $1$ modulo $\mq$, and by \eqref{Mdef1}, \eqref{Mdef3}, \eqref{tmdefflat}, $\gamma$ is represented by a matrix $\sabcd\in\GL_2(F)$
satisfying
\[a,d\in\mo,\qquad a-d\in\mq,\qquad b\in \theta_i\mo,\qquad c\in \theta_i^{-1}(\mn\cap\mq),\qquad\text{$ad-bc = mu$ for some $u\in\mo^\times$}.\]
Here we have that $bc\in\mq$, hence $ad-bc\in a^2+\mq$, and so the unit $u\in\mo^\times$ is a quadratic residue modulo $\mq$. By the choice of $\mq$ (cf.\ Section~\ref{section1}), we see that $u=v^2$ for some $v\in\mo^\times$, and our relations become
\[a,d\in\mo,\qquad a-d\in\mq,\qquad b\in \theta_i\mo,\qquad c\in \theta_i^{-1}(\mn\cap\mq),\qquad \text{$ad-bc = mv^2$ for some $v\in\mo^\times$}.\]
Multiplying $\sabcd\in\GL_2(F)$ by the inverse of $\bigl(\begin{smallmatrix}v&\\&v\end{smallmatrix}\bigr)\in\Z(\mo)$ does not change the class $\gamma\in\Gamma$ represented, nor does it change any of the congruence conditions on the entries $a,b,c,d$, so we can and we shall assume that $v=1$. Looking at the range of $t^\flat_{m\mo}$ (cf.\ \eqref{tmdefflat}), we arrive at
\begin{equation}\label{ampbound2}
L^2|\phi(g)|^2\preccurlyeq \vol(K^\flat\fin)^{-1}\sum_{0\neq m\in\mo} \frac{|w_m|}{\sqrt{\N(m\mo)}}
\sum_{\gamma\in\Gamma(i,m)}\left|k\bigl(u(\gamma P,P)\bigr)\right|,
\end{equation}
where
\begin{equation}\label{Rm}
\begin{split}
\Gamma(i,m):=\left\{\abcd\in\GL_2(F): \ a,d\in\mo,\ a-d\in\mq,\ b\in \theta_i\mo,\ c\in \theta_i^{-1}(\mn\cap\mq),\ ad-bc = m \right\}.
\end{split}
\end{equation}
We note that the entry $b$ here is always integral, because our ideal class representative $\theta_i\in\Aa\fin^\times$ lies in $\hat\mo$ (cf.\ Section~\ref{section1}). Most of the time we shall ignore the condition $a-d\in\mq$. It will only become important in Lemma~\ref{j=2}, which is the cornerstone for the proof of Theorem~\ref{thm3}.

We now associate to each $\gamma\in\Gamma(i,m)$ a certain dyadic
\[\delta=(\delta_v)_{v\mid\infty}=(2^{k_v})_{v\mid\infty}\in (0, 4]^{v\mid \infty}\]
with the property that the contribution $|k(u(\gamma P,P))|$ to the right-hand side of \eqref{ampbound2} can be estimated in terms of $|\delta|_{\infty}$.
Pick any $\gamma\in\Gamma(i,m)$, and for each $v\mid\infty$ consider the smallest integer $k_v\in\ZZ$ such that
\[\max\left(T_v^{-2},u_v(\gamma_v P_v,P_v)\right)\leq 2^{k_v}.\]
We can restrict to the case $k_v\leq 2$, because otherwise $u_v(\gamma_v P_v,P_v)>4$ (noting that $T_v^{-2}\leq 4$), and hence $k(u(\gamma P,P))=0$ by \eqref{kudef} and Lemma~\ref{pairlemma}. Denoting $\delta_v:=2^{k_v}$ for a moment, we have either $T_v^{-2}>\delta_v/2$ or $u_v(\gamma_v P_v,P_v)>\delta_v/2$, hence in both cases we get by \eqref{boundk}
\[k_v\bigl(u_v(\gamma_v P_v,P_v)\bigr)\ll |T_v|_v^{1/2} |\delta_v|_v^{-1/4}.\]
This in turn implies, by \eqref{kudef},
\begin{equation}\label{kubound}
k\bigl(u(\gamma P,P)\bigr)\ll |T|_\infty^{1/2} |\delta|_\infty^{-1/4}.
\end{equation}

In \eqref{ampbound2}, we rearrange the matrices $\gamma\in\Gamma(i,m)$ according to their integral vectors $\bk\in\ZZ^{\{v\mid\infty\}}$, and we group together the nonzero integers $m\in\mo$ with the same number of prime factors ($0$ or $2$ or $4$ prime factors in $P(L)$ as in \eqref{ffin2}). To summarize our findings, we define for any $j\in\{0,1,2\}$ and for any nonnegative vector $\delta=(\delta_v)_{v \mid \infty}$,
\begin{equation}\label{Mdef}
M(L, j, \delta) := \#\bigcup_{l_1,l_2\in P(L)}
\left\{\gamma \in \Gamma(i, l_1^jl_2^j): \text{$u_v(\gamma_v P_v,P_v)\leq \delta_v$ for all $v\mid\infty$}\right\}.
\end{equation}
We observe that in \eqref{ampbound2} we have, by \eqref{Kdef}, \eqref{Kdef3}, \eqref{notationformula},
\[\vol(K^{\flat}\fin)^{-1} = \vol(K\fin)^{-1}\prod_{\mpr\mid\mq}[K_\mpr:K^\flat_\mpr]\ll\vol(K\fin)^{-1}\preccurlyeq\N\mn,\]
while also $w_1\ll L$ and $w_m\ll 1$ for $m\neq 1$ by \eqref{signcoeff} and \eqref{ffin2}. We conclude, also using \eqref{kubound},
\begin{equation}\label{ampbound3}
|\phi(g)|^2\preccurlyeq(\N\mn)\sum_{\substack{\bk\in\ZZ^{\{v\mid\infty\}}\\T_v^{-2}\leq\delta_v=2^{k_v}\leq 4}}
\frac{|T|_\infty^{1/2}}{|\delta|_\infty^{1/4} }
\left(\frac{M(L,0,\delta)}{L}+\frac{M(L,1,\delta)}{L^3}+\frac{M(L,2,\delta)}{L^4}\right).
\end{equation}
We stress here that $g\in\GL_2(\Aa)$ is a special matrix of the form \eqref{gspecial}. It remains to bound the matrix counts $M(L,j,\delta)$ for the special $2$-adic vectors $\delta=(\delta_v)_{v\mid\infty}$ as above. This is the subject of the next section.

\section{Counting matrices}\label{countingmatrices}

We shall analyze in depth the matrices $\gamma=\sabcd\in\GL_2(F)$ counted by $M(L, j, \delta)$ for some
\begin{equation}\label{deltabounds}
\delta=(\delta_v)_{v\mid\infty}\qquad\text{with}\qquad 0<\delta_v\leq 4,
\end{equation}
which we assume for the rest of the paper. We recall the definitions \eqref{Rm} and \eqref{Mdef}. In particular, the determinant $l:=ad-bc$ is totally positive: it is of the form $l=l_1^jl_2^j$ with $l_1,l_2\in P(L)$, hence it satisfies\footnote{Hopefully, the typographical similarity of $l_v$ and $l_1$, $l_2$ will cause no confusion.} $l_v\asymp L^{2j/n}$ for each archimedean place $v\mid\infty$.

\subsection{Estimates on matrix entries}\label{estimatesonentries}
We claim that the following inequalities hold true at any archimedean place $v\mid\infty$, with absolute implied constants (which are independent of the number field $F$, the auxiliary ideal $\mq$, etc.):
\begin{equation}\label{est_1}
|c_vP_v+d_v|=|l_v|^{1/2}(1+O(\sqrt{\delta_v})),
\end{equation}
\begin{equation}\label{est_2}
|c_vP_v-a_v|=|l_v|^{1/2}(1+O(\sqrt{\delta_v})),
\end{equation}
\begin{equation}\label{est_3}
c_vy_v\ll |l_v|^{1/2},
\end{equation}
\begin{equation}\label{est_4}
2c_vx_v-a_v+d_v\ll |l_v|^{1/2},
\end{equation}
\begin{equation}\label{est_5}
a_v+d_v\ll |l_v|^{1/2},
\end{equation}
\begin{equation}\label{est_6}
\overline{c_v}\frac{l_v}{|l_v|}y_v^2-c_vx_v^2+(a_v-d_v)x_v+b_v\ll y_v|l_v|^{1/2}\sqrt{\delta_v},
\end{equation}
\begin{equation}\label{est_7}
-c_vx_v^2+ (a_v-d_v)x_v+b_v\ll y_v |l_v|^{1/2},
\end{equation}
\begin{equation}\label{est_8c}
\re\left(\frac{2c_vx_v-a_v+d_v}{\sqrt{l_v}}\right) \ll \sqrt{\delta_v},
\end{equation}
\begin{equation}\label{est_9c}
\im\left(\frac{a_v+d_v}{\sqrt{l_v}}\right)\ll \sqrt{\delta_v}.
\end{equation}
Here $\sqrt{l_v}$ denotes either of the two square-roots of $l_v$. At the complex places, these bounds follow from \cite[Section~6]{BHM}
upon noting that $\delta_v\ll 1$. We show below that these bounds also hold at the real places.

Let $v$ be a real place. We restrict ourselves to $l_v>0$, which is the only case needed for this paper, but we stress that the bounds also hold when $l_v<0$; we omit the proof for this case in order to save space.

Starting from $\gamma_v P_v = (a_vP_v+b_v)/(c_vP_v+d_v)$, it follows that
\[\frac{\Im(\gamma_v P_v)}{\Im(P_v)} = \frac{l_v}{\|c_vP_v+d_v\|^2},\]
and so
\[\delta_v \geq u_v(\gamma_v P_v, P_v) \geq \frac{|\Im(\gamma_v P_v) - \Im(P_v)|^2}{2\Im (\gamma_v P_v) \Im(P_v)}
= \frac{1}{2}\left| \frac{|l_v|^{1/2}}{\|c_vP_v+d_v\|} - \frac{\|c_vP_v+d_v\|}{|l_v|^{1/2}} \right|^2.\]
This gives \eqref{est_1} immediately, and also \eqref{est_2} upon noting that $u_v(\gamma_v P_v, P_v)=u_v(P_v, \gamma_v^{-1} P_v)$. By considering $\im(c_vP_v+d_v)$, we obtain \eqref{est_3} from \eqref{est_1} upon noting that $\delta_v\ll 1$. By considering $\re(c_vP_v+d_v)\pm\re(c_vP_v-a_v)$, we obtain \eqref{est_4} and \eqref{est_5} from \eqref{est_1} and \eqref{est_2} upon noting that $\delta_v\ll 1$. Using also
\[\delta_v \geq u_v(\gamma_v P_v, P_v) = \frac{\|a_vP_v + b_v - c_vP_v^2 - d_vP_v\|^2}{2 l_v y_v^2},\]
we obtain \eqref{est_6} and \eqref{est_8c} by considering the real and imaginary parts of the complex number in the numerator. Finally, \eqref{est_7} is a consequence of \eqref{est_6} coupled with \eqref{est_3} and $\delta_v\ll 1$, while \eqref{est_9c} is trivial as its left hand side vanishes.

\begin{remark} As noted in the beginning of this subsection, all implied constants are absolute here. In particular, our arguments yield the following explicit version of \eqref{est_5}:
\begin{equation}\label{est_5b}
|a_v+d_v|<17|l_v|^{1/2}.
\end{equation}
\end{remark}

\subsection{Preliminary bounds} We shall use part (a) of Corollary~\ref{corollary1} several times below, whenever we need to count lattice points in a box.

\begin{lemma}\label{prelimbound} We have
\[M(L, 0, \delta) \ll 1+ |y|_{\infty}|\delta|_{\infty}^{1/2}.\]
\end{lemma}

\begin{proof} Combining \eqref{est_3} with part (b) of Lemma~\ref{lemma5}, we see that the number of possibilities for $c$ is
\begin{displaymath}
\# c \ll 1+ |y|_{\infty}^{-1}(\N\mn)^{-1} \ll 1.
\end{displaymath}
For a fixed $c$, the number of possibilities for the difference $a-d$ can be bounded similarly, using \eqref{est_4}, namely
\begin{displaymath}
\# (a-d) \ll 1.
\end{displaymath}
For a fixed pair $(c,a-d)$ and $l = \det \gamma = 1$, the number of possibilities for $b$ is, by \eqref{est_6},
\begin{displaymath}
\# b \ll 1+|y|_{\infty}|\delta|_{\infty}^{1/2}.
\end{displaymath}
Finally, the quadruple $(c,a-d,l,b)$ determines $(c,a-d,ad,b)$, and the latter determines $(c,a,d,b)$ up to two choices.
Hence we conclude the proof by multiplying the above bounds.
\end{proof}

Let $M_0(L, j, \delta)$ and $M_1(L, j, \delta)$ be the number of matrices $\gamma=\sabcd\in\GL_2(F)$ counted by $M(L, j, \delta)$ with $c=0$ and $c\neq 0$, respectively.

\begin{lemma}\label{lemma10} For $j\in\{1,2\}$ we have
\[M_0(L, j, \delta) \preccurlyeq L^2+L^{2+j}|y|_{\infty}|\delta|_{\infty}^{1/2}.\]
\end{lemma}

\begin{proof} For the determinant $l=l_1^jl_2^j$ with $l_1,l_2\in P(L)$, there are $\ll L^2$ choices. Let us fix the determinant;
then by $l=ad$ there are $\preccurlyeq 1$ possibilities for the ideals $a\mo$ and $d\mo$. On the other hand, \eqref{est_1} and \eqref{est_2} imply
that $a_v,d_v \ll |l_v|^{1/2}$, hence by part (b) of Corollary~\ref{corollary1} there are $\preccurlyeq 1$ choices for the pair $(a,d)$. Fixing this pair, the number of possibilities for $b$ is $\ll 1+L^j|y|_{\infty}|\delta|_{\infty}^{1/2}$ by \eqref{est_6}, and we conclude the proof by multiplying the above bounds.
\end{proof}

\subsection{Parabolic matrices}\label{parabolicsection} Let us write
\[M_1(L, j, \delta) = M_2(L, j, \delta) + M_3(L, j, \delta),\qquad j\in\{1,2\},\]
where $M_2(L, j, \delta)$ and $M_3(L, j, \delta)$ stand for the number of parabolic and nonparabolic matrices counted by $M_1(L, j, \delta)$, respectively. Here we call a matrix $\gamma$ parabolic if $(\tr\gamma)^2 = 4\det\gamma$. The main result in this subsection is the bound for $M_2(L, j, \delta)$ in Lemma~\ref{parabolic-real}. Estimates for $M_3(L,j,\delta)$ are the subject of the next subsection.

\begin{lemma}\label{lemma11}
Let $j\in\{1,2\}$. Then $M_2(L, j, \delta)=0$ unless $|\delta|_{\infty}\gg L^{-2j}|y|_{\infty}^{-2}$.
\end{lemma}

\begin{proof} Assume that $\gamma=\sabcd\in\GL_2(F)$ is a matrix counted by $M_2(L, j, \delta)$, so that $c\neq 0$ and $(a+d)^2=4(ad-bc)$. Then $\gamma$ has a unique fixed point in $\Hb\cup F_\infty$, namely the field element $(a-d)/(2c)\in F$ embedded in $F_\infty$. For convenience we write $p:=a-d$ and $q:=2c$, and we extend these elements to an invertible matrix $\bigl(\begin{smallmatrix} s & -r \\ -q & p \end{smallmatrix}\bigr) \in \GL_2(F)$.

By Lemma~\ref{lemma3}, and recalling again \eqref{gspecial}, we have a decomposition
\[\begin{pmatrix} s & -r \\ -q & p \end{pmatrix}\yx\begin{pmatrix} \theta_i \\ & 1 \end{pmatrix}=
\tst\yxt \begin{pmatrix} \theta_j \\ & 1 \end{pmatrix} k,\]
where $\stst\in\gP(F)$, $\syxt\in\gP(F_\infty)$, and $k\in K^\ast$. Multiplication on the left by $\stst^{-1}$ does not affect the bottom row of $\bigl(\begin{smallmatrix} s & -r \\ -q & p \end{smallmatrix}\bigr)$, and so we may assume that $s,r\in F$ have been chosen so that
\begin{equation}\label{srfeature}
\begin{pmatrix} s & -r \\ -q & p \end{pmatrix}\yx\begin{pmatrix} \theta_i \\ & 1 \end{pmatrix}=
\yxt \begin{pmatrix} \theta_j \\ & 1 \end{pmatrix} k.
\end{equation}
Furthermore, by changing $k$ if necessary, we may assume that $\tilde y_v>0$ for any archimedean place $v\mid\infty$ (cf.\ \eqref{rotate}). Note that $|y|_\infty\geq|\tilde y|_\infty$ by $\syx\bigl(\begin{smallmatrix}\theta_i&\\&1\end{smallmatrix}\bigr)\in\FF(\mn)$.

With such a choice of $\bigl(\begin{smallmatrix} s&-r\\ -q&p\end{smallmatrix}\bigr)$ and $k$, we set $\sigma := \frac{1}{ps-qr}\bigl(\begin{smallmatrix} p & r \\ q & s \end{smallmatrix}\bigr)\in\GL_2(F)$ for the corresponding inverse, and we examine the parabolic matrix $\sigma^{-1}\gamma\sigma\in\GL_2(F)$ by looking at its infinite and finite component separately. The infinite component has (unique) fixed point $\infty$, which implies readily that
\begin{equation}\label{conjugatedgamma}
\sigma^{-1}\gamma\sigma=\begin{pmatrix}\lambda & b' \\ & \lambda\end{pmatrix}\in\GL_2(F)
\end{equation}
with $\lambda^2=l$. In particular, $\lambda\in F$ implies $\lambda\in\mo$ (i.e.\ $l$ is a square). We claim that $b'\in\theta_j\mo$. Note that $b'\neq 0$, because $c\neq 0$; therefore, this will furnish the useful bound $|b'|_{\infty}\gg 1$.

To justify the claim, we start by observing that \eqref{srfeature} yields
\[\sigma\fin^{-1}=\begin{pmatrix} s & -r \\ -q & p \end{pmatrix}\fin =
\begin{pmatrix} \theta_j \\ & 1 \end{pmatrix} k\fin\begin{pmatrix} \theta_i^{-1} \\ & 1 \end{pmatrix},\]
\[(\sigma^{-1}\gamma\sigma)\fin=
\begin{pmatrix} \theta_j \\ & 1 \end{pmatrix} k\fin
\begin{pmatrix} \theta_i^{-1} \\ & 1 \end{pmatrix}\abcd\fin\begin{pmatrix} \theta_i \\ & 1 \end{pmatrix}
k\fin^{-1}\begin{pmatrix} \theta_j^{-1} \\ & 1 \end{pmatrix}.\]
As the determinant $ad-bc$ is coprime to $\mn$, our assumptions (cf.\ \eqref{Rm}) yield for the middle part
\[\begin{pmatrix} \theta_i^{-1} \\ & 1 \end{pmatrix}\abcd\fin\begin{pmatrix} \theta_i \\ & 1 \end{pmatrix}\in
\prod_{\mpr\nmid\mn}\M_2(\mo_\mpr)\prod_{\mpr\mid\mn}K_\mpr.\]
The set on the right hand side is normalized by the Atkin--Lehner group $K\fin^*:=\prod_\mpr K_\mpr^*$, hence we infer
\[(\sigma^{-1}\gamma\sigma)\fin\in \begin{pmatrix} \theta_j \\ & 1 \end{pmatrix}\M_2(\hat\mo)\begin{pmatrix} \theta_j \\ & 1 \end{pmatrix}^{-1}.\]
By \eqref{conjugatedgamma}, this implies $b'\in\theta_j\mo$ as claimed.

Looking at the infinite component again, \eqref{srfeature} shows that $\sigma^{-1}P=\tilde P$, where $P=P(x,y)\in\Hb$ is as before, and
$\tilde P:=P(\tilde x,\tilde y)\in\Hb$. With this notation, we get for any archimedean place $v\mid\infty$,
\begin{displaymath}
u_v\left(\gamma P_v, P_v \right) =
u_v\big( \sigma^{-1} \gamma P_v, \sigma^{-1} P_v \big) =
u_v\big(\sigma^{-1} \gamma \sigma \tilde P_v, \tilde P_v \big) =
u_v\left(\begin{pmatrix}\lambda & b' \\ & \lambda\end{pmatrix} \tilde P_v, \tilde P_v\right).
\end{displaymath}
Bounding the left hand side by $\delta_v$ and evaluating the right hand side via \eqref{udef}, we get with the usual absolute value in each archimedean completion $F_v$ that $\delta_v\gg|b'_v|^2|\lambda_v|^{-2}{\tilde y}_v^{-2}$. In particular,
$|\delta|_{\infty}\gg|b'|_{\infty}^2L^{-2j}|\tilde y|_\infty^{-2}$. However, as we have remarked earlier, $|b'|_{\infty}\gg 1$ and $|y|_\infty\geq |\tilde y|_\infty$, and therefore $|\delta|_{\infty}\gg L^{-2j}|y|_\infty^{-2}$.
\end{proof}

We recall now the notation
\[|\delta|_\infty=|\delta|_{\RR}\cdot|\delta|_{\RR},\qquad
|\delta|_{\RR}=\prod_{\text{$v$ real}}\delta_v,\qquad |\delta|_{\CC}=\prod_{\text{$v$ complex}}\delta_v^2,\]
which is a special case of \eqref{RnormCnorm}.

\begin{lemma}\label{parabolic-real} For $j\in\{1,2\}$ we have
\[M_2(L, j, \delta) \preccurlyeq L^{3j} |\delta|_{\RR}^{3/4} |\delta|_{\CC}^{1/4} (\N\mn)^{-1}.\]
\end{lemma}

\begin{remark}\label{rem4} This result is a number field version of \cite[Lemma~4.4]{Te}. Unfortunately, we were unable to reconstruct the proof of Templier's lemma. It remains unclear to us how the referenced argument in \cite{HT1} would generalize to produce the bound of \cite[Lemma~4.4]{Te}, as there does not appear to be an obvious reason why the number of scaling matrices $\sigma_{\ma}$ is uniformly bounded. Here we give a simple but robust proof of an alternative bound that suffices for our purposes and would also suffice for \cite{Te}.
\end{remark}

\begin{proof} Assume that $\gamma=\sabcd\in\GL_2(F)$ is a matrix counted by $M_2(L, j, \delta)$, so that $c\neq 0$ and $(a+d)^2=4(ad-bc)$. Then, as in the previous proof, the determinant $l=ad-bc$ is a square, and $a+d=2\lambda$ holds for one the two square-roots $\lambda$ of $l$. At a real place $v$, we can estimate $c_v$ by \eqref{est_1} and \eqref{est_8c} as follows:
\[(c_v y_v)^2=|c_v P_v + d_v|^2-(c_vx_v+d_v)^2=\lambda_v^2(1+O(\sqrt{\delta_v}))^2-\lambda_v^2(1+O(\sqrt{\delta_v}))^2
\ll \lambda_v^2\sqrt{\delta_v},\]
and hence
\begin{equation}\label{c-real}
c_vy_v \ll L^{j/n} \delta_v^{1/4}.
\end{equation}
At a complex place $v$, we record the simpler and weaker bound \eqref{est_3}:
\begin{equation}\label{c-complex}
c_vy_v \ll L^{j/n}.
\end{equation}
For a fixed $c\neq 0$, the identity $(a-d)^2+4bc=0$ shows that the integer $a-d$ is divisible by a fixed ideal of norm at least $(\N(c\mo))^{1/2}$, hence by \eqref{est_8c} and \eqref {est_4} combined with part (a) of Corollary~\ref{corollary1}, there are
\[\#(a-d)\ll 1 + L^{j}|\delta|_{\RR}^{1/2} (\N(c\mo))^{-1/2}\]
choices for this integer. We estimate the total number of pairs $(c,a-d)$ by summing these bounds over all nonzero elements $c\in\theta_i^{-1}\mn$ that satisfy \eqref{c-real} and \eqref{c-complex}, collecting first the pairs corresponding to a given fractional ideal $c\mo$, and then applying part (b) of Corollary~\ref{corollary1} for each such subsum. In this way we get, for any $\eps>0$,
\[\#(c, a-d) \preccurlyeq \frac{L^{j} |\delta|_{\RR}^{1/4} }{|y|^{1+\eps}_{\infty}(\N\mn)} +
\frac{L^{3j/2} |\delta|_{\RR}^{5/8}}{|y|^{1/2+\eps}_{\infty}(\N\mn)}
\preccurlyeq \frac{L^{2j} |\delta|_{\RR}^{3/4} |\delta|_{\CC}^{1/4} }{\N\mn},\]
where in the last step we have bounded $|y|_{\infty}$ from below by invoking Lemma~\ref{lemma11}.
Finally, the trace $a+d=2\lambda$ can be chosen in $\ll L^{j}$ ways, so by $(a-d)^2+4bc=0$ the total number of possibilities for the parabolic matrix $\gamma$ is
\[\#(c,a-d,a+d)\preccurlyeq L^{3j} |\delta|_{\RR}^{3/4} |\delta|_{\CC}^{1/4} (\N\mn)^{-1}.\]
The proof is complete.
\end{proof}

\subsection{Generic matrices} Again, we shall use part (a) of Corollary~\ref{corollary1} several times below.

\begin{lemma}\label{newlemma} We have
\begin{align}
\label{subdiv14}M_3(L,1,\delta)\ll L^2&+\frac{L^{5/2}|\delta|_\RR^{1/4}}{(\N\mn)^{1/4}}+\frac{L^4|\delta|_\RR|\delta|_\CC^{3/4}}{\N\mn},\\[4pt]
\label{subdiv15}M_3(L,2,\delta)\preccurlyeq L^2&+\frac{L^4|\delta|_\RR^{1/2}}{(\N\mn)^{1/2}}+\frac{L^6|\delta|_\RR|\delta|_\CC^{1/2}}{\N\mn}.
\end{align}
\end{lemma}

\begin{remark}
Our proof actually shows that \eqref{subdiv14} holds for $M_1(L,1,\delta)$ in place of $M_3(L,1,\delta)$, but we preferred the current formulation for harmony.
\end{remark}

\begin{proof} Let $j\in\{1,2\}$. If $M_3(L, j, \delta)$ vanishes, then the bound stated for it holds trivially. Otherwise, we fix some $c\neq 0$ that occurs in $M_3(L, j, \delta)$. By \eqref{est_3}, the number of possibilities for $c$ satisfies
\begin{equation}\label{subdiv6}
\# c \ll \frac{L^j}{|y|_{\infty}(\N\mn)}.
\end{equation}
We denote by $M_3(L, j, \delta, c)$ the subcount of $M_3(L,j,\delta)$ with the given $c$, and we shall subdivide it as
\begin{equation}\label{subdiv7}
M_3(L, j, \delta, c)=M_3(\ast)=\sum_\bn M_3(\ast,\bn)=\sum_{\bn,\bp} M_3(\ast,\bn,\bp)=\sum_{\bn,\bp,\bq} M_3(\ast,\bn,\bp,\bq),
\end{equation}
where $\ast$ abbreviates ``$L, j, \delta, c$'', and $\bn=(n_v)$, $\bp=(p_v)$, $\bq=(q_v)$ are vectors in $\ZZ^{\{\text{$v$ complex}\}}$ satisfying
\begin{equation}\label{subdiv8}
n_v,p_v,q_v\ll\delta_v^{-1/2},\qquad\text{$v$ complex}.
\end{equation}
The role of the parameters $\bn$, $\bp$, $\bq$ is to partially localize $l$, $a-d$, $a+d$ at the various complex places, so that we can make the most of the bounds collected in Subsection~\ref{estimatesonentries}.

The components of $\bn\in\ZZ^{\{\text{$v$ complex}\}}$ satisfy
\[0 \leq n_v < 2\pi/\sqrt{\delta_v},\qquad\text{$v$ complex},\]
and we denote by $M_3(\ast,\bn)$ the subcount of $M_3(\ast)$ with the additional condition
\begin{equation}\label{ncond}
n_v \sqrt{\delta_v} \leq \arg(l_v) < (n_v + 1) \sqrt{\delta_v},\qquad\text{$v$ complex},
\end{equation}
for the determinant $l=ad-bc$. If $M_3(\ast,\bn)$ vanishes, then we subdivide it trivially as
\[M_3(\ast,\bn)=M_3(\ast,\bn,\mathbf{0})=M_3(\ast,\bn,\mathbf{0},\mathbf{0})=0.\]
Otherwise, we fix a matrix $\gamma_\bn=\bigl(\begin{smallmatrix}a_\bn&b_\bn\\c&d_\bn\end{smallmatrix}\bigr)$ counted by $M_3(\ast,\bn)$. Any matrix $\gamma=\sabcd$ counted by $M_3(\ast,\bn)$ is determined by the differences
\begin{equation}\label{diff}
a':=a-a_\bn,\qquad d':=d-d_\bn,\qquad b':=b-b_\bn.
\end{equation}
By \eqref{ncond}, the determinants $l:=\det\gamma$ and $l_\bn:=\det\gamma_\bn$ satisfy
\begin{equation}\label{lcond}
\frac{l_v}{|l_v|}-\frac{l_{\bn,v}}{|l_{\bn,v}|}\ll \sqrt{\delta_v}\quad\text{and}\quad
\frac{\sqrt{l_v}}{\sqrt{l_{\bn,v}}} = \left|\frac{l_v}{l_{\bn,v}}\right|^{1/2} + O(\sqrt{\delta_v}),\qquad v\mid\infty,
\end{equation}
with a suitable choice of the square-roots, upon noting that our determinants are totally positive.
Using $l_v,l_{\bn,v}\asymp L^{2j/n}$ and $y_v\asymp|y|_\infty^{1/n}$ (cf.\ \eqref{fund}), the first part of \eqref{lcond} combined with \eqref{est_3} and \eqref{est_6} yields
\begin{equation}\label{subdiv5}
(a_v'-d_v')x_v + b_v' \ll L^{j/n}|y|_\infty^{1/n}\sqrt{\delta_v},\qquad v\mid\infty,
\end{equation}
while the second part of \eqref{lcond} combined with \eqref{est_8c} and \eqref{est_9c} yields
\begin{equation}\label{subdiv1}
\re\left(\frac{a'_v - d'_v}{\sqrt{l_{\bn,v}}}\right), \ \im\left(\frac{a'_v + d'_v}{\sqrt{l_{\bn,v}}}\right)\ll \sqrt{\delta_v},
\qquad v\mid\infty.
\end{equation}
We complement these bounds with the simpler relations that follow from \eqref{est_4} and \eqref{est_5},
\begin{equation}\label{subdiv2}
\im\left(\frac{a'_v - d'_v}{\sqrt{l_{\bn,v}}}\right), \ \re\left(\frac{a'_v + d'_v}{\sqrt{l_{\bn,v}}}\right)\ll 1,
\qquad v\mid\infty.
\end{equation}
Of course, the above bounds for imaginary parts are trivial at the real places.

Now we denote by $M_3(\ast,\bn,\bp,\bq)$ the subcount of $M_3(\ast,\bn)$ with the additional conditions
\begin{align}
\label{subdiv3}p_v\sqrt{\delta_v}&\leq\im\left(\frac{a'_v - d'_v}{\sqrt{l_{\bn,v}}}\right)<(p_v+1)\sqrt{\delta_v},\qquad\text{$v$ complex};\\
\label{subdiv4}q_v\sqrt{\delta_v}&\leq\re\left(\frac{a'_v + d'_v}{\sqrt{l_{\bn,v}}}\right)<(q_v+1)\sqrt{\delta_v},\qquad\text{$v$ complex}.
\end{align}
By \eqref{subdiv2}, $M_3(\ast,\bn,\bp,\bq)$ vanishes unless $p_v,q_v\ll\delta_v^{-1/2}$, so we shall restrict to $\bp,\bq\in\ZZ^{\{\text{$v$ complex}\}}$ satisfying this condition. If $M_3(\ast,\bn,\bp,\bq)\neq 0$, then
we fix a matrix $\gamma_{\bn,\bp,\bq}=\bigl(\begin{smallmatrix}a_{\bn,\bp,\bq}&b_{\bn,\bp,\bq}\\c&d_{\bn,\bp,\bq}\end{smallmatrix}\bigr)$ counted by $M_3(\ast,\bn,\bp,\bq)$. Any matrix $\gamma=\sabcd$ counted by $M_3(\ast,\bn,\bp,\bq)$ is determined by the differences
\[a'':=a-a_{\bn,\bp,\bq},\qquad d'':=d-d_{\bn,\bp,\bq},\qquad b'':=b-b_{\bn,\bp,\bq}.\]
We remark that with \eqref{diff} and the analogous notation
\[a_{\bn,\bp,\bq}':=a_{\bn,\bp,\bq}-a_\bn,\qquad d_{\bn,\bp,\bq}':=d_{\bn,\bp,\bq}-d_\bn,\qquad b_{\bn,\bp,\bq}':=b_{\bn,\bp,\bq}-b_\bn,\]
we can also write
\[a''=a'-a_{\bn,\bp,\bq}',\qquad d''=d'-d_{\bn,\bp,\bq}',\qquad b''=b'-b_{\bn,\bp,\bq}'.\]
The point is that \eqref{subdiv5}--\eqref{subdiv4} also hold with $(a_{\bn,\bp,\bq}',d_{\bn,\bp,\bq}',b_{\bn,\bp,\bq}')$ in place of $(a',d',b')$. In order to balance out the different sizes of $\delta_v$ at the various archimedean places, we fix a totally positive unit $s\in\mo^\times_+$ such that
(cf.\ Remark~\ref{fundamentaldomainremark})
\[s_v\sqrt{\delta_v} \asymp |\delta|_\infty^{1/(2n)},\qquad v\mid\infty,\]
and we switch to the variables
\[\tilde a:=sa'',\qquad\tilde d:=sd'',\qquad\tilde b:=sb''.\]
By \eqref{subdiv5}, these scaled differences satisfy
\begin{equation}\label{subdiv9}
(\tilde a_v-\tilde d_v)x_v + \tilde b_v \ll L^{j/n}|y|_\infty^{1/n}|\delta|_\infty^{1/(2n)},\qquad v\mid\infty,
\end{equation}
and by \eqref{subdiv1}--\eqref{subdiv4} they also satisfy
\begin{align}
\label{subdiv10}&\tilde a_v-\tilde d_v\ll L^{j/n}|\delta|_\infty^{1/(2n)}\qquad\text{and}\qquad
\tilde a_v+\tilde d_v\ll L^{j/n}s_v,&&\text{$v$ real};\\
\label{subdiv11}&\tilde a_v-\tilde d_v\ll L^{j/n}|\delta|_\infty^{1/(2n)}\qquad\text{and}\qquad
\tilde a_v+\tilde d_v\ll L^{j/n}s_v\sqrt{\delta_v},&&\text{$v$ complex}.
\end{align}
By \eqref{subdiv9} and the first parts of \eqref{subdiv10}--\eqref{subdiv11}, the Euclidean norm of the lattice point $(\tilde a -\tilde d)P+\tilde b\in\Lambda(P)$ (cf.\ \eqref{LambdaP}) can be bounded as
\[\bigl\|(\tilde a -\tilde d)P+\tilde b\bigr\|\ll L^{j/n}|y|_\infty^{1/n}|\delta|_\infty^{1/(2n)},\]
and hence by part (d) of Lemma~\ref{lemma5},
\[\#(\tilde a-\tilde d,\tilde b) \ll 1 + L^j|y|_\infty|\delta|_\infty^{1/2}(\N\mn)^{1/2} + L^{2j}|y|_\infty|\delta|_\infty.\]
In addition, by the second parts of \eqref{subdiv10}--\eqref{subdiv11},
\begin{equation}\label{a+d}
\#(\tilde a+\tilde d) \ll 1+L^j|\delta|_\CC^{1/2}.
\end{equation}
Let us assume $|\delta|_\CC\gg L^{-2j}$ for a moment. Then, as any matrix $\gamma=\sabcd$ counted by $M_3(\ast,\bn,\bp,\bq)$ is determined by the triple $(\tilde a-\tilde d,\tilde b,\tilde a+\tilde d)$, we see that
\[M_3(\ast,\bn,\bp,\bq)\ll
L^j|\delta|_\CC^{1/2}\left(1 + L^j|y|_\infty|\delta|_\infty^{1/2}(\N\mn)^{1/2} + L^{2j}|y|_\infty|\delta|_\infty\right).\]
We combine this bound with \eqref{subdiv6}--\eqref{subdiv8}. Using also part (b) of Lemma~\ref{lemma5}, we obtain
\begin{align}M_3(L, j, \delta)
\notag&\ll\frac{L^j}{|y|_{\infty}(\N\mn)}\frac{L^j|\delta|_\CC^{1/2}}{|\delta|_\CC^{3/4}}
\left(1 + L^j|y|_\infty|\delta|_\infty^{1/2}(\N\mn)^{1/2} + L^{2j}|y|_\infty|\delta|_\infty\right)\\
\notag&\ll\frac{L^{2j}}{|\delta|_\CC^{1/4}}+
\frac{L^{3j}|\delta|_\RR^{1/2}|\delta|_\CC^{1/4}}{(\N\mn)^{1/2}}+\frac{L^{4j}|\delta|_\RR|\delta|_\CC^{3/4}}{\N\mn}\\
\label{subdiv12}&\ll\frac{L^{2j}}{|\delta|_\CC^{1/4}}+\frac{L^{4j}|\delta|_\RR|\delta|_\CC^{3/4}}{\N\mn}.
\end{align}
In the last step, we dropped the middle term, because it is the geometric mean of the other two terms.
The obtained bound is valid under the assumptions \eqref{deltabounds} and $|\delta|_\CC\gg L^{-2j}$.

We shall use \eqref{subdiv12} for $j=1$ only, because for $j=2$ we can do without \eqref{a+d}, hence also without the second parts of \eqref{subdiv10}--\eqref{subdiv11}, by the following observation. For any matrix $\gamma=\sabcd$ counted by $M_3(\ast,\bn)$, the determinant $l=l_1^2l_2^2$ is a square and $(a-d)^2 + 4bc\neq 0$. The identity
\[0\neq (a-d)^2 + 4bc = (a+d)^2 - 4l = (a+d-2l_1l_2)(a+d+2l_1l_2)\]
combined with \eqref{est_5} implies that each pair $(a-d, b)$ gives rise to $\preccurlyeq 1$ choices for $a+d$. Indeed, there are $\preccurlyeq 1$ choices for the ideals $(a+d\pm 2l_1l_2)\mo$ as their product is a fixed nonzero ideal of norm $\ll L^{4}$. Using part (b) of Corollary~\ref{corollary1}, and again keeping in mind~\eqref{est_5}, for each choice of ideals there are $\preccurlyeq 1$ possibilities for their generators $a+d\pm 2l_1l_2\in\mo$, which in turn determine $a+d$. So, in the case of $j=2$, we only need the first parts of \eqref{subdiv10}--\eqref{subdiv11}, along with \eqref{subdiv9}. Hence, instead of $M_3(\ast,\bn,\bp,\bq)$, we consider $M_3(\ast,\bn,\bp)$ defined as the subcount of $M_3(\ast,\bn)$ with the additional condition \eqref{subdiv3}, and we obtain an improved version of \eqref{subdiv12} even without the assumption $|\delta|_\CC\gg L^{-2j}$:
\begin{align}M_3(L, 2, \delta)
\notag&\preccurlyeq\frac{L^2}{|y|_{\infty}(\N\mn)}\frac{1}{|\delta|_\CC^{1/2}}
\left(1 + L^2|y|_\infty|\delta|_\infty^{1/2}(\N\mn)^{1/2} + L^4|y|_\infty|\delta|_\infty\right)\\
\notag&\preccurlyeq\frac{L^2}{|\delta|_\CC^{1/2}}+\frac{L^4|\delta|_\RR^{1/2}}{(\N\mn)^{1/2}}+\frac{L^6|\delta|_\RR|\delta|_\CC^{1/2}}{\N\mn}\\
\label{subdiv13}&\preccurlyeq\frac{L^2}{|\delta|_\CC^{1/2}}+\frac{L^6|\delta|_\RR|\delta|_\CC^{1/2}}{\N\mn}.
\end{align}
In the last step, we dropped the middle term, because it is the geometric mean of the other two terms.
The obtained bound is valid under the assumption \eqref{deltabounds}.

Now we derive \eqref{subdiv14} from \eqref{subdiv12} specialized to $j=1$. If $|\delta|_\CC>L^{-2}|\delta|_\RR^{-1}(\N\mn)$, then \eqref{subdiv12} is applicable, and the second term dominates in it, so \eqref{subdiv14} follows. If $|\delta|_\CC\leq L^{-2}|\delta|_\RR^{-1}(\N\mn)$, then we can replace each $\delta_v$ by some $\delta_v\leq\tilde\delta_v\leq 4$ so that
\[|\tilde\delta|_\RR=|\delta|_\RR\qquad\text{and}\qquad
|\tilde\delta|_\CC=\min\left(16^{r_2},L^{-2}|\delta|_\RR^{-1}(\N\mn)\right).\]
As \eqref{subdiv12} is applicable with $\tilde\delta$ in place of $\delta$, we obtain
\[M_3(L,1,\delta)\leq M_3(L,1,\tilde\delta)\ll\frac{L^2}{|\tilde\delta|_\CC^{1/4}}+\frac{L^4|\delta|_\RR|\tilde\delta|_\CC^{3/4}}{\N\mn}
\ll \frac{L^2}{|\tilde\delta|_\CC^{1/4}} \ll L^2+\frac{L^{5/2}|\delta|_\RR^{1/4}}{(\N\mn)^{1/4}},\]
and \eqref{subdiv14} follows again.

Finally, we derive \eqref{subdiv15} from \eqref{subdiv13}. If $|\delta|_\CC>L^{-4}|\delta|_\RR^{-1}(\N\mn)$, then the second term dominates in \eqref{subdiv13}, so \eqref{subdiv15} follows. If $|\delta|_\CC\leq L^{-4}|\delta|_\RR^{-1}(\N\mn)$, then we can replace each $\delta_v$ by some $\delta_v\leq\tilde\delta_v\leq 4$ so that
\[|\tilde\delta|_\RR=|\delta|_\RR\qquad\text{and}\qquad
|\tilde\delta|_\CC=\min\left(16^{r_2},L^{-4}|\delta|_\RR^{-1}(\N\mn)\right).\]
Applying \eqref{subdiv13} with $\tilde\delta$ in place of $\delta$, we obtain
\[M_3(L,2,\delta)\leq M_3(L,2,\tilde\delta)\preccurlyeq\frac{L^2}{|\tilde\delta|_\CC^{1/2}}+\frac{L^6|\delta|_\RR|\tilde\delta|_\CC^{1/2}}{\N\mn}
\ll \frac{L^2}{|\tilde\delta|_\CC^{1/2}} \ll L^2+\frac{L^4|\delta|_\RR^{1/2}}{(\N\mn)^{1/2}},\]
and \eqref{subdiv15} follows again.
\end{proof}

Our final two lemmas are the main applications of the rigidity Lemma~\ref{aplusd}. Lemma~\ref{j=2} below
is the crucial input for the proof of Theorem~\ref{thm3} and the only point where the congruences in the matrix count imposed by the ideal $\mq$ become relevant. Let $F_0$ be the maximal totally real subfield of $F$, and put $m:=[F:F_0]$.

\begin{lemma}\label{improvement} Suppose that $F$ is not totally real, i.e.\ $m \geq 2$. Then
\[M_3(L,1,\delta)\ll L^2+L^{2m}|\delta|_\RR^{1/2}|\delta|_\CC^{1/4}+\frac{L^{2m+1}|\delta|_\RR|\delta|_\CC^{3/4}}{\N\mn}.\]
\end{lemma}

\begin{remark}
Our proof actually shows that this bound holds for $M_1(L,1,\delta)$ in place of $M_3(L,1,\delta)$, but we preferred the current formulation for harmony. The result is the precise analogue of \cite[Subsection~11.1]{BHM}, but the present proof is very different and applies to all number fields without any special treatment of CM-fields. The bound is sharp for very small distances and is responsible for a strong exponent of $|T|_{\infty}$ in Theorem~\ref{thm3}. In fact, by arguing similarly as in Lemma~\ref{j=2} with $F_1:=F_0\big((a+d)^2/l\big)$, we can prove that $M_3(L,1,\delta)=0$ unless $1\ll L^{8(m-1)}|\delta|_\CC$. We do not see how to exploit this conclusion in the context of estimating $M_3(L,1,\delta)$ in the endgame in Section~\ref{endgame}, so we omit this statement and its proof.
\end{remark}

\begin{proof}
The proof shares several common elements with the proof of Lemma~\ref{newlemma}, so we shall be brief at certain points.
If $M_3(L,1,\delta)$ vanishes, then the stated bound holds trivially. Otherwise, we fix some $c\neq 0$ that occurs in $M_3(L,1,\delta)$.
By \eqref{est_3}, the number of possibilities for $c$ satisfies \eqref{subdiv6}, where we set $j:=1$. We denote by $M_3(L,1,\delta,c)$ the subcount of $M_3(L,1,\delta)$ with the given $c$, and we split it as
\begin{equation}\label{subdivnew1}
M_3(L, 1, \delta, c)=M_3(\ast)=M_3(\ast,0)+M_3(\ast,1),
\end{equation}
where $\ast$ abbreviates ``$L, 1, \delta, c$'', and the subsums $M_3(\ast,0)$ and $M_3(\ast,1)$ refer to two complementary ranges described in \eqref{subdivnew2} and \eqref{subdivnew4} below.

Specifically, $M_3(\ast,0)$ denotes the subcount of $M_3(\ast)$ satisfying the additional condition
\begin{equation}\label{subdivnew2}
|2cx-a+d|_\infty\leq 1.
\end{equation}
In order to estimate $M_3(\ast,0)$, we shall subdivide it as
\[M_3(\ast,0)=\sum_{\bn,\bq} M_3(\ast,0,\bn,\bq),\]
where $\bn,\bq\in\ZZ^{\{\text{$v$ complex}\}}$ are as in the proof of Lemma~\ref{newlemma}, and $M_3(\ast,0,\bn,\bq)$ denotes the subcount of $M_3(\ast,0)$ satisfying \eqref{ncond} and \eqref{subdiv4}. By \eqref{subdivnew2}, the number of choices for $a-d$ is $\ll 1$. For fixed $c$, $a-d$, $\bn$, $\bq$, \eqref{subdiv9} and \eqref{a+d} give
\[\#b \ll 1+L|y|_\infty|\delta|_\infty^{1/2}\qquad\text{and}\qquad \#(a+d)\ll 1+L|\delta|_\CC^{1/2}.\]
Let us assume $|\delta|_\CC\geq L^{-2}$ for a moment. As the triple $(a-d,b,a+d)$ determines the matrix $\gamma=\sabcd$ counted by $M_3(\ast,0,\bn,\bq)$, we see that
\begin{equation}\label{subdivnew3}
M_3(\ast,0)=\sum_{\bn,\bq}M_3(\ast,0,\bn,\bq)\ll\frac{L|\delta|_\CC^{1/2}}{|\delta|_\CC^{1/2}}\left(1+L|y|_\infty|\delta|_\infty^{1/2}\right)
=L+L^2|y|_\infty|\delta|_\infty^{1/2}.
\end{equation}
For a general $\delta=(\delta_v)_{v\mid\infty}$, we obtain
\begin{equation}\label{subdivnew7}
M_3(\ast,0)\ll L+L^2|y|_\infty|\delta|_\infty^{1/2}+L|y|_\infty|\delta|_\RR^{1/2}.
\end{equation}
Indeed, if $|\delta|_\CC\geq L^{-2}$, then this is obvious by \eqref{subdivnew3}. If $|\delta|_\CC<L^{-2}$, then we can replace each $\delta_v$ by some $\delta_v\leq\tilde\delta_v\leq 4$ so that $|\tilde\delta|_\RR=|\delta|_\RR$ and $|\tilde\delta|_\CC=L^{-2}$. As \eqref{subdivnew3} is applicable with $\tilde\delta$ in place of $\delta$, and the left hand side of \eqref{subdivnew3} is non-decreasing in $|\delta|_\CC$, the bound \eqref{subdivnew7} follows again.

We turn to the analysis of $M_3(\ast,1)$, which we define as the subcount of $M_3(\ast)$ satisfying the additional condition
\begin{equation}\label{subdivnew4}
|2cx-a+d|_\infty>1.
\end{equation}
We subdivide $M_3(\ast,1)$ in terms of \emph{dyadic vectors} $\bz=(z_v)\in\NN^{\{v\mid\infty\}}$,
\begin{equation}\label{subdivnew1b}
M_3(\ast,1)=\sum_\bz M_3(\ast,1,\bz),
\end{equation}
where $M_3(\ast,1,\bz)$ denotes the subcount of $M_3(\ast,1)$ with the additional condition
\begin{equation}\label{subdivnew5}
\frac{1}{2}<\frac{z_v|2c_vx_v-a_v+d_v|}{C_1|l_v|^{1/2}\sqrt{\delta_v}}\leq 1\ \ \text{for $v$ real,}\qquad\quad
\frac{1}{2}<\frac{z_v|2c_vx_v-a_v+d_v|}{C_1|l_v|^{1/2}}\leq 1\ \ \text{for $v$ complex.}
\end{equation}
Here, $C_1>0$ is the maximum of the two implied constants in \eqref{est_4} and \eqref{est_8c}, so that \eqref{subdivnew1b} holds with dyadic vectors $\bz\in\NN^{\{v\mid\infty\}}$. By \eqref{subdivnew4} and \eqref{subdivnew5}, either $M_3(\ast,1,\bz)$ is empty or
$|\bz|_\infty\ll L|\delta|_\RR^{1/2}$, and hence by $l_v\asymp L^{2/n}$ the number of choices for $a-d$ is
\begin{equation}
\label{impr_choicesforaminusd}
\#(a-d)\ll L|\delta|^{1/2}_{\RR}|\bz|_{\infty}^{-1}.
\end{equation}
Fixing $a-d$ and combining \eqref{est_8c} with \eqref{subdivnew5}, a moment of thought gives that
\begin{equation}
\label{interval}
\arg(l_v)\equiv\pi+2\arg(2c_vx_v-a_v+d_v)+O(z_v\sqrt{\delta_v})\pmod{2\pi},\qquad\text{$v$ complex}.
\end{equation}

At this point the stage is set for the application of realness rigidity, Lemma~\ref{aplusd}. We write
\[N:=\left\lceil\Bigl(C_2L^{2(m-1)}|\delta|_{\CC}^{1/4}\Bigr)^{1/r_2}\right\rceil,\]
where $C_2>0$ is a sufficiently large constant (depending only on $F$) to be chosen shortly. Inequality \eqref{interval} states that
$\arg(l_v)$ belongs to a fixed interval of length $\ll z_v\sqrt{\delta_v}$ for every complex place $v$. We split each of these intervals into subintervals $I(v,j_v)$ of length at most $\sqrt{\delta_v}/N$, where $j_v\ll z_vN$ is a positive integer. Then, writing $\bj:=(j_v)$, the total number of combinations of these subintervals at the various complex places can be bounded as
\begin{equation}
\label{impr_choicesforintervals}
\#\bj\ll |\bz|_{\CC}^{1/2}N^{r_2}\ll |\bz|_{\CC}^{1/2}\left(1+L^{2(m-1)}|\delta|_{\CC}^{1/4}\right).
\end{equation}
We claim that for each combination $\bj$, there exists at most one determinant $l=l_1l_2$ with $l_1,l_2\in P(L)$ such that $\arg(l_v)\in I(v,j_v)$ for all complex places $v$. Indeed, let $l=l_1l_2$ and $l'=l'_1l'_2$ be any two such determinants, and let $\xi:=l'/l$ be their ratio. Then, $\arg(l_v)$ and $\arg(l'_v)$ differ by at most $\sqrt{\delta_v}/N$ at every complex place $v$, and hence by the definition of $P(L)$ in Subsection~\ref{napart},
\[\xi_v\ll 1\quad\text{and}\quad \im \xi_v \ll \sqrt{\delta_v}/N,\qquad v\mid\infty.\]
Consider the number field $F_1:=F_0(\xi)$. If $F=F_1$, then Lemma~\ref{aplusd} is applicable, and we obtain
\[1\ll L^{8(m-1)}|\delta|_\CC/N^{4r_2}\leq 1/C_2^{4}.\]
By choosing the constant $C_2>0$ to be sufficiently large, this inequality is impossible, hence $F_1$ is a \emph{proper} subfield of $F$. We claim that $\xi\in\mo$. We can write $\xi\mo_{F_1}$ as a ratio $\ma/\mb$ of nonzero coprime ideals $\ma,\mb\subseteq\mo_{F_1}$ in $F_1$. Then $\xi\mo$ is the ratio $(\ma\mo)/(\mb\mo)$ of the nonzero coprime ideals $\ma\mo,\mb\mo\subseteq\mo$ in $F$, and therefore $\mb\mo$ divides $l\mo$. By the definition of $P(L)$, $l\mo$ is a product of totally split prime ideals such that distinct prime ideal factors lie above distinct rational primes, hence we infer that $\mb\mo=\mo$. As a result, $\xi\in\ma\mo\subseteq\mo$ is an integer as claimed. By switching the roles $l$ and $l'$, we also see that $\xi^{-1}\in\mo$, i.e.\ $\xi\in\mo^\times$ is a unit. However, again by the definition of $P(L)$, this forces $\xi=1$, i.e.\ $l=l'$.

By the previous paragraph, we have (cf.\ \eqref{impr_choicesforintervals})
\[\# l\leq \#\bj \ll |\bz|_{\CC}^{1/2}\left(1+L^{2(m-1)}|\delta|_{\CC}^{1/4}\right).\]
Moreover, by \eqref{est_6}, for each fixed pair $(a-d,l)$ we have (noting that $c$ is also fixed)
\[\#b \ll 1+L|y|_\infty|\delta|_\infty^{1/2}.\]
As the triple $(a-d,l,b)$ determines the matrix $\gamma=\sabcd$ counted by $M_3(\ast,1,\bz)$ up to two choices, we see by
\eqref{impr_choicesforaminusd} and our last two bounds that
\[M_3(\ast,1,\bz)\ll \frac{L|\delta|^{1/2}_{\RR}}{|\bz|_{\RR}|\bz|_{\CC}^{1/2}}
\left(1+L^{2(m-1)}|\delta|_{\CC}^{1/4}\right)\left(1+L|y|_\infty|\delta|_\infty^{1/2}\right).\]
Summing up over all dyadic vectors $\bz\in\NN^{\{v\mid\infty\}}$, we infer
\begin{equation}\label{subdivnew6}
M_3(\ast,1)\ll \left(L+L^{2m-1}|\delta|_\RR^{1/2}|\delta|_\CC^{1/4}\right)\left(1+L|y|_\infty|\delta|_\infty^{1/2}\right).
\end{equation}

Finally, combining \eqref{subdiv6}, \eqref{subdivnew1}, \eqref{subdivnew7}, \eqref{subdivnew6},
and using also part (b) of Lemma~\ref{lemma5}, we obtain
\begin{align*}
M_3(L,1,\delta)
\ll&\ \frac{L}{|y|_{\infty}(\N\mn)}\left(L+L^2|y|_\infty|\delta|_\infty^{1/2}+L|y|_\infty|\delta|_\RR^{1/2}\right)\\
+&\ \frac{L}{|y|_{\infty}(\N\mn)}\left(L+L^{2m-1}|\delta|_\RR^{1/2}|\delta|_\CC^{1/4}\right)\left(1+L|y|_\infty|\delta|_\infty^{1/2}\right)\\
\ll&\ L^2+L^{2m}|\delta|_\RR^{1/2}|\delta|_\CC^{1/4}+\frac{L^{2m+1}|\delta|_\RR|\delta|_\CC^{3/4}}{\N\mn}.
\end{align*}
The proof is complete.
\end{proof}

\begin{lemma}\label{j=2} We have $M_3(L,2,\delta)=0$ unless
\begin{equation}\label{unlessconclusion}
1 \ll L^{8(m-1)}|\delta|_\CC.
\end{equation}
\end{lemma}

\begin{proof} Assume that $\gamma=\sabcd\in\GL_2(F)$ is a matrix counted by $M_3(L,2,\delta)$. Recall that the determinant is of the form $l=l_1^2l_2^2$ with $l_1,l_2\in P(L)$. We observe that
\begin{equation}\label{j=2a}
\xi := \frac{a+d}{l_1l_2} \in F
\end{equation}
satisfies, by \eqref{est_5b} and \eqref{est_9c},
\begin{equation}\label{xiwbounds}
|\xi_v|<17\quad\text{and}\quad \im \xi_v \ll \sqrt{\delta_v},\qquad v\mid\infty.
\end{equation}
Consider the number field $F_1:=F_0(\xi)$. If $F=F_1$, then Lemma~\ref{aplusd} is applicable, and we obtain \eqref{unlessconclusion} from \eqref{j=2a} and \eqref{xiwbounds}. Otherwise, $F_1$ is a \emph{proper} subfield of $F$. We claim that $\xi\in\mo$. If $\xi=0$, then the claim is trivial. If $\xi\neq 0$, then we can write $\xi\mo_{F_1}$ as a ratio $\ma/\mb$ of nonzero coprime ideals $\ma,\mb\subseteq\mo_{F_1}$ in $F_1$. Then $\xi\mo$ is the ratio $(\ma\mo)/(\mb\mo)$ of the nonzero coprime ideals $\ma\mo,\mb\mo\subseteq\mo$ in $F$, and therefore $\mb\mo$ divides $l_1l_2\mo$. By the definition of $P(L)$ in Subsection~\ref{napart}, $l_1l_2\mo$ is a product of totally split prime ideals such that distinct prime ideal factors lie above distinct rational primes, hence we infer that $\mb\mo=\mo$. As a result, $\xi\in\ma\mo\subseteq\mo$ is an integer as claimed. By $a-d\in\mq$ and $bc\in\mq$ (cf.\ \eqref{Rm}), we even see that $\xi^2-4=((a-d)^2+4bc)/l\in\mq$. Now $\xi^2-4\in\mq$ is nonzero (because $\gamma$ is nonparabolic), hence its norm $|\xi^2-4|_\infty$ is at least $\N\mq$. However, this is impossible by \eqref{sizeq} and \eqref{xiwbounds} as the following short calculation shows:
\[300^n\leq\N\mq\leq|\xi^2-4|_\infty=\prod_{v\mid\infty} {|\xi_v^2-4|}_v<\prod_{v\mid\infty}|17^2+4|_v=293^n.\]
The proof is complete.
\end{proof}

\section{The endgame}\label{endgame}

Theorems~\ref{thm2} and \ref{thm3} are trivial when $|T|_\infty(\N\mn)$ is bounded, hence we can assume that $|T|_\infty(\N\mn)$ is sufficiently large in terms of the number field $F$. We recall \eqref{ampbound3} in the form
\begin{equation}\label{simpleampbound}
|\phi(g)|^2\preccurlyeq (\N\mn)
\sup_{\substack{\delta =(\delta_v)_{v\mid\infty}\\T_v^{-2}\leq\delta_v\leq 4\text{ for all } v \mid \infty}}
\frac{|T|_\infty^{1/2}}{|\delta|_\infty^{1/4}}\left( \frac{ M(L,0,\delta)}{L} +\frac{M(L,1,\delta)}{L^3}+\frac{M(L,2,\delta)}{L^4}\right),
\end{equation}
where $L$ is an arbitrary amplifier length satisfying \eqref{sizeL}. We note that the vector $\delta$ satisfies
\begin{equation}\label{deltaT}
|T|_{\RR}^{-2}\leq |\delta|_{\RR}, \qquad |T|_{\CC}^{-2}\leq |\delta|_{\CC}.
\end{equation}

First we prove Theorem~\ref{thm2}. For $j\in\{1,2\}$ we have:
\begin{align*}
&M(L,0,\delta)\ll 1 + |y|_{\infty}|\delta|_{\infty}^{1/2}&&\text{by Lemma~\ref{prelimbound}},\\
&M_0(L,j,\delta)\preccurlyeq L^2 + L^{2+j} |y|_{\infty}|\delta|_{\infty}^{1/2}&&\text{by Lemma~\ref{lemma10}},\\
&M_2(L,j,\delta)\preccurlyeq L^{3j} |\delta|_{\RR}^{3/4} |\delta|_{\CC}^{1/4} (\N\mn)^{-1}&&\text{by Lemma~\ref{parabolic-real}},\\
&M_3(L,1,\delta)\ll L^2+L^{5/2} |\delta|_{\RR}^{1/4} (\N\mn)^{-1/4}+L^{4}|\delta|_{\RR}|\delta|_{\CC}^{3/4} (\N\mn)^{-1}&&\text{by Lemma~\ref{newlemma}},\\
&M_3(L,2,\delta)\preccurlyeq L^2+L^{4} |\delta|^{1/2}_{\RR} (\N\mn)^{-1/2}+L^{6}|\delta|_{\RR}|\delta|_{\CC}^{1/2} (\N\mn)^{-1}&&\text{by Lemma~\ref{newlemma}}.
\end{align*}
From \eqref{simpleampbound} and \eqref{deltaT} we infer
\[|\phi(g)|^2 \preccurlyeq (\N\mn) \left(\frac{|T|_{\infty}}{L} + |T|_{\infty}^{1/2}|y|_{\infty} + \frac{L^2 |T|_{\infty}^{1/2}}{\N\mn} + \frac{ |T|_{\RR}^{1/2} |T|_{\CC}}{(\N \mn)^{1/2}}\right).\]
We equate the first and third term by choosing $L:=|T|_{\infty}^{1/6} (\N\mn)^{1/3}$, and then our bound becomes
\begin{equation}\label{thm1stronger}
\phi(g)\preccurlyeq |T|_{\infty}^{5/12} (\N\mn)^{1/3} + |T|_{\RR}^{1/4} |T|_{\CC}^{1/2} (\N\mn)^{1/4}
\end{equation}
as long as $|y|_{\infty} \leq |T|_{\infty}^{1/3}(\N\mn)^{-1/3}$, but \eqref{thm1stronger} remains true in the opposite case by Lemma~\ref{lemma7}.
This completes the proof of Theorem~\ref{thm2}.

We prove Theorem~\ref{thm3} following a similar strategy. For $j\in\{1,2\}$ we have:
\begin{align*}
&M(L,0,\delta)\ll 1 + |y|_{\infty}|\delta|_{\infty}^{1/2}&&\text{by Lemma~\ref{prelimbound}},\\
&M_0(L,j,\delta)\preccurlyeq L^2 + L^{2+j} |y|_{\infty}|\delta|_{\infty}^{1/2}&&\text{by Lemma~\ref{lemma10}},\\
&M_2(L,j,\delta)\preccurlyeq L^{3j} |\delta|_{\RR}^{3/4} |\delta|_{\CC}^{1/4} (\N\mn)^{-1}&&\text{by Lemma~\ref{parabolic-real}},\\
&M_3(L,1,\delta)\ll L^2+L^{2m}|\delta|_\RR^{1/2}|\delta|_\CC^{1/4}+L^{2m+1}|\delta|_\RR|\delta|_\CC^{3/4}(\N\mn)^{-1}&&\text{by Lemma~\ref{improvement}},\\
& M_3(L, 2, \delta) \preccurlyeq L^2 + L^{2m+2}|\delta|^{1/2}_{\RR} |\delta|_{\CC}^{1/4} (\N\mn)^{-1/2} + L^6 |\delta|_{\RR} |\delta|_{\CC}^{1/2}(\N \mn)^{-1}&&\text{by Lemmata~\ref{newlemma} and \ref{j=2}},
\end{align*}
where $m = [F:F_0] \geq 2$, and we inserted the factor $L^{2(m-1)}|\delta|_{\CC}^{1/4} \gg 1$ artificially (cf.\ \eqref{unlessconclusion}) in the second term on the last line. From \eqref{simpleampbound} and \eqref{deltaT} we infer
\[|\phi(g)|^2 \preccurlyeq (\N\mn)
\left(\frac{|T|_{\infty}}{L} + |T|^{1/2}_{\infty}\left(|y|_{\infty}+ L^{2m-3} + \frac{L^{2m-2}}{(\N\mn)^{1/2}}\right) \right).\]
Introducing the constant $C_3:=2^nC_0$ (cf.\ \eqref{sizeL}) and choosing
\[L:=\min\left(\bigl(|T|_{\infty}\N\mn\bigr)^{\frac{1}{4m-2}}, C_3|T|_{\infty}^{\frac{1}{4m-4}}\right),\]
we obtain
\begin{equation}\label{thm2final}
\phi(g) \preccurlyeq \bigl(|T|_{\infty}\N\mn\bigr)^\frac{1}{2}
\left(\bigl(|T|_{\infty}\N\mn\bigr)^{-\frac{1}{8m-4}} + |T|_{\infty}^{-\frac{1}{8m-8}}\right)
\end{equation}
as long as $|y|_{\infty}\leq|T|_{\infty}^{1/4}$, but \eqref{thm2final} remains true in the opposite case by Lemma~\ref{lemma7}.
Using also \eqref{thm1stronger}, we can strengthen \eqref{thm2final} to
\[\phi(g) \preccurlyeq \bigl(|T|_{\infty}\N\mn\bigr)^\frac{1}{2}
\left(\bigl(|T|_{\infty}\N\mn\bigr)^{-\frac{1}{8m-4}} + \min\left(|T|_{\infty}^{-\frac{1}{8m-8}}, (\N\mn)^{-\frac{1}{4}}\right)\right).\]
However,
\[\min\left(|T|_{\infty}^{-\frac{1}{8m-8}}, (\N\mn)^{-\frac{1}{4}}\right) \leq \left(|T|_{\infty}^{-\frac{1}{8m-8}}\right)^{\frac{2m-2}{2m-1}} \left((\N\mn)^{-\frac{1}{4}}\right)^{\frac{1}{2m-1}} = \bigl(|T|_{\infty} \N\mn\bigr)^{-\frac{1}{8m-4}},\]
hence in fact our bound simplifies to
\[\phi(g) \preccurlyeq \bigl(|T|_{\infty}\N\mn\bigr)^{\frac{1}{2}-\frac{1}{8m-4}}.\]
This completes the proof of Theorem~\ref{thm3}.


\begin{thebibliography}{1000}

\bibitem{As} E. Assing, \emph{On sup-norm bounds part I: ramified Maa{\ss} newforms over number fields}, {\tt arXiv:1710.00362}

\bibitem{BB} V. Blomer, F. Brumley, \emph{On the Ramanujan conjecture over number fields}, Ann. of Math. (2) \textbf{174} (2011), 581--605.

\bibitem{BHM} V. Blomer, G. Harcos, D. Mili\'cevi\'c, \emph{Bounds for eigenforms on arithmetic hyperbolic 3-manifolds}, Duke Math. J. \textbf{165} (2016), 625--659.

\bibitem{BHo} V. Blomer, R. Holowinsky, \emph{Bounding sup-norms of cusp forms of large level}, Invent. Math. \textbf{179} (2010), 645--681.

\bibitem{BM} V. Blomer, P. Michel, \emph{Hybrid bounds for automorphic forms on ellipsoids over number fields}, J. Inst. Math. Jussieu \textbf{12} (2013), 727--758.

\bibitem{BP} V. Blomer, A. Pohl, \emph{The sup-norm problem on the Siegel modular space of rank two}, Amer. J. Math. \textbf{138} (2016), 999--1027.

\bibitem{BMa} V. Blomer, P. Maga, \emph{Subconvexity for sup-norms of cusp forms on {\rm PGL}(n)}, Selecta Math. \textbf{22} (2016), 1269--1287.

\bibitem{BV} T. Browning, P. Vishe, \emph{Cubic hypersurfaces and a version of the circle method for number fields}, Duke Math. J. \textbf{163} (2014), 1825--1883.

\bibitem{Ca} W. Casselman, \emph{On some results of Atkin and Lehner}, Math. Ann. \textbf{201} (1973), 301--314.

\bibitem{DS} S. Das, J. Sengupta, \emph{$L^\infty$ norms of holomorphic modular forms in the case of compact quotient}, Forum Math. \textbf{27} (2015), 1987--2001.

\bibitem{DFI} W. Duke, J. B. Friedlander, H. Iwaniec, \emph{Bounds for automorphic $L$-functions}, Invent. Math. \textbf{112} (1993), 1--8; \emph{II}, ibid. \textbf{115} (1994), 219--239; \emph{Erratum for II}, ibid. \textbf{140} (2000), 227--242; \emph{III}, ibid. \textbf{143} (2001), 221--248.

\bibitem{EGM} J. Elstrodt, F. Grunewald, J. Mennicke, \emph{Groups acting on hyperbolic space}, Springer Monographs in Mathematics, Springer-Verlag, Berlin, 1998.

\bibitem{Fl} D. Flath, \emph{Decomposition of representations into tensor products}, In: Automorphic forms, representations, and $L$-functions (A. Borel, W. Casselman eds.), Part 1, 179--183, Proc. Sympos. Pure Math. \textbf{33}, Amer. Math. Soc., Providence, RI, 1979.

\bibitem{FI} J. Friedlander, H. Iwaniec, \emph{A mean-value theorem for character sums}, Michigan Math. J. \textbf{39} (1992), 153--159.

\bibitem{Ga} A. Gamburd, \emph{On the spectral gap for infinite index ``congruence'' subgroups of $\SL_2(\mathbf{Z})$}, Israel J. Math. \textbf{127} (2002), 157--200.

\bibitem{GL} P. M. Gruber, C. G. Lekkerkerker, \emph{Geometry of numbers}, 2nd edition, North-Holland Mathematical Library \textbf{37}, North-Holland Publishing Co., Amsterdam, 1987.

\bibitem{HM} G. Harcos, P. Michel, \emph{The subconvexity problem for Rankin--Selberg $L$-functions and equidistribution of Heegner points. II}, Invent. Math. \textbf{163} (2006), 581--655.

\bibitem{HT1} G. Harcos, N. Templier, \emph{On the sup-norm of Maass cusp forms of large level: II}, Int. Math. Res. Not. \textbf{2012}, no. 20, 4764--4774.

\bibitem{HT2} G. Harcos, N. Templier, \emph{On the sup-norm of Maass cusp forms of large level. III}, Math. Ann. \textbf{356} (2013), 209--216.

\bibitem{HRR} R. Holowinsky, G. Ricotta, E. Royer, \emph{On the sup-norm of $SL_3$ Hecke-Maass cusp form}, {\tt arXiv:1404.3622}

\bibitem{Iw1} H. Iwaniec, \emph{Fourier coefficients of modular forms of half-integral weight}, Invent. Math. \textbf{87} (1987), 385--401.

\bibitem{Iw2} H. Iwaniec, \emph{Small eigenvalues of Laplacian for $\Gamma_0(N)$}, Acta Arith. \textbf{56} (1990), 65--82.

\bibitem{Iw} H. Iwaniec, \emph{Spectral methods of automorphic forms}, 2nd edition, Graduate Studies in Mathematics \textbf{53}, Amer. Math. Soc., Providence, RI, 2002.

\bibitem{IS} H. Iwaniec, P. Sarnak, \emph{$L^{\infty}$ norms of eigenfunctions of arithmetic surfaces}, Ann. of Math. (2) \textbf{141} (1995), 301--320.

\bibitem{JK1} J. Jorgenson, J. Kramer, \emph{Bounding the sup-norm of automorphic forms}, Geom. Funct. Anal. \textbf{14} (2004), 1267--1277.

\bibitem{JK2} J. Jorgenson, J. Kramer, \emph{Bounds on Faltings's delta function through covers}, Ann. of Math. (2) \textbf{170} (2009), 1--43.

\bibitem{JK3} J. Jorgenson, J. Kramer, \emph{Effective bounds for Faltings's delta function}, Ann. Fac. Sci. Toulouse Math. (6) \textbf{23} (2014), 665--698.

\bibitem{Ki} E. M. K\i ral, \emph{Bounds on sup-norms of half integral weight modular forms}, Acta Arith. \textbf{165} (2014), 385--399.

\bibitem{MR} C. Maclachlan, A. Reid, \emph{The arithmetic of hyperbolic 3-manifolds}, Graduate Texts in Mathematics \textbf{219}, Springer-Verlag, New York, 2003.

\bibitem{Ma} P. Maga, \emph{Subconvexity and shifted convolution sums over number fields}, Ph.D. thesis, Central European University, 2013, available at \verb|http://www.etd.ceu.hu/2014/maga_peter.pdf|

\bibitem{Mar} S. Marshall, \emph{Sup norms of Maass forms on semisimple groups}, {\tt arXiv:1405.7033}

\bibitem{Me} J. Mercer, \emph{Functions of positive and negative type, and their connection with the theory of integral equations}, Philos. Trans. Roy. Soc. London Ser. A \textbf{209} (1909), 415--446.

\bibitem{Mil} D. Mili\'cevi\'c, \emph{Large values of eigenfunctions on arithmetic hyperbolic 3-manifolds}, Geom. Funct. Anal. \textbf{21} (2011), 1375--1418.

\bibitem{Mi} T. Miyake, {\em On automorphic forms on $GL_2$ and Hecke operators}, Ann. of Math. \textbf{94} (1971), 174--189.

\bibitem{Ne} J. Neukirch, \emph{Algebraic number theory}, Translated from the 1992 German original and with a note by N. Schappacher, Grundlehren der Mathematischen Wissenschaften \textbf{322}, Springer-Verlag, Berlin, 1999.

\bibitem{Ok} R. Okazaki, \emph{Inclusion of CM-fields and divisibility of relative class numbers}, Acta Arith. \textbf{92} (2000), 319--338.

\bibitem{RSz} F. Riesz, B. Sz.-Nagy, \emph{Functional analysis}, Translated from the 2nd French edition by L. F. Boron, Frederick Ungar Publishing Co., New York, 1955.

\bibitem{Sa} A. Saha, \emph{Hybrid sup-norm bounds for Maass newforms of powerful level}, Algebra Number Theory \textbf{11} (2017), 1009--1045.

\bibitem{Sa3} P. Sarnak, \emph{Arithmetic quantum chaos}, Expanded version of the Schur lectures (Tel-Aviv) and the Blyth lectures (Toronto), May 1993, available at {\tt http://publications.ias.edu/sarnak}

\bibitem{Sa2} P. Sarnak, \emph{Spectra of hyperbolic surfaces}, Bull. Amer. Math. Soc. (N.S.) \textbf{40} (2003), 441--478.

\bibitem{Sa1} P. Sarnak, \emph{Letter to Cathleen Morawetz}, August 2004, available at {\tt http://publications.ias.edu/sarnak}

\bibitem{Sh} G. Shimura, \emph{The special values of the zeta functions associated with Hilbert modular forms}, Duke Math. J. \textbf{45} (1978), 637--679; \emph{Corrections}, ibid. \textbf{48} (1981), 697.

\bibitem{Sh2} G. Shimura, \emph{Introduction to the arithmetic theory of automorphic functions}, Reprint of the 1971 original, Princeton University Press, Princeton, NJ, 1994.

\bibitem{Sp} F. Spinu, \emph{The $L^4$ norm of the Eisenstein series}, Ph.D. thesis, Princeton University, 2003, available at \verb|http://www.math.jhu.edu/~fspinu/math/thesis.pdf|

\bibitem{St} R. S. Steiner, \emph{Supnorm of modular forms of half-integral weight in the weight aspect}, Acta Arith. \textbf{177} (2017), 201--218.

\bibitem{Te} N. Templier, \emph{Hybrid sup-norm bounds for Hecke--Maass cusp forms}, J. Eur. Math. Soc. \textbf{17} (2015), 2069--2082.

\bibitem{VdK} J. VanderKam, \emph{$L^{\infty}$ norms and quantum ergodicity on the sphere}, Int. Math. Res. Not. \textbf{1997}, no. 7, 329--347; \emph{Correction}, ibid. \textbf{1998}, no. 1, 65.

\bibitem{Ve1} A. Venkatesh, \emph{{\rm Beyond Endoscopy} and special forms on $GL(2)$}, J. Reine Angew. Math. \textbf{577} (2004), 23--80.

\bibitem{Ve} A. Venkatesh, \emph{Sparse equidistribution problems, period bounds, and subconvexity}, Ann. of Math. (2) \textbf{172} (2010), 989--1094.

\bibitem{We} A. Weil, \emph{Basic number theory}, 3rd edition, Grundlehren der Mathematischen Wissenschaften \textbf{144}, Springer-Verlag, New York-Berlin, 1974.

\bibitem{X} H. Xia, \emph{On $L^{\infty}$ norms of holomorphic cusp forms}, J. Number Theory \textbf{124} (2007), 325--327.

\end{thebibliography}
\end{document}